%% file: Abelianisation.tex
\documentclass{amsart}

\usepackage[english]{babel}
\usepackage[applemac]{inputenc}
\usepackage{amsmath,amsthm, amssymb,amsfonts,nicefrac}
\usepackage[all]{xy}
\usepackage{amsmath}
\usepackage{graphicx}
\usepackage{textcomp}
\usepackage{color}
\usepackage{graphpap}
\usepackage[backref, colorlinks, linktocpage, citecolor = blue, linkcolor = blue]{hyperref}
\usepackage{enumitem}
\usepackage{wrapfig}
\usepackage{placeins}

\newlist{erel}{enumerate}{1}
\setlist[erel,1]{label={\bfseries (rel. \arabic*)}}

\newcommand{\N}{\ensuremath{\mathbb{N}}}
\newcommand{\Z}{\ensuremath{\mathbb{Z}}}

\newcommand{\K}{\ensuremath{\mathbb{K}}}
\newcommand{\R}{\ensuremath{\mathbb{R}}}
\newcommand{\C}{\ensuremath{\mathbb{C}}}
\newcommand{\PP}{\ensuremath{\mathbb{P}}}
\newcommand{\A}{\ensuremath{\mathbb{A}}}
\newcommand{\FF}{\ensuremath{\mathbb{F}}}
\newcommand{\PO}{\ensuremath{\mathbb{P}\mathrm{O}}}
\newcommand{\HH}{\ensuremath{\mathbb{H}}}
\newcommand{\J}{\ensuremath{\mathcal{J}}}
\renewcommand\k{\mathbb{K}}

\DeclareMathOperator{\p}{p}
\DeclareMathOperator{\Aut}{Aut}
\DeclareMathOperator{\Bir}{Bir}

\DeclareMathOperator{\Bp}{Bp}

\DeclareMathOperator{\Id}{Id}

\DeclareMathOperator{\PGL}{PGL}

\def\dashmapsto{\mapstochar\dashrightarrow}

\newcommand\SB[1][\scalebox{0.6}]{#1}

\newcommand\SBa[1][\scalebox{0.8}]{#1}
\newcommand\SBb[1][\scalebox{0.9}]{#1}

\newtheorem{Thm}{Theorem}[section]
\newtheorem*{Thm*}{Theorem}
\newtheorem{Cor}[Thm]{Corollary}
\newtheorem{Lem}[Thm]{Lemma}
\newtheorem{Prop}[Thm]{Proposition}
\theoremstyle{definition}
\newtheorem{Def}[Thm]{Definition}
\newtheorem{Rmk}[Thm]{Remark}

\newtheorem{Not}[Thm]{Notation}

\title{The abelianisation of the real Cremona group}
\author{Susanna Zimmermann}
\thanks{The author gratefully acknowledges support by the Swiss National Science Foundation Grant ``Birational geometry'' PP00P2\_153026 /1}

\address{Susanna Zimmermann\\
 Mathematisches Institut\\
Universit\"at Basel\\
Spiegelgasse 1\\
4051 Basel, Switzerland}
\email{Susanna.Zimmermann@unibas.ch}

\begin{document}
\maketitle
\thispagestyle{empty}

\begin{abstract}
We present the abelianisation of the group of birational transformations of $\PP^2_\R$.  
\end{abstract}

\tableofcontents

\section{Introduction}
Let $\Bir_{\R}(\PP^2)\subset \Bir_\C(\PP^2)$ be the groups of birational transformations of the projective plane defined over the respective fields of real and complex numbers, and $\Aut_\R(\PP^2)\simeq\mathrm{PGL}_3(\R)$, $\Aut_\C(\PP^2)\simeq\mathrm{PGL}_3(\C)$ the respective subgroups of linear transformations.

According to the Noether-Castelnuovo Theorem \cite{Cas}, the group $\Bir_\C(\PP^2)$ is generated by $\Aut_\C(\PP^2)$ and the standard quadratic transformation $\sigma_0\colon [x:y:z]\dashmapsto [yz:xz:xy]$. As an abstract group, it is not simple \cite{CL13,L15}, i.e. there exist non-trivial, proper normal subgroups $N\subset \Bir(\PP^2)$. The construction in \cite{L15} implies that $\Bir_\C(\PP^2)/N$ is SQ-universal (just like $\Bir_\C(\PP^2)$ itself), i.e. any countable group embeds into a quotient of $\Bir_\C(\PP^2)/N$ \cite[\S8]{DGO17}. Moreover, the normal subgroup generated by any non-trivial element which preserves a pencil of lines or which has degree $d\leq 4$ is the whole group \cite[Lemma 2]{Giz}, and $N$ has uncountable index (see Remark~\ref{rmk conjugate}). As $\PGL_3(\C)$ is perfect, it follows that $\Bir_\C(\PP^2)$ is perfect.

For $\Bir_\R(\PP^2)$ the situation is quite different. First of all, the group generated by $\Aut_{\R}(\PP^2)=\mathrm{PGL}_3(\R)$ and $\sigma_0$ is certainly not the whole group, as all its elements have only real base-points, and $\Bir_\R(\PP^2)$ contains for instance the circle inversion $\sigma_1\colon [x:y:z]\dasharrow [xz:yz:x^2+y^2]$ of Appolonius, which has non-real base-points. However, the group $\Bir_{\R}(\PP^2)$ is generated by $\mathrm{PGL}_3(\R)$, $\sigma_0$, $\sigma_1$, and all standard quintic birational maps (see Definition~\ref{def 5.1}) \cite{BM12}. Using these generators, we find an explicit presentation of the group $\Bir_\R(\PP^2)$ (see Theorem~\ref{prop technical thm}) and a natural quotient, which is our main result.

\begin{Thm}\label{thm 2}
\begin{enumerate}
\item\label{thm 2:1} The group $\Bir_\R(\PP^2)$ is not perfect: its abelianisation  is isomorphic to 
\[\Bir_\R(\PP^2)/[\Bir_\R(\PP^2),\Bir_\R(\PP^2)]\simeq \bigoplus_{(0,1]} \mathbb{Z}/2\mathbb{Z}.\]
\item\label{thm 2:2} In particular, $\Bir_{\R}(\PP^2)$ is not generated by $\Aut_{\R}(\PP^2)$  and a countable set of elements. 
\item\label{thm 2:3} The commutator subgroup $[\Bir_\R(\PP^2),\Bir_\R(\PP^2)]$ is the smallest normal subgroup containing $\Aut_\R(\PP^2)\simeq\mathrm{PGL}_3(\R)$. It is perfect and contains all elements of $\Bir_\R(\PP^2)$ of degree~$\leq 4$.
\end{enumerate}
\end{Thm}

The second statement of the theorem is similar to a result for higher dimensional Cremona groups: For $n\geq3$, the group $\Bir(\PP^n)$ is not generated by $\Aut(\PP^n)$ and a countable number of elements \cite{Pan}.

Let $X$ be a real variety. We denote by $X(\R)$ its set of real points, and by $\Aut(X(\R))\subset\Bir_\R(X)$ the subgroup of birational transformations defined at each point of $X(\R)$. It is also called the group of {\em birational diffeomorphisms} of $X(\R)$, and is, in general, strictly larger than the group of automorphisms $\Aut_\R(X)$ of $X$ defined over $\R$. The group $\Aut(\PP^2(\R))$ is generated by $\Aut_\R(\PP^2)$ and the standard quintic transformations (see Definition~\ref{def 5.1}) \cite{RV05,BM12}. 
In the following, let $\PP^3\supset\mathcal{Q}_{3,1}$ be the smooth quadric surface given by $x^2+y^2+z^2=w^2$, and $\FF_0\simeq\PP^1\times\PP^1$ the real surface whose antiholomorphic involution is the complex conjugation on each factor. 
Using real birational transformations $\PP^2\dashrightarrow X$ and the explicit construction of the quotient in Theorem~\ref{thm 2}~(\ref{thm 2:1}), we find the following corollary.

\begin{Cor}\label{cor 2}
For any real birational map $\psi\colon\FF_0\dashrightarrow\PP^2$, the group $\psi\Aut(\FF_0(\R))\psi^{-1}$ is a subgroup of $\ker\left(\Bir_\R(\PP^2)\rightarrow\bigoplus_{(0,1]}\Z/2\Z\right)$. 
There exist surjective group homomorphisms 
{\small
\[\Aut(\PP^2(\R))\rightarrow\bigoplus_{(0,1]} \Z/2\Z,\quad\Aut(\A^2(\R))\rightarrow\bigoplus_{(0,1]}\Z/2\Z,\quad\Aut(\mathcal{Q}_{3,1}(\R))\rightarrow\bigoplus_{(0,1]}\Z/2\Z.\]
} In particular, $\Aut(\PP^2(\R))$ is not generated by $\Aut_\R(\PP^2)$ and a coundable number of elements. The same holds for $\Aut(\A^2(\R))$ and $\Aut(\mathcal{Q}_{3,1}(\R))$ if we replace $\Aut_\R(\PP^2)$ by the affine automorphism group of $\A^2$ and by $\Aut_\R(\mathcal{Q}_{3,1})$, respectively.
\end{Cor}

\begin{Cor}\label{cor 1}
For any $n\in\N$ there is a normal subgroup of $\Bir_{\R}(\PP^2)$ $($resp. $\Aut(\PP^2(\R))$, resp. $\Aut(\A^2(\R))$, resp. $\Aut(\mathcal{Q}_{3,1})$$)$ of index $2^n$ containing all its elements of degree $\leq 4$. 
\end{Cor}

\begin{Cor}\label{cor 4} 
Any normal subgroup of $\Bir_\R(\PP^2)$ generated by a countable set of elements of $\Bir_\R(\PP^2)$ is a proper subgroup of $\Bir_\R(\PP^2)$. 
The same statement holds for $\Aut(\PP^2(\R))$, $\Aut(\A^2(\R))$ and $\Aut(\mathcal{Q}_{3,1}(\R))$.
\end{Cor}

The plan of the article is as follows: After giving the basic definitions and notations in Section~\ref{section groups}, we define in Section~\ref{section quotient Jcirc} a surjective group homomorphism from the subgroup $\mathcal{J}_{\circ}\subset\Bir_\R(\PP^2)$ of elements preserving a pencil of conics to the group $\bigoplus_{(0,1]}\Z/2\Z$. Section~\ref{section technical thm} is entirely devoted to the proof of an explicit presentation of $\Bir_\R(\PP^2)$ by generators and relations, given in Theorem~\ref{prop technical thm}, on which the subsequent section is based. The proof is rather long and technical, and one might skip this section for more comfortable reading and return to it at the end of the paper.
In Section~\ref{section quotient}, we extend the homomorphism to a surjective group homomorphism $\Bir_{\R}(\PP^2)\rightarrow\bigoplus_{(0,1]}\Z/2\Z$ and prove the second statement of Theorem~\ref{thm 2}, Corollary~\ref{cor 1} and Corollary~\ref{cor 2}. In Section~\ref{section kernel}, we prove that its kernel is the smallest normal subgroup containing $\Aut_{\R}(\PP^2)$, which will turn out to be the commutator subgroup of $\Bir_{\R}(\PP^2)$. This will conclude Theorem~\ref{thm 2}. \par
In \cite{Pol15} one can find a description of the elementary links between real rational surfaces and relations between them. However, this description was not used in the proof of Theorem~\ref{prop technical thm}.

\vskip\baselineskip

{\sc Acknowledgements:} I would like to thank Jérémy Blanc for the priceless discussions about quotients and relations, Jean-Philippe Furter, Fr\'ed\'eric Mangolte and Christian Urech for their useful remarks, and J\'anos Koll\'ar for pointing out the connection to the spinor norm. I would like to thank the referees for their careful reading and their very helpful suggestions to re-organise and shorten the article.


\section{Basic notions}\label{section groups}

\subsection{Linear systems and base-points}
We now give some basic notations and definitions, whose standard vocabulary can for instance be found in \cite{AC02} and \cite[\S7]{D12}. 

Throughout the article every surface and rational map is defined over $\R$, unless stated otherwise.

We recall that a real birational transformation $f$ of $\PP^2$ is given by
\[f\colon[x_0:x_1:x_2]\dashmapsto[f_0(x_0,x_1,x_2):\dots:f_2(x_0,x_1,x_2)]\]
where $f_0,f_1,f_2\in\R[x_0,x_1,x_2]$ are homogeneous of equal degree and without common factor, and $f$ has an inverse of the same form. We say that the pre-image by $f$ of the linear system of lines in $\PP^2$ is {\em the linear system of $f$} (also called {\em homaloidal net of $f$}). It is in the linear system of curves in $\PP^2$ generated by the curves $\{f_0=0\},\{f_1=0\},\{f_2=0\}$ and has no fixed components. The linear system of $f$ is a tool to study $f$ from a geometric point of view, as many properties of $f$ can be found by looking at its linear system. The base-locus of the linear system of $f$ is a finite set of points, and by abuse of terminology, we refer to it as set of {\em base-points of $f$}. Some base-points of $f$ might not be points in $\PP^2$ but points on a blow-up of $\PP^2$. 

Let $X_n\stackrel{\pi_n}\rightarrow X_{n-1}\stackrel{\pi_{n-1}}\rightarrow X_{n-2}\cdots X_1\stackrel{\pi_1}\rightarrow\PP^n$ be a sequence of blow-ups of points $q_0\in\PP^2,q_1\in X_1,\dots,q_{n-1}\in X_{n-1}$. Let $E_i:=\pi_{i+1}^{-1}(q_i)\subset X_{i+1}$ be the exceptional divisor of $q_i$. A point in $X_{i+1}$  in the {\em first neighbourhood} of $q_i$ if it is contained in $E_i$. A point in $X_n$ is  
\begin{itemize}
\item {\em infinitely near} $q_i$ if it is contained in the total transform of $E_i$ in $X_n$,
\item {\em proximate} to $q_i$ if it is contained in the strict transform of $E_i$ in $X_n$,
\item  a {\em proper point of} $\PP^2$ if it is not contained in the total transform of any of the exceptional divisors. We may identify it with its image in $\PP^2$.
\end{itemize}
By abuse of notation, we call the multiplicity of the linear system of $f$ in a point $p$ the {\em multiplicity of $f$ in $p$}. A simple base-point of $f$ is a base-point of multiplicity $1$.

\begin{Def}
For $f\in\Bir_\R(\PP^2)$, the characteristic of $f$ is the sequence $(\deg(f);\ m_1^{e_1},\dots,m_k^{e_k})$ where $m_1,\dots,m_k$ are the multiplicities of the base-points of $f$ and $e_i$ is the number of base-points of $f$ which have multiplicity $m_i$.
\end{Def}

\begin{Def}
Let $C\subset\PP^2$ be an irreducible (closed) curve, $f\in\Bir_{\R}(\PP^2)$ and $\Bp(f)$ the set of base-points of $f$. We denote by
\[f(C):=\overline{f(C\setminus\Bp(f))}\]
the (Zariski-) closure of the image by $f$ of $C$ minus the base-points of $f$, and call it {\em the image of $C$ by $f$}. 
\end{Def}

\subsection{Transformation preserving a pencil of lines or conics}

Throughout the article, we fix the notation
\[p_1:=[1:i:0],\quad p_2:=[0:1:i]\]
for these two specific points of $\PP^2$, because we will use them extremely often. 

\begin{Def}\label{def basic map}
We define two rational fibrations
\begin{align*}& \pi_*\colon\PP^2\dashrightarrow\PP^1,\quad[x:y:z]\mapsto[y:z]\\[6pt]
&\pi_{\circ}\colon\PP^2\dashrightarrow\PP^1,\quad[x:y:z]\mapsto[y^2+(x+z)^2:y^2+(x-z)^2]
\end{align*}
whose fibres are respectively the lines through $[1:0:0]$ and the conics through $p_1,p_2$, and their conjugates \mbox{$\bar{p}_1=[1:-i:0],\bar{p}_2=[0:1:-i]$.} 

We define by $\mathcal{J}_*$, $\mathcal{J}_{\circ}$ the subgroups of $\Bir_\R(\PP^2)$ preserving the fibrations $\pi_*,\pi_{\circ}$:
\begin{align*}
&\mathcal{J}_*= \{f\in\Bir_\R(\PP^2)\mid \exists\hat{f}\in\Aut_\R(\PP^1)\colon \hat{f}\pi_*=\pi_*f\}\\
&\mathcal{J}_{\circ}=\{f\in\Bir_\R(\PP^2)\mid\exists\hat{f}\in\Aut_\R(\PP^1)\colon \hat{f}\pi_{\circ}=\pi_{\circ}f\}
\end{align*}
\end{Def}
$\mathcal{J}_*$ is the group of transformations preserving the pencil of lines through $[1:0:0]$. 
In affine coordinates $(x,y)=[x:y:1]$ its elements are of the form
{\small \[(x,y)\dashmapsto\left(\frac{\alpha(y)x+\beta(y)}{\gamma(y)x+\delta(y)},\frac{ay+b}{cy+d}\right)\] }
where $a,b,c,d\in\R$ and $\alpha,\beta,\gamma,\delta\in\R(y)$, and so $\mathcal{J}_*\simeq\mathrm{PGL}_2(\R(y))\rtimes\mathrm{PGL}_2(\R)$.

The group $\mathcal{J}_{\circ}$ is the group of transformations preserving the pencil of conics through $p_1,\bar{p}_1,p_2,\bar{p}_2$. A description as semi-direct product is given in Lemma~\ref{lem scale on P1}.

Extending the scalars to $\C$, the analogues of these groups are conjugate in $\Bir_\C(\PP^2)$ and are called {\em de Jonquières groups}. In $\Bir_\R(\PP^2)$, the groups $\mathcal{J}_{\circ},\mathcal{J}_*$ are not conjugate. This can, for instance, be seen as consequence of Proposition~\ref{thm quotient} (see Remark~\ref{rmk conjugate}).


\section{A quotient of $\mathcal{J}_{\circ}$}\label{section quotient Jcirc}

We first construct a surjective group homomorphism $\varphi_{\circ}\colon\mathcal{J}_{\circ}\rightarrow\bigoplus_{(0,1]}\Z/2\Z$ and then (in Section~\ref{section quotient}) use the representation of $\Bir_\R(\PP^2)$ by generators and relations (Theorem~\ref{prop technical thm}) to extend $\varphi_{\circ}$ to a homomorphism $\varphi\colon\Bir_\R(\PP^2)\rightarrow\bigoplus_{(0,1]}\Z/2\Z$. Both quotients are generated by classes of standard quintic transformations contained in $\mathcal{J}_{\circ}$, as we will see from the construction in Subsection~\ref{subsection quotient Jcirc}.
In order to construct the surjective homomorphism $\varphi_{\circ}$, we need some additional information about the elements of $\mathcal{J}_{\circ}$, such as their characteristic (Lemma~\ref{lem char}) and their action on the pencil of conics passing through $p_1,\bar{p}_1,p_2,\bar{p}_2$ (Lemma~\ref{lem scale on P1}).

\subsection{The group $\mathcal{J}_{\circ}$}\label{subsection Jcirc}

We will use the properties stated in the following lemmata to obtain the action of $\mathcal{J}_{\circ}$ on the pencil of conics through $p_1,\dots,\bar{p}_2$, which will be used to construct the quotients. In Section~\ref{section technical thm} (proof of Theorem~\ref{prop technical thm}), we will use the properties to study linear systems and their base-points when playing with the relations given in Definition~\ref{def G}.

\begin{Def}\label{def basic} 
Let $\eta\colon X\rightarrow\PP^2$ be the blow-up of $p_1,\bar{p}_1,p_2,\bar{p}_2$. The morphism $\tilde{\pi}_{\circ}:=\pi_{\circ}\eta\colon X\rightarrow\PP^1$ is a real conic bundle with fibres being the strict transforms of the conics passing through $p_1,\dots,\bar{p}_2$. 
\[\xymatrix{X\ar[r]^{\eta}\ar@/_.75pc/[rr]_{\tilde{\pi}_{\circ}}&\PP^2\ar@{-->}[r]^{\pi_{\circ}}&\PP^1}\]
Let $\eta'\colon Y\rightarrow X$ be a birational morphism and let $q\in Y$ be contained in the strict transform of a fibre $f$ of $\tilde{\pi}_{\circ}$. The curve 
\[C_q:=\eta\eta'(f)\subset\PP^2\]
is a conic passing through $p_1,\bar{p}_1,p_2,\bar{p}_2$. 
We say that $C_q$ is the conic passing through $p_1,\bar{p}_1,p_2,\bar{p}_2,q$. 
The curve $C_q$ is irreducible or the union of two lines. The latter corresponds to $\pi_{\circ}(C_q)\in\{[1:0],[0:1],[1:1]\}$. We define
\[C_1:=\pi_{\circ}^{-1}([0:1]),\quad C_2:=\pi_{\circ}^{-1}([1:0]),\quad C_3:=\pi_{\circ}^{-1}([1:1]).\]
If we denote by $L_{r,s}\subset\PP^2$ the line passing through $r,s\in\PP^2$, then
\[C_1=L_{p_1,p_2}\cup L_{\bar{p}_1,\bar{p}_2},\quad C_2=L_{p_1,\bar{p}_2}\cup L_{\bar{p}_1,p_2},\quad C_3=L_{p_1,\bar{p}_1}\cup L_{p_2,\bar{p}_2}.\]
\end{Def}

\begin{Lem}\label{lem char}
Any element of $\mathcal{J}_{\circ}$ of degree $d>1$ has characteristic
\begin{align*}
&{\tiny\left(d;\ \left(\frac{d-1}{2}\right)^4,\ 2^{\frac{d-1}{2}}\right)},\quad\text{if $\deg(f)$ is odd}\\
&{\tiny\left(d;\ \left(\frac{d}{2}\right)^2,\ \left(\frac{d-2}{2}\right)^2,\ 2^{\frac{d-2}{2}},\ 1\right)},\quad\text{if $\deg(f)$ is even}
\end{align*}
and $p_1,\dots,\bar{p}_2$ are base-points of multiplicity $\frac{d-1}{2}$ or $\frac{d}{2}$ and $\frac{d-2}{2}$. \par
Furthermore, 
\begin{enumerate}
\item no two double points are contained in the same conic through $p_1,\bar{p}_1,p_2,\bar{p}_2$, 
\item any element of $\mathcal{J}_{\circ}$ exchanges or preserves the real reducible conics $C_1$ and $C_2$, and does not contract their components.
\item any element of $\mathcal{J}_{\circ}$ of even degree contracts $L_{p_1,\bar{p}_1}$ or $L_{p_2,\bar{p}_2}$ onto a point on a real conic different from $C_1,C_2$.
\end{enumerate}
\end{Lem}

\begin{proof}
Let $f\in\mathcal{J}_{\circ}$ be of degree $d>1$ and let $C$ be a general conic passing through $p_1,\bar{p}_1,p_2,\bar{p}_2$. By definition of $\mathcal{J}_{\circ}$, the curve $f(C)$ is a conic through $p_1,\bar{p}_1,p_2,\bar{p}_2$. Let $m(q)$ be the multiplicity of $f$ at the point $q$. Computing the intersection of $C$ on the blow-up of the base-points of $f$ with the linear system of $f$ gives the degree of $f(C)$:
\[2=\deg(f(C))=2d-m(p_1)-m(\bar{p}_1)-m(p_2)-m(\bar{p}_2)=(d-2m(p_1))+(d-2m(p_2)).\]
For $i=1,2$, computing the intersection of the linear system of $f$ with the line $L_{p_i,\bar{p}_i}$, we obtain that $d\geq2m(p_i)$.

If $d-2m(p_1)=d-2m(p_2)=1$, then 
\[m(p_1)=m(p_2)=\frac{d-1}{2}.\] 

Else, $d-2m(p_i)=0$, $d-2m(p_{3-i})=2$ for some $i\in\{1,2\}$, and so 
\[m(p_i)=\frac{d}{2},\ m(p_{3-i})=\frac{d-2}{2},\quad i\in\{1,2\}.\] 
Let $q$ be a base-point of $f$ not equal to $p_1,\bar{p}_1,p_2,\bar{p}_2$. 
Suppose that there exits a conic $C_q$ passing through $p_1,\bar{p}_1,p_2,\bar{p}_2,q$ (see Definition~\ref{def basic}). Then $0\leq\deg(f(C_q))\leq2$ and
\[0\leq \deg(f(C_q))\leq2d-2m(p_1)-2m(p_2)-m(q)=2-m(q)\leq2\]
In particular, $m(q)\in\{1,2\}$. If $C_q$ does not exist, then $q$ is infinitely near to a point $q'$ for which there exists a conic $C_{q'}$ passing through $p_1,\dots,\bar{p}_2,q'$, and $m(q')\leq m(q)\leq2$. Let $D$ be a general member of the linear system of $f$. The genus formula
{\small
\[0=g(D)=\frac{(d-1)(d-2)}{2}-\sum_{q\ \text{base-point of}\ f}\frac{m(q)(m(q)-1)}{2}\]}
 and $m(q)\in\{1,2\}$ for all base-points $q$ of $f$ different from $p_1,\bar{p}_1,p_2,\bar{p}_2$ imply that
 {\small
 \[\frac{(d-1)(d-2)}{2}=2\sum_{i=1}^2\frac{m(p_i)(m(p_i)-1)}{2}+|\{\text{base-points of multiplicity 2}\}|\]}
 and in particular that
 \[|\{\text{base-points of multiplicity 2}\}|=\begin{cases}\frac{d-1}{2},&d\ \text{odd}\\ \frac{d-2}{2},&d\ \text{even}\end{cases}\]
Intersecting two general elements of the linear system of $f$, we get the classical equality 
 {\small\[d^2-1=\sum_{q\ \text{base-point of}\ f}m(q)^2.\] }
It yields that $f$ has exactly one simple base-point if $d$ is even and none otherwise. This yields the characteristics. \par
Calculating the intersection of a conic through $p_1,\bar{p}_1,p_2,\bar{p}_2$ with a general curve of the linear system of $f$ implies that no two double points are contained in the same conic through $p_1,\bar{p}_1,p_2,\bar{p}_2$. The conics $C_1$, $C_2$, $C_3$ are the only reducible conics through $p_1,\dots,\bar{p}_2$, and $C_1,C_2$ each consist of two non-real lines while $C_3$ consists of two real lines. 
If $f$ has even degree, it has base-points $p_i,\bar{p}_i$ of multiplicity $m(p_i)=\frac{d}{2}$ and therefore contracts the line $L_{p_i,\bar{p}_i}$ onto the base-point of $f^{-1}$ of multiplicity $1$, and no other line is contracted (because $f^{-1}$ has only one base-point of multiplicity $1$). Because of this and the multiplicities of the base-points of $f$, $f$ sends $L_{p_i,p_j},L_{p_i,\bar{p}_j},i\neq j$, onto non-real lines. This is also true if $f$ has odd degree, simply because of the multiplicities of its base-points. Thus $f$ preserves or exchanges $C_1,C_2$. 

In particular, the induced automorphism $\hat{f}$ on $\PP^1$ does not send $\pi_{\circ}(C_3)$ onto either of $\pi_{\circ}(C_1),\pi_{\circ}(C_2)$. It follows that if $f$ has even degree, the point $f(L_{p_i,\bar{p}_i})$ is contained in the conic $\pi_{\circ}^{-1}(\hat{f}(\pi_{\circ}(C_3)))\neq C_1,C_2$. In particular, the simple base-point $f(L_{p_i,\bar{p}_i})$ of $f^{-1}$ is not contained in $C_1,C_2$. By symmetry, the same holds for $f$.
\end{proof}

\begin{Rmk}\label{rmk Jcirc elt}
Remark that $\sigma_1\colon[x:y:z]\dashrightarrow[y^2+z^2:xy:xz]$ is contained in $\mathcal{J}_{\circ}$. 
The linear map $[x:y:z]\mapsto[z:-y:x]$ exchanges $p_1$ and $p_2$ (and $\bar{p}_1$ and $\bar{p}_2$), and the linear map $[x:y:z]\mapsto[-x:y:z]$ exchanges $p_1$ and $\bar{p}_1$ and fixes $p_2$. Both are contained in $\Aut_\R(\PP^2)\cap\mathcal{J}_{\circ}$.   
\end{Rmk}

We will need the following two lemmata concerning the existence of quadratic and cubic transformations in $\mathcal{J}_{\circ}$. 
 
\begin{Lem}\label{lem sigma}
For any $q\in\PP^2(\R)$ not collinear with any two of $\{p_1,\bar{p}_1,p_2,\bar{p}_2\}$ except the pair $(p_2,\bar{p}_2)$, there exists $f\in\mathcal{J}_{\circ}$ of degree $2$ with base-points $p_1,\bar{p}_1,q$.

Let $f\in\mathcal{J}_{\circ}$ of even degree $d$, the points $p_i,\bar{p}_i$ its base-points of multiplicity $\frac{d}{2}$ and $r$ its simple base-point or the proper point of $\PP^2$ to which the simple base-point is infinitely near.
Then there exists $\tau\in\mathcal{J}_{\circ}$ of degree $2$ with base-points $p_i,\bar{p}_i,r$.
\end{Lem}

\begin{proof}
Since $q$ is not collinear with $p_1,\bar{p}_1$, there exists $\alpha\in\Aut_\R(\PP^2)$ that sends $p_1,\bar{p}_1,q$ onto $p_1,\bar{p}_1,[0:0:1]$. 
Then the quadratic transformation $\sigma_1\alpha$ has base-points $p_1,\bar{p}_1,q$.  
By assumption, the points $p_2,\bar{p}_2$ are not on any line contracted by $\sigma_1\alpha$ and so $\sigma_1\alpha$ is an isomorphism around them. Define $t:=(\sigma_1\alpha)(p_2)$. 
There exists $\beta\in\Aut_\R(\PP^2)$ that fixes $p_1,\bar{p}_1$ and sends $t,\bar{t}$ onto $p_2,\bar{p}_2$. By construction $\sigma_1\alpha$ sends the pencil of conics through $p_1,\bar{p}_1,p_2,\bar{p}_2$ onto the pencil of conics through $p_1,\bar{p}_1,t,\bar{t}$, which is sent by $\beta$ onto the pencil of conics through $p_1,\bar{p}_1,p_2,\bar{p}_2$.
\begin{figure}
\def\svgwidth{0.9\textwidth}
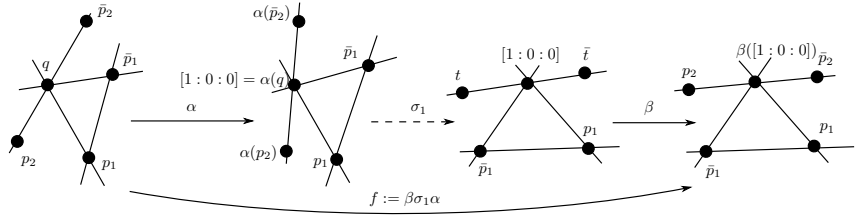
\caption{The construction of the quadratic map $f$ with base-points $p_1,\bar{p}_1,q$.}\label{figure 1}
\end{figure}

Let $f\in\mathcal{J}_{\circ}$ of even degree $d$, $p_i,\bar{p}_i$ its base-points of multiplicity $\frac{d}{2}$ and $r$ its simple base-point or the proper point of $\PP^2$ to which the simple base-point is infintely near. By Bézout, $r,p_i,\bar{p}_i$ are not collinear and by Lemma~\ref{lem char} the points $r,p_i,p_{3-i}$ and $r,\bar{p}_i,p_{3-i}$ are not collinear. Hence there exists $\tau\in\mathcal{J}_{\circ}$ of degree 2 with base-points $r,p_i,\bar{p}_i$. 
\end{proof}

A similar statement holds for transformations of degree $3$, which we will use in Section~\ref{section technical thm} and Section~\ref{section kernel}.
\begin{Lem}\label{lem sigma deg 3}
For every $r\in\PP^2(\R)$ not collinear with any two of $p_1,\bar{p}_1,p_2,\bar{p}_2$ there exists $f\in\mathcal{J}_{\circ}$ of degree $3$ with base-points $r,p_1,\bar{p}_1,p_2,\bar{p}_2$ $($with double point $r)$.
\end{Lem}

\begin{proof}
By assumption $r$ is not collinear with any two of $p_1,\bar{p}_1,p_2,\bar{p}_2$, so there exists $\tau_1\in\mathcal{J}_{\circ}$ quadratic with base-points $r,p_1,\bar{p}_1$ (Lemma~\ref{lem sigma}). The base-points of its inverse are $s,p_i,\bar{p}_i$ for some $s\in\PP^2(\R)$ and $i\in\{1,2\}$. We can assume that $i=1$ by exchanging $p_1,p_2$ if necessary (Remark~\ref{rmk Jcirc elt}). In particular, $\tau_1$ preserves the set $\{p_2,\bar{p}_2\}$ and is an isomorphism around these two points. 
By assumption, $r,p_2,\bar{p}_2$ are not collinear and $\tau_1$ sends the lines through $r$ onto the lines through $s$, so $s,p_2,\bar{p}_2$ are not collinear as well.
Hence there exists $\tau_2\in\mathcal{J}_{\circ}$ of degree $2$ with base-point $s,p_2,\bar{p}_2$ (Lemma~\ref{lem sigma}). The map $\tau_2\tau_1\in\mathcal{J}_{\circ}$ is of degree $3$ with base-points $r,p_1,\bar{p}_1,p_2,\bar{p}_2$.
\end{proof}

\begin{Rmk}\label{rmk deg 3}
Let $f\in\mathcal{J}_{\circ}$ of degree $3$ and $r\in\PP^2(\R)$ its double point. The points $p_1,\bar{p}_1,p_2,\bar{p}_2$ are simple base-points of $f$. Note that for $i\in\{1,2\}$, the map $f$ contracts the line passing through $r,p_i$ onto one of $p_1,\bar{p}_1,p_2,\bar{p}_2$ and that $r$ is not collinear with any two of $p_1,\bar{p}_1,p_2,\bar{p}_2$.
By Lemma~\ref{lem sigma} there is a quadratic transformation $g\in\mathcal{J}_{\circ}$ with base-points $r,p_1,\bar{p}_1$. The map $fg^{-1}\in\J_{\circ}$ is of degree $2$ and thus every element of $\mathcal{J}_{\circ}$ of degree $3$ is the composition of two quadratic elements of $\mathcal{J}_{\circ}$.
\end{Rmk}

We define a type of real birational transformation called {\em standard quintic tranformation} and {\em special quintic transformations}, which have been described often, for instance in \cite[\S3]{BM12}, \cite[Lemma 6.3.10]{CS16}, \cite[\S VI.23]{Hud27}, \cite[\S1]{RV05}.

\begin{Def}[Standard quintic transformations]\label{def 5.1} 
Let $q_1,\bar{q}_1,q_2,\bar{q}_2,q_3,\bar{q}_3\in\PP^2$ be three pairs of non-real conjugate points of $\PP^2$, not lying on the same conic. Denote by $\pi\colon X\rightarrow\PP^2$ the blow-up of these points. The strict transforms of the six conics passing through exactly five of the six points are three pairs of non-real conjugate $(-1)$-curves. Their contraction yields a birational morphism $\eta\colon X\rightarrow Y\simeq\PP^2$ which contracts the curves onto three pairs of non-real points $r_1,\bar{r}_1,r_2,\bar{r}_2,r_3,\bar{r}_3\in\PP^2$. We choose the order so that $r_i$ is the image of the conic not passing through $q_i$. The birational map $\varphi:=\eta\pi^{-1}$ is contained in $\Bir_\R(\PP^2)$, is of degree $5$ and is called {\em standard quintic transformation}. 

Note that a different choice of identification $Y$ with $\PP^2$ yields another standard quintic transformation, which we obtain from $\varphi$ by left composition with an automorphism of $\PP^2$. 
\end{Def}

\begin{Def}[Special quintic transformations]\label{def 5.2}
Let $q_1,\bar{q}_1,q_2,\bar{q}_2\in\PP^2$ be two pairs of non-real points of $\PP^2$, not on the same line. Denote by $\pi_1:X_1\rightarrow\PP^2$ the blow-up of the four points, and by $E_1,\bar{E}_1\subset X_1$ the curves contracted onto $q_1,\bar{q}_1$ respectively. Let $q_3\in E_1$ be a point, and $\bar{q}_3\in\bar{E}_1$ its conjugate. We assume that there is no conic of $\PP^2$ passing through $q_1,\bar{q}_1,q_2,\bar{q}_2,q_3,\bar{q}_3$ and let $\pi_2\colon X_2\rightarrow X_1$ be the blow-up of $q_3,\bar{q}_3$. 

On $X_2$ the strict transforms of the two conics $C,\bar{C}$ of $\PP^2$ passing through $q_1,\bar{q}_1,q_2,\bar{q}_2,q_3$ and $q_1,\bar{q}_1,q_2,\bar{q}_2,\bar{q}_3$ respectively, are non-real conjugate disjoint $(-1)$ curves. The contraction of these two curves gives a birational morphism $\eta_2\colon X_2\rightarrow Y_1$, contracting $C,\bar{C}$ onto two points $r_3,\bar{r}_3$. On $Y_1$ we find two pairs of non-real $(-1)$ curves, all four curves being disjoint: these are the strict transforms of the exceptional curves associated to $q_1,\bar{q}_1$, and of the conics passing through $q_1,\bar{q}_1,q_2,q_3,\bar{q}_3$ and $q_1,\bar{q}_1,\bar{q}_2,q_3,\bar{q}_3$ respectively. The contraction of these curves gives a birational morphism $\eta_1:Y_1\rightarrow Y_0\simeq\PP^2$ and the images of the four curves are points $r_1,\bar{r}_1,r_2,\bar{r}_2$ respectively. The real birational map $\psi=\eta_1\eta_2(\pi_1\pi_2)^{-1}\colon\PP^2\dashrightarrow\PP^2$ is of degree $5$ and called {\em special quintic transformation}.

Note that a different choice of identification $Y_0$ with $\PP^2$ yields another special quintic transformation, which we obtain from $\psi$ by left composition with an automorphism of $\PP^2$. 
\end{Def}

Standard and special quintic transformations have the following properties, which can be checked straight forwardly (see also \cite[Example 3.1]{BM12})

\begin{Lem}\label{lem 5.1}
Let $\theta\in\Bir_\R(\PP^2)$ be a standard or special quintic transformation. Then:
\begin{enumerate}
\item The points $q_1,\bar{q}_1,q_2,\bar{q}_2,q_3,\bar{q}_3$ are the base-points of $\theta$ and $r_1,\bar{r}_1,r_2,\bar{r}_2,r_3,\bar{r}_3$ are the base-points of $\theta^{-1}$, and they are all of multiplicity $2$.
\item For $i,j=1,2,3$, $i\neq j$, the map $\theta$ sends the pencil of conics through $q_i,\bar{q}_i,q_j,\bar{q}_j$ onto the pencil of conics through $r_i,\bar{r}_i,r_j,\bar{r}_j$.
\item We have $\theta\in\Aut(\PP^2(\R)):=\{f\in\Bir_\R(\PP^2)\mid f,f^{-1}\ \text{defined on}\ \PP^2(\R)\}$.
\end{enumerate}
\end{Lem}

\begin{Rmk}\label{rmk quintlin 2}
If for $i\neq j$ we have $q_i=r_i=p_1$ and $q_j=r_j=p_2$, then the standard or special quintic transformation preserves the pencil of conics through $p_1,\bar{p}_1,p_2,\bar{p}_2$ and is hence contained in $\J_{\circ}$.
Their characteristic is $(5;2^4,2^2)$ if written as in Lemma~\ref{lem char}. 
\end{Rmk}

\begin{Lem}\label{lem quintlin}
For any standard or special quintic transformation $\theta$ there exists $\alpha,\beta\in\Aut_\R(\PP^2)$ such that $\beta\theta\alpha\in\mathcal{J}_{\circ}$. 
\end{Lem}

\begin{proof}
For any two non-collinear pairs of non-real conjugate points there exists $\alpha\in\Aut_\R(\PP^2)$ that sends the two pairs onto $p_1:=[1:i:0],p_2:=[0:1:i]$ and their conjugates \mbox{$\bar{p}_1=[1:-i:0],\bar{p}_2=[0:1:-i]$.} Let $\theta$ be a standard or special quintic transformation. Then there exists $\alpha,\beta\in\Aut_\R(\PP^2)$ that send $q_1,q_2$ (resp. $r_1,r_2$) onto $p_1,p_2$. The transformation $\beta\theta\alpha^{-1}$ preserves the pencil of conics through $p_1,\bar{p}_1,p_2,\bar{p}_2$ (Lemma~\ref{lem 5.1}) and is thus contained in $\mathcal{J}_{\circ}$.
\end{proof}

\begin{Rmk}
By composing quadratic and standard quintic transformations in $\mathcal{J}_{\circ}$, we obtain transformations in $\mathcal{J}_{\circ}$ of every degree.
\end{Rmk}

\begin{Lem}\label{lem Jcirc gen}
The group $\mathcal{J}_{\circ}$ is generated by its linear, quadratic and standard quintic elements.
\end{Lem}
\begin{proof}
Let $f\in\mathcal{J}_{\circ}$. We use induction on the degree $d$ of $f$. We can assume that $d>2$. 

$\bullet$ If $d$ is even, it has a (real) simple base-point. Denote by $r$ the simple base-point of $f$ or, if the simple base-points is not a proper point of $\PP^2$, the proper point of $\PP^2$ to which the simple base-point is infinitely near. Let $p_i,\bar{p}_i$, $i\in\{1,2\}$ be the points of multiplicity $\frac{d}{2}$ (Lemma~\ref{lem char}). By Lemma~\ref{lem sigma} there exists a quadratic transformation $\tau\in\mathcal{J}_{\circ}$ with base-points $p_i,\bar{p}_i,r$. The map $f\tau^{-1}\in\mathcal{J}_{\circ}$ is of degree $\leq d-1$. 

$\bullet$ Suppose that $d$ is odd and has a real base-point $q$. By Lemma~\ref{lem char}, the points $q,p_1,p_2$ are of multiplicity $2,\frac{d-1}{2},\frac{d-1}{2}$ respectively. We can assume that $q$ is a proper point of $\PP^2$ (since no real point is infinitely near $p_1,\dots,\bar{p}_2$). By Bézout, $q$ is not collinear with $p_i,p_j$, $i,j\in\{1,2\}$, and so there exists $\tau\in\mathcal{J}_{\circ}$ of degree 2 with base-points $q,p_1,\bar{p}_1$ (Lemma~\ref{lem sigma}). The map $f\tau^{-1}\in\mathcal{J}_{\circ}$ is of degree $d-1$.

$\bullet$ Suppose that $d$ is odd and has no real base-points. Then $d\geq5$ by Lemma~\ref{lem char}. If it has a double point $q$ different from $p_1,\dots,\bar{p}_2$ which is a proper point of $\PP^2$, then $p_1,\bar{p}_1,p_2,\bar{p}_2,q,\bar{q}$ are not on the same conic (Lemma~\ref{lem char}). In particular, there exists a standard quintic transformation $\theta\in\mathcal{J}_{\circ}$ with those points its base-points (Definition~\ref{def 5.1}, Lemma~\ref{lem quintlin}). The map $f\theta^{-1}\in\mathcal{J}_{\circ}$ is of degree $d-4$.

If it has no double points different from $p_1,\dots,\bar{p}_2$ that are proper points of $\PP^2$, there exists a double point $q$ infinitely near one of the $p_i$. By Lemma~\ref{lem char}, $p_1,\bar{p}_1,p_2,\bar{p}_2,q,\bar{q}$ are not contained on one conic, hence there exists a special quintic transformation $\theta\in\mathcal{J}_{\circ}$ with base-points $p_1,\bar{p}_1,p_2,\bar{p}_2,q,\bar{q}$ (Definition~\ref{def 5.2}). The map $f\theta^{-1}\in\mathcal{J}_{\circ}$ is of degree $d-4$. By \cite[Lemma 3.7]{BM12} and Remark~\ref{rmk Jcirc elt}, $\theta$ is the composition of standard quintic and linear transformations contained in $\mathcal{J}_{\circ}$. 

If all base-points of $f$ are double points, it is of degree $5$ by Lemma~\ref{lem char} and thus a composition of linear and standard quintic transformations in $\mathcal{J}_{\circ}$ by \cite[Lemma 3.7]{BM12} and Remark~\ref{rmk Jcirc elt}.
\end{proof}

\begin{Rmk}\label{rmk:del Pezzo}
Blowing up the four base-points $p_1,\bar{p}_1,p_2,\bar{p}_2$ of the rational map $\pi_{\circ}\colon\PP^2\dashrightarrow\PP^1$ and contracting the strict transform of $L_{p_1,\bar{p}_1}$ (or $L_{p_2,\bar{p}_2}$) yields a del Pezzo surface $Z$ of degree $6$. The fibration $\pi_{\circ}$ becomes a morphism $\pi_{\circ}'\colon Z\rightarrow\PP^1$, which is a conic bundle with two singular fibres, both having only one real point. The group $\mathcal{J}_{\circ}$ is the group of birational maps of $Z$ preserving this conic bundle structure. The contraction of the two $(-1)$-sections on $Z$ is a morphism 
\[Z\rightarrow S=\{wz=x^2+y^2\}\subset\PP^3,\] 
onto the quadric in $\PP^3$ whose real part is diffeomorphic to the sphere. We can choose the images of the sections to be the points $[0:1:i:0],[0:1:-i:0]\in\PP^3$ and obtain that $Z=\{([w:x:y:z],[u:v])\in\PP^3\times\PP^1\mid uz=vw,wz=x^2+y^2\}$ and
\[\xymatrix{Z\ar[r]\ar[d]^{\pi_{\circ}'}& S\ar@{-->}[d]\\ \PP^1\ar[r]&\PP^1}\quad\footnotesize\xymatrix{([w:x:y:z],[u:v])\ar[r]\ar[d]&[w:x:y:z]\ar@{-->}[d]\\ [u:v]\ar@{=}[r]&[w:z]}\]
The generic fibre of $\pi_{\circ}'$ is the conic $C$ in $\PP^2_{\R(t)}$ given by $x^2+y^2-tz^2=0$ and $\pi_{\circ}'(Z(\R))=\pi_{\circ}(\PP^2(\R))=[0,\infty]$. Let $\Aut_\R(\PP^1,[0,\infty])\subset\Aut_\R(\PP^1)$ be the group of automorphisms preserving the interval $[0,\infty]$. It is isomorphic to $\R_{>0}\rtimes\Z/2\Z$, where $\Z/2\Z$ is generated by $[x:y]\mapsto[y:x]$. The projection $\pi_{\circ}'$ induces an exact sequence
\[1\rightarrow\Aut_{\R(t)}(C)\longrightarrow\mathcal{J}_{\circ}\longrightarrow\Aut_\R(\PP^1,[0,\infty])\simeq\R_{>0}\rtimes\Z/2\Z\rightarrow1.\] 
\end{Rmk}

\begin{Lem}\label{lem scale on P1}\item
\begin{enumerate}
\item The action of $\mathcal{J}_{\circ}$ on $\PP^1$ gives rise to a split exact sequence
\[1\rightarrow\mathrm{SO}(x^2+y^2-tz^2,\R(t))\longrightarrow\mathcal{J}_{\circ}\longrightarrow\Aut_\R(\PP^1,[0,\infty])\simeq\R_{>0}\rtimes\Z/2\Z\rightarrow1\]
\item Any element of $\R_{>0}$ is the image of a quadratic element of $\mathcal{J}_{\circ}$ and $\Z/2\Z$ is the image of a linear element. 
\item The cubic transformations are sent onto $(1,0)$ if they contract $L_{p_i,q}$ onto $p_i$ or $\bar{p}_i$, $i=1,2$, where $q$ is the double point, and onto $(1,1)$ otherwise. 
\item The standard quintic transformations are sent onto $(1,0)$ or $(1,1)$. 
\end{enumerate}
\end{Lem}

\begin{proof} 
The sequence comes from Remark~\ref{rmk:del Pezzo}. The group $\Aut_{\R(t)}(C)$ is isomorphic to the subgroup of $\mathrm{PGL}_3(\R(t))$ preserving the quadratic form $x^2+y^2-tz^2$, and is therefore isomorphic to \mbox{$\mathrm{SO}(x^2+y^2-tz^2,\R(t))$}. The sequence is exact and split by  
\begin{align*}(\lambda,0)&\mapsto\left([w:x:y:z]\mapsto[\lambda w:\sqrt{\lambda}x:\sqrt{\lambda}y:z]\right),\quad
(0,1)\mapsto([w:x:y:z]\mapsto[z:x:y:w]).
\end{align*}
Let $f\in\mathcal{J}_{\circ}$. It lifts to a birational transformation of $Z$ (see Remark~\ref{rmk:del Pezzo}) preserving its conic bundle structure. By Lemma~\ref{lem char}, it preserves or exchanges $C_1,C_2$. Its induced automorphism on $\PP^1$ is therefore determined by the image of any fibre different from $C_1,C_2$. Hence $f$ induces $[u:v]\mapsto[au:bv]$ or $[u:v]\mapsto[av:bu]$ where $[a:b]=\pi_{\circ}(f(C_3))$ on $\PP^1$. 

The linear map $[x:y:z]\mapsto[-x:y:z]$ induces $[u:v]\mapsto[v:u]$, i.e. is sent onto the generator of $\Z/2\Z$. Let $\tau\in\mathcal{J}_{\circ}$ be a quadratic map and  $q=[a:b:1]$ the real base-points of $\tau^{-1}$. Then $\pi_{\circ}(\tau(C_3))=\pi_{\circ}(q)=[b^2+(a-1)^2:b^2+(a+1)^2]$. We claim that each element of $\R_{>0}$ is induced by a quadratic map. Note that $\pi_{\circ}(\PP^2(\R))=[0,\infty]$, and $\pi_{\circ}^{-1}(0)\cap\PP^2(\R)=\{[1:0:-1]\}$ and $\pi_{\circ}^{-1}(\infty)\cap\PP^2(\R)=\{[1:0:1]\}$. By Lemma~\ref{lem sigma}, for any real point $q\in\pi_{\circ}^{-1}(]0,\infty[)\cap\PP^2(\R)$ there exists a quadratic transformation in $\mathcal{J}_{\circ}$ with $q$ as real base-point. Hence any element of $\R_{>0}$ is the image of a quadratic map.

Any standard quintic transformation and any cubic transformation satisfying the assumptions of the lemma preserves $C_3$ (see Figure~\ref{fig:deg 3}).
\begin{figure}
\def\svgwidth{0.75\textwidth}
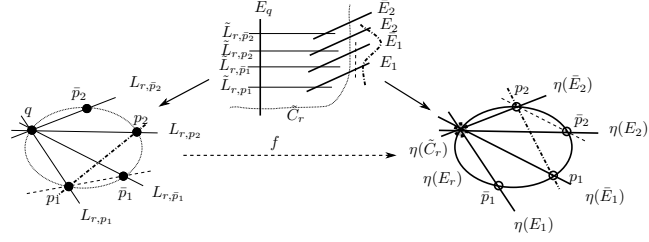
\caption{A cubic transformation that contracts $L_{p_i,q}$ onto $\bar{p}_i$.}\label{fig:deg 3}
\end{figure}
\end{proof}

Last but not least we look at what happens to the image of conics under $\pi_{\circ}$ after the re-assignment proposed in the following lemma.

\begin{Lem}\label{lem linear}
Let $q\in\PP^2$ be a non-real point not collinear with any two of $p_1,\bar{p}_1,p_2,\bar{p}_2$. Suppose that $\alpha_q\in\Aut_\R(\PP^2)$ fixes $p_1$ and sends $q$ onto $p_2$. Then $\alpha_q(C_q)=C_{\alpha_q(p_2)}$ and
\[\pi_{\circ}(C_{\alpha_q(p_2)})\in\R_{<0}\cdot\pi_{\circ}(C_q).\]
\end{Lem}

\begin{proof}
Consider the map
\[\psi\colon\PP^2(\C)\setminus\{z=0\}\longrightarrow\PP^2(\C)\setminus\{z=0\},\quad q\mapsto\alpha_q(p_2).\]
Then
\[\pi_{\circ}(C_{\alpha_q(p_2)})=\pi_{\circ}(\alpha_q(p_2))=\pi_{\circ}(\psi(q))\in\pi_{\circ}\left(\psi(C_q\setminus\{z=0\})\right),\]
and we prove the claim by computing $\pi_{\circ}(\psi(C_q\setminus\{z=0\}))\subset\R_{<0}\cdot\pi_{\circ}(C_q)$. 

Via the parametrisation 
\[\iota\colon\R^2\rightarrow\PP^2(\C),\quad(u,v,x,y)\mapsto[u+iv:x+iy:1],\] the map $\psi$ is conjugate to the real birational involution
{\small
\[\hat{\psi}\colon\R^4\dashrightarrow\R^4,\quad(u,v,x,y)\dashmapsto\small\left(\frac{uy-vx}{v^2+y^2},\frac{-v}{v^2+y^2},\frac{uv+xy}{v^2+y^2},\frac{y}{v^2+y^2}\right)\] }
whose domain is $\R^4\setminus\{v=y=0\}=\iota^{-1}\left( \PP^2(\C)\setminus(\{z=0\}\cup\PP^2(\R) )\right)$.
To understand $\psi(C_q\setminus\{z=0\})$, we use the parametrisation 
{\small
\begin{align*}\text{par}\colon&\C\longrightarrow C_q\setminus\{z=0\},\\ \tiny 
&t\mapsto\small\left[\frac{1}{2}\frac{(t-1)(t+1)(\lambda+\mu)}{\lambda t+\mu t +\lambda-\mu}:-\frac{1}{2}\frac{i(\lambda t^2+\mu t^2+2\lambda t-2\mu t+\lambda+\mu)}{\lambda t+\mu t+\lambda -\mu}:1\right],
\end{align*}
}
which is the inverse of the projection of $C_q$ centred at $p_1$. Consider the commutative diagram
{\small
\[\xymatrix{ &\iota^{-1}(C_q\setminus\{z=0\})\ar[d]^{\iota}\ar[r]^(.45){\hat{\psi}} & \hat{\psi}(\iota^{-1}(C_q\setminus\{z=0\}))\ar[d]^{\iota}&\\
\C\ar[r]^(.35){\text{par}} \ar@/_1pc/[rrr]&C_q\setminus\{z=0\}\ar[r]^(.45){\psi}&\psi(C_q\setminus\{z=0\})\ar[r]^(.7){\pi_{\circ}}&\PP^1}\]
}
We calculate that $(\pi_{\circ}\circ\psi\circ\text{par})$ is given by
{\small
\[(\pi_{\circ}\circ\psi\circ\text{par})\colon x+iy\longmapsto\small \left[\frac{-Q_{\rho,\nu}(x,y)}{4(\nu^2+\rho^2)}(\rho+i\nu):1\right] \]
}
where $\pi_{\circ}(C_q)=[\rho+i\nu:1]$ with $\rho,\nu\in\R$, $\nu\neq0$, and
 {\small \[Q_{\rho,\nu}(x,y)=(\nu^2+\rho^2+2\rho+1)(x^2+y^2)+2x(\nu^2+\rho^2-1)+4\nu y+(\nu^2+\rho^2-2\rho+1)\in\R[\rho,\nu,x,y].\]}
We consider $Q_{\rho,\nu}(x,y)\in\R[\rho,\nu,y][x]$ as polynomial in $x$ and calculate that its discriminant is negative for all $\nu,\rho,y$. Further, for all $\nu$ the coefficient $(\nu^2+\rho^2+2\rho+1)$ has negative discriminant with respect to $\rho$, hence it has only positive values. So, $Q_{\rho,\nu}(x,y)$ has only positive values, and 
{\small
\[\left\{\left[\frac{-Q_{\rho,\nu}(x,y)}{4(\nu^2+\rho^2)}(\rho+i\nu):1\right]\mid x,y\in\R \right\}\subset\R_{<0}\cdot[\rho+i\nu:1]=\R_{<0}\cdot\pi_{\circ}(C_q).\]}
\end{proof}

\subsection{The quotient}\label{subsection quotient Jcirc}
Using Lemma~\ref{lem scale on P1}, we now construct a surjective group homomorphism $\varphi_{\circ}\colon\mathcal{J}_{\circ}\rightarrow\bigoplus_{(0,1]}\Z/2\Z$. 
There are two constructions of the quotient - one geometrical and the other using the spinor norm on $\mathrm{SO}(x^2+y^2-tz^2,\R(t))$.
We first give the construction via the spinor norm and then the geometrical one.

\subsubsection{Construction of quotient using the spinor norm}\label{spinor norm}

The surjective morphism $\varphi_{\circ}\colon\mathcal{J}_{\circ}\rightarrow\bigoplus_{(0,1]}\Z/2\Z$ is given by the spinor norm as explained in the following.

\noindent{\bf Idea:} the spinor norm induces a surjective morphism $\bar{\theta}_\R\colon\mathcal{J}_{\circ}\rightarrow\bigoplus_{\HH}\Z/2\Z$, where $\HH\subset\C$ is the upper half plane. The action of $\Aut_\R(\PP^1,[0,\infty])$ on $\mathcal{J}_{\circ}$ and the re-assignment in Lemma~\ref{lem linear} then induce $\varphi_{\circ}$. 

\smallskip

\noindent{\bf The spinor norm and an induced morphism $\bar{\theta}_\R$.}
By Remark~\ref{rmk:del Pezzo} and Lemma~\ref{lem scale on P1}, the action of $\mathcal{J}_{\circ}$ on $\PP^1$ induces the split exact sequence
\[1\rightarrow\mathrm{SO}(x^2+y^2-tz^2,\R(t))\longrightarrow\mathcal{J}_{\circ}\longrightarrow\Aut_\R(\PP^1,[0,\infty])\simeq\R_{>0}\rtimes\Z/2\Z\rightarrow1.\]
For $\k\in\{\C,\R\}$, the spinor norm $\theta_\k$ is given by the exact sequence
\[0\rightarrow\Z/2\Z\rightarrow\mathrm{Spin}(x^2+y^2-tz^2,\k(t))\rightarrow\mathrm{SO}(x^2+y^2-tz^2,\k(t))\stackrel{\theta_\k}\longrightarrow\k(t)^*/(\k(t)^*)^2\]
where $\theta_\R=(\theta_\C)|_{\mathrm{SO}(x^2+y^2-tz^2,\R(t))}$ is the restriction of $\theta_\C$.  
More precisely, for a reflection $f$ at a vector $v=(a(t),b(t),c(t))$, the spinor norm is the the length of $v$ squared, i.e. 
\[\theta_\k(f)=a(t)^2+b(t)^2-tc(t)^2.\]
More information about the spinor norm may be found in \cite[\S55]{O'Me73}.

As squares are modded out, we may assume that $a(t),b(t),c(t)\in\K[t]$. An element $g\in\K[t]$ is a square if and only if every root of $g$ appears with even multiplicity. 

If $\k=\R$, we can identify $\R(t)^*/(\R(t)^*)^2$ with polynomials in $\R[t]$ having only simple roots, i.e. with $\Z/2\Z\oplus\bigoplus_{\overline{\HH}}\Z/2\Z$, where $\overline{\HH}\subset\C$ is the closed upper half plane and the first factor is the sign of the polynomial. A non-real root $a\pm ib$ is the root of $(t-a)^2+b^2$. In particular, the spinor norm induces a surjective homomorphism 
\[\bar{\theta}_\R\colon\mathrm{SO}(x^2+y^2-tz^2,\R(t))\rightarrow\bigoplus_{\HH}\Z/2\Z.\]
Let's look at it geometrically. Extending the scalars to $\C(t)$, the isomorphism $\Aut_{\R(t)}(C)\simeq\mathrm{SO}(x^2+y^2-tz^2,\R(t))$ (see Remark~\ref{rmk:del Pezzo}) extends to
\[\Aut_{\R(t)}(C)\subset\Aut_{\C(t)}(C)\simeq\mathrm{PGL}_2(\C(t))\simeq\mathrm{SO}(x^2+y^2-tz^2,\C(t))\simeq\mathrm{SO}(tx^2-yz,\C(t)),\]
where $\alpha\colon\mathrm{PGL}_2(\C(t))\stackrel{\simeq}\rightarrow\mathrm{SO}(tx^2-yz,\C(t))$ is given as follows: every automorphism of $\PP^1_{\C(t)}$ extends to an automorphism of $\PP^2_{\C(t)}$ preserving the image of $\beta\colon\PP^1_{\C(t)}\hookrightarrow\PP^2_{\C(t)}$, $[u:v]\stackrel{\beta}\mapsto[uv:tu^2:v^2]$. Then $\alpha$ is given by
\[\alpha\colon\tiny\left(\begin{matrix}a&b\\c&d\end{matrix}\right)\mapsto\left(\begin{matrix}\frac{ad+bc}{ad-bc}&\frac{ac}{t(ad-bc)}&bd\\ \frac{2abt}{ad-bc}&\frac{a^2}{ad-bc}&\frac{tb^2}{ad-bc}\\ \frac{2cd}{ad-bc}&\frac{c^2}{t(ad-bc)}&\frac{d^2}{ad-bc}\end{matrix}\right)\]
The group $\mathrm{PGL}_2(\C(t))$ is generated by its involutions: any element of $\mathrm{PGL}_2(\C(t))$ is conjugate to a matrix of the form 
\[\tiny\left(\begin{matrix}0&a\\1&b\end{matrix}\right),\]
which is a composition of involutions:
\[\tiny\left(\begin{matrix}0&a\\1&b\end{matrix}\right)=\left(\begin{matrix}0&a\\1&0\end{matrix}\right)\left(\begin{matrix}1&0\\0&-1\end{matrix}\right)\left(\begin{matrix}1&b\\0&-1\end{matrix}\right).\]
All involutions in $\mathrm{PGL}_2(\C(t))$ are conjugate to matrices of the form 
\[P:=\small\left(\begin{matrix}0&p\\1&0\end{matrix}\right),\quad p\in\C[t].\]
The image of $P$ via $\alpha$ is
\[\alpha(P)=\tiny\left(\begin{matrix}-1&0&0\\0&0&-tp\\0&-1/tp&0\end{matrix}\right)\in \mathrm{SO}(tx^2-yz,\C(t)),\]
which is a reflection at its eigenvector $v=(0,-tp,1)$ of eigenvalue $1$. In particular, $\theta(P)$ is equal to the length of $v$ squared, i.e. $\theta_\C(P)=tp\in\C(t)^*/(\C(t)^*)^2$. 

The isomorphism $\alpha\colon\mathrm{PGL}_2(\C(t))\simeq\mathrm{SO}(tx^2-yz,\C(t))$ is induced by a non-real birational map $\eta_{\alpha}\colon Z\rightarrow\PP^1_{\C}\times\PP^1_{\C}$ that contracts one component in each singular fibre. The non-zero roots of $\theta_\C(P)$ correspond to the fibres contracted by $f_P\colon\PP^1_{\C}\times\PP^1_{\C}\dashrightarrow\PP^1_{\C}\times\PP^1_{\C}$, $(x,y)\mapsto(p/x,y)$ that have an odd number of base-points on them (counted without multiplicity). So, for $f\in\Aut_{\C(t)}(C)$, the spinor norm $\theta_\C(f)$ corresponds to the fibres of $Z$ having an odd number of base-points on then (see Figure~\ref{figure theta}). Let $f\in\Aut_{\R(t)}(C)$. Since $\eta_{\alpha}$ is just the contraction of a component in each singular fibre, $\theta_\R(f)$ corresponds to the conics on which $f$ has an odd number of base-points. 
By construction, $\bar{\theta}_\R$ throws away the real roots, i.e. the real conics. 
Therefore, $\bar{\theta}_\R(f)$ corresponds to non-real conics on which $f$ has an odd number of base-points.
More precisely, 
\begin{align*}&\bar{\theta}_\R\colon f\longmapsto\bar{\theta}_\R(f)=(\bar{\theta}_\R(f)_h)_{h\in\HH}\in\bigoplus_{\HH}\Z/2\Z,\\
&(\bar{\theta}_\R(f))_h=\begin{cases}1,&\text{$f$ has an odd number of base-points on $\pi_{\circ}^{-1}(h)$}\\
0,&\text{else}\end{cases}
\end{align*}
In other words, $\bar{\theta}_\R$ counts the number of base-points of $f$ on each non-real conic modulo $2$ (see Figure~\ref{figure theta}).
\begin{figure}\label{fig:theta}
\def\svgwidth{0.8\textwidth}
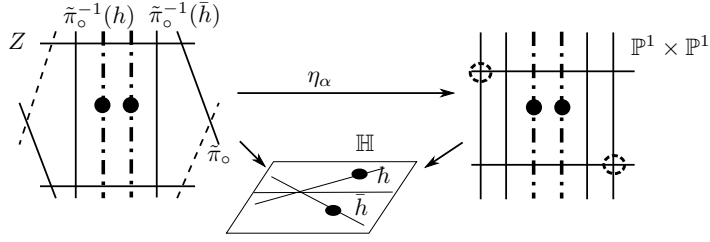
\caption{The element $\bar{\theta}_\R(f)=(\bar{\theta}_\R(f)_h)_{h\in\HH}\in\bigoplus_{\HH}\Z/2\Z$ corresponds to the fibres on which $f$ has an odd number of base-points.}\label{figure theta}
\end{figure}
\smallskip

\noindent{\bf The action of $G$, the re-assignment and the morphism $\varphi_{\circ}$.}
The group $G:=\Aut_\R(\PP^1,[0,\infty])$ acts on $\mathcal{J}_{\circ}\simeq\mathrm{SO}(x^2+y^2-tz^2,\R(t))\rtimes G$ by conjugation. It induces the quotient 
\[\mathcal{J}_{\circ}\rightarrow\mathcal{J}_{\circ}/G\simeq\mathrm{SO}(x^2+y^2-tz^2,\R(t))/G\longrightarrow\small\left(\bigoplus_{\HH}\Z/2\Z\right)/G.\]
Let us explore the action of $G$ on $\bigoplus_{\HH}\Z/2\Z$. 
The action of $G$ on $\mathcal{J}_{\circ}$ induces an action on $\PP^1$, the set of conics. For $f\in\mathcal{J}_{\circ}$ the element $\bar{\theta}_\R(f)\in\bigoplus_{\HH}\Z/2\Z$ 
corresponds to the non-real conics on which $f$ has an odd number of base-points. Then $G$ acts on $\bigoplus_{\HH}\Z/2\Z$ by acting on $\HH$, i.e. $\left(\bigoplus_{\HH}\Z/2\Z\right)/G=\bigoplus_{\HH/G}\Z/2\Z$.
By Lemma~\ref{lem scale on P1}, we have $G\simeq\R_{>0}\rtimes\Z/2\Z$ and it acts on $\PP^1$ is by real positive scaling and taking inverse. 
The upper half plane $\HH$ is the set of non-real conics up to taking inverse, and $G$ acts on it by positive real scaling. The orbit space $\HH/G$ is the upper half circle. We obtain the surjective morphism
\begin{align*}\mathcal{J}_{\circ}\longrightarrow(\bigoplus_{\HH}\Z/2\Z)/G,\ f\mapsto\begin{cases} \bar{\theta}_\R(f),&f\in\mathrm{SO}(x^2+y^2-tz^2,\R(t))\\ 0,& f\in\Aut_\R(\PP^1,[0,\infty])\end{cases}\end{align*}

We want to lift the quotient $\J_{\circ}\rightarrow \bigoplus_{\text{upper half circle}}\Z/2\Z$ onto $\Bir_\R(\PP^2)$, so recall the re-assignment in Lemma~\ref{lem linear} for the base-points of standard quintic transformations in $\J_{\circ}$. It induces a scaling by negative real numbers on the conics, and on $\HH/G$ it identifies conics having the same imaginary part. The orbit space is a quarter circle which we identify with the interval $(0,1]$, and obtain the surjective morphism
\[\varphi_{\circ}\colon\mathcal{J}_{\circ}\longrightarrow((\bigoplus_{\HH}\Z/2\Z)/G)\rightarrow\bigoplus_{(0,1]}\Z/2\Z.\]
\FloatBarrier

\subsubsection{Geometric construction}\label{geometric construction}
What follows is the geometric construction of $\varphi_{\circ}\colon\mathcal{J}_{\circ}\rightarrow\bigoplus_{(0,1]}\Z/2\Z$. In the rest of the paper, we will use this description of the quotient and in particular Remark~\ref{rmk map}, as the working tools are geometric ones. 

\begin{Def}\label{def map Jcirc}
Let $f\in\mathcal{J}_{\circ}$. For any {\em non-real} base-point $q$ of $f$ different from $p_1,\bar{p}_1,p_2,\bar{p}_2$, we have
$\pi_{\circ}(C_q)=[a+ib:1]$ and $\pi_{\circ}(C_{\bar{q}})=[a-ib:1]$ for some $a,b\in\R$, $b\neq0$ (see Definition~\ref{def basic} for the definition of $C_q$). We define 
\[\nu(C_q):=1-\frac{\mid a\mid}{a^2+b^2}\ \in(0,1].\]
Note that $\nu(C_{q'})=\nu(C_q)$ if and only if $\pi_{\circ}(C_q)=\lambda\pi_{\circ}(C_{q'})$ or $\pi_{\circ}(C_q)=\lambda\pi_{\circ}(C_{\bar{q}'})$ for some $\lambda\in\R$. In fact, seeing $\PP^1$ as space of conics, $\nu\colon\PP^1\setminus\PP^1(\R)\rightarrow(0,1]\hat{=}(\HH/G)/\R_{<0}$ is the quotient from Subsection~\ref{spinor norm}. 

We define $e_{\delta}\in\oplus_{(0,1]}\Z/2\Z$ to be the "standard vector" given by
\[(e_{\delta})_{\varepsilon}=\begin{cases}1,&\delta=\varepsilon\\ 0,&\text{else}\end{cases}\]
\end{Def}

\begin{figure}\label{figure quotient}
\def\svgwidth{0.6\textwidth}
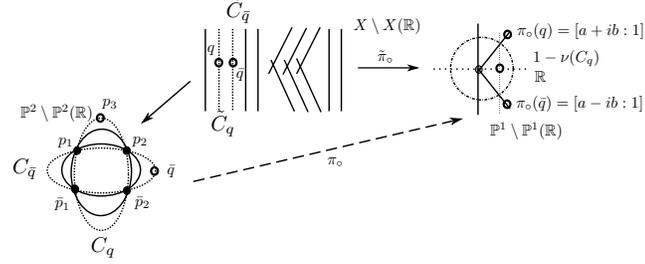
\caption{The map $\nu$ in Definition~\ref{def map Jcirc}}\label{figure 1}
\end{figure}

\begin{Def}\label{def varphicirc}
Let $f\in\mathcal{J}_{\circ}$ and $S(f)$ be the set of non-real conjugate {\em pairs} of base-points of $f$ different from $p_1,\dots,\bar{p}_2$. We define
\[\varphi_{\circ}\colon\mathcal{J}_{\circ}\longrightarrow\bigoplus_{(0,1]}\Z/2\Z,\qquad f\longmapsto\sum\limits_{(q,\bar{q})\in S(f)}e_{\nu(C_q)}\]
which is a well defined map according to Definition~\ref{def map Jcirc}.
\end{Def}

The construction of $\varphi_{\circ}$ and Remark~\ref{rmk map}~(\ref{rmk map 1}) yield the following lemma.
\begin{Lem}\label{lem hom}
The map $\varphi_{\circ}\colon\mathcal{J}_{\circ}\longrightarrow\bigoplus_{(0,1]}\Z/2\Z$ in Definition~$\ref{def varphicirc}$ coincides with the surjective homomorphism constructed in Section~$\ref{spinor norm}$. Its kernel contains all elements of degree $\leq4$.
\end{Lem}

\noindent By looking at the linear systems of transformations, one can also prove directly that the map $\varphi_{\circ}$ in Definition~\ref{def varphicirc} is a homomorphism by using Lemma~\ref{lem Jcirc gen} and Remark~\ref{rmk map}.

\begin{Rmk}\label{rmk map}The following remarks directly follow from the definition of $\varphi_{\circ}$.
\begin{enumerate}
\item\label{rmk map 0} If $S(f)=\emptyset$, then $\varphi_0(f)=0$. 
\item\label{rmk map 1} Let $f\in\mathcal{J}_{\circ}$ of degree $\leq4$. Its characteristic (see Lemma~\ref{lem char}) implies that the non-real base-points of $f$ are among $p_1,\bar{p}_1,p_2,\bar{p}_2$. Thus, the set $S(f)$ is empty and $\varphi_0(f)=0$.
\item Let $\theta\in\mathcal{J}_{\circ}$ be a standard quintic transformation. Then $|S(f)|=1$ and therefore $\varphi_{\circ}(\theta)$ is a "standard vector" by Definition~\ref{def varphicirc}.
\item\label{rmk map 3} It follows from the definition of standard quintic transformations (Definition~\ref{def 5.1}) that for every $\delta\in(0,1]$ there exists a standard quintic transformation $\theta\in\mathcal{J}_{\circ}$ such that $\varphi_{\circ}(\theta)=e_{\delta}$.
\item Let $\theta_1,\theta_2\in\mathcal{J}_{\circ}$ be standard quintic transformations and $S(\theta_i)=\{(q_i,\bar{q}_i)\}$, $i=1,2$. If $C_{q_1}=C_{q_2}$ (or $C_{q_1}=C_{\bar{q}_2}$), then $\varphi_{\circ}(\theta_1)=\varphi_{\circ}(\theta_2)$. 
\item\label{rmk map 2} Let $\theta\in\mathcal{J}_{\circ}$ be a standard quintic transformation. Let $S(\theta)=\{(q_1,\bar{q}_1)\}$ and $S(\theta^{-1})=\{(q_2,\bar{q}_2)\}$. Since $\theta$ induces $\Id$ or $[x:y]\mapsto[y:x]$ on $\PP^1$ (Lemma~\ref{lem scale on P1}), it follows that $\nu(C_{q_1})=\nu(C_{q_2})$ and in particular $\varphi_{\circ}(\theta)=\varphi_{\circ}(\theta^{-1})$. 
\item\label{rmk map 5} Let $f\in\mathcal{J}_{\circ}$ and $C$ be any non-real conic passing through $p_1,\dots,\bar{p}_2$. The automorphism $\hat{f}$ on $\PP^1$ induced by $f$ is a scaling by a positive real number (Lemma~\ref{lem scale on P1}), thus $\nu\circ\hat{f}=\nu$. In particular, $e_{\nu(f(C))}=e_{\nu(\hat{f}(C))}=e_{\nu(C)}$.
\end{enumerate}
\end{Rmk}

\section{Presentation of $\Bir_\R(\PP^2)$ by generating sets and relations}\label{section technical thm}

This section is devoted to the rather technical proof of Theorem~\ref{prop technical thm}. We remind of the notation $p_1:=[1:i:0],p_2:=[0:1:i]$. 

\subsection{The presentation}

The family of standard quintic transformations plays an important role in $\Bir_\R(\PP^2)$. Recall the two quadratic transformations 
\[\sigma_0\colon [x:y:z]\dasharrow [yz:xz:xy],\qquad\sigma_1\colon [x:y:z]\dasharrow [y^2+z^2:xy:xz].\]

\begin{Thm}[\cite{RV05},\cite{BM12}]\label{thm BM}
The group $\Bir_\R(\PP^2)$ is generated by $\sigma_0,\sigma_1$, $\Aut_\R(\PP^2)$ and all standard quintic transformations. 
\end{Thm}

From $\sigma_0\in\mathcal{J}_*$, $\sigma_1\in\mathcal{J}_{\circ}$ and Lemma~\ref{lem quintlin}, we obtain the following corollary:

\begin{Cor}\label{cor BM}
The group $\Bir_\R(\PP^2)$ is generated by $\Aut_\R(\PP^2)$, $\mathcal{J}_*$, $\mathcal{J}_{\circ}$. 
\end{Cor}

Using these generating groups, we can give a representation of $\Bir_\R(\PP^2)$ in terms of generating sets and relations:

Define $S:=\Aut_\R(\PP^2)\cup\mathcal{J}_*\cup\mathcal{J}_{\circ}$ and let $F_S$ be the free group generated by $S$. Let $w\colon S\rightarrow F_S$ be the canonical word map.

\begin{Def}\label{def G}
We denote by $\mathcal{G}$ be the following group:
\[F_{S} / \scalebox{1.4}{\Bigg\langle} \begin{array}{ll}w(f)w(g)w(h),& f,g,h\in\Aut_{\R}(\PP^2), \ fgh=1\ \text{in}\ \Aut_{\R}(\PP^2)\\
 w(f)w(g)w(h), & f,g,h\in\mathcal{J}_*,\ fgh=1\ \text{in}\ \mathcal{J}_*\\ 
 w(f)w(g)w(h),&f,g,h\in\mathcal{J}_{\circ},\ fgh=1\ \text{in}\ \mathcal{J}_{\circ}\\ 
\text{the relations in the list below}\end{array}\scalebox{1.4}{\Bigg\rangle}\]

\begin{erel}
\item\label{relrel1} Let $\theta_1,\theta_2\in\mathcal{J}_{\circ}$ be standard quintic transformations and $\alpha_1,\alpha_2\in\Aut_{\R}(\PP^2)$. 
\[w(\alpha_2)w(\theta_1)w(\alpha_1)=w(\theta_2)\ \text{in}\ \mathcal{G}\quad \text{if}\quad \alpha_2\theta_1\alpha_1=\theta_2.\]

\item\label{relrel2} Let $\tau_1,\tau_2\in\mathcal{J}_*\cup\mathcal{J}_{\circ}$ be both of degree $2$ or of degree $3$ and $\alpha_1,\alpha_2\in\Aut_{\R}(\PP^2)$.  
\[w(\tau_1)w(\alpha_1)=w(\alpha_2)w(\tau_2)\ \text{in}\ \mathcal{G}\quad \text{if}\quad \tau_1\alpha_1=\alpha_2\tau_2.\]

\item\label{relrel3} Let $\tau_1,\tau_2,\tau_3\in\mathcal{J}_*$ all be of degree 2, or $\tau_1,\tau_2$ of degree 2 and $\tau_3$ of degree 3, and $\alpha_1,\alpha_2,\alpha_3\in\Aut_{\R}(\PP^2)$.
\[w(\tau_2)w(\alpha_1)w(\tau_1)=w(\alpha_3)w(\tau_3)w(\alpha_2)\quad \text{in}\quad \mathcal{G}\ \text{if}\ \tau_2\alpha_1\tau_1=\alpha_3\tau_3\alpha_2.\]

\end{erel}
\end{Def} 

\begin{Thm}[Structure theorem]\label{prop technical thm}
The natural surjective group homomorphism $\mathcal{G}\rightarrow\Bir_\R(\PP^2)$ is an isomorphism.
\end{Thm}
Its proof is situated at the very end of Section~\ref{section technical thm}. The method to prove it is to study linear systems of birational transformations of $\PP^2$ and their base-points, and has been described in \cite{Bla12}, \cite{I85} and \cite{Z15}.
 
\begin{Rmk}\label{rmk general amalgam} 
The generalised amalgamated product of $\Aut_\R(\PP^2)$, $\mathcal{J}_*$, $\mathcal{J}_{\circ}$ along all pairwise intersections is the quotient of the free product of the three groups modulo the relations given by the pairwise intersections.\par
Note that the group $\mathcal{G}$ is isomorphic to the quotient of the generalised amalgamated product of $\Aut_\R(\PP^2)$, $\mathcal{J}_*$, $\mathcal{J}_{\circ}$ along all pairwise intersections by relations \ref{relrel1}, \ref{relrel2} and \ref{relrel3}. \par
Since $\Bir_\R(\PP^2)$ is generated by $\Aut_\R(\PP^2),\mathcal{J}_*,\mathcal{J}_{\circ}$ (Corollary~\ref{cor BM}), there exists a natural surjective group homomorphism $F_S\rightarrow\Bir_\R(\PP^2)$ which factors through a group homomorphism $\mathcal{G}\rightarrow\Bir_\R(\PP^2)$ because all relations above hold in $\Bir_\R(\PP^2)$.
\end{Rmk}

By abuse of notation, we also denote by
\[w\colon\Aut_{\R}\cup\mathcal{J}_*\cup\mathcal{J}_{\circ}\rightarrow\mathcal{G}\]
the composition of $S\rightarrow F_S$ with the canonical projection $F_S\rightarrow\mathcal{G}$.

\begin{Rmk}\label{rmk rel sp}
Suppose $\theta_1,\theta_2\in\mathcal{J}_{\circ}$ are special quintic transformations (see Definition~\ref{def 5.2}). If there exist $\alpha_1,\alpha_2\in\Aut_\R(\PP^2)$ such that
$\theta_2=\alpha_2\theta_1\alpha_1$ then $\alpha_1,\alpha_2$ permute $p_1,\bar{p}_1,p_2,\bar{p}_2$ and are thus contained in $\mathcal{J}_{\circ}$. So, the relation
\begin{equation}w(\theta_2)=w(\alpha_2)w(\theta_1)w(\alpha_2)\quad \text{if}\quad\theta_2=\alpha_2\theta_1\alpha_1\ \text{in}\ \Bir_\R(\PP^2)\tag{\bf rel. 4}\label{relrel4}\end{equation}
is true in $\mathcal{G}$ and even in the generalised amalgamated product of $\Aut_\R(\PP^2)\,\mathcal{J}_*,\mathcal{J}_{\circ}$ along all the pairwise intersections. Therefore, we need not list this relation in Definition~\ref{def G}.
\end{Rmk}

\begin{Rmk}
In the definition of standard (resp. special) quintic transformations, we choose an identification of $Y$ (reps. $Y_0$) with $\PP^2$. 
Changing this choice means composing from the left with an automorphism of $\PP^2$, and relations \ref{relrel1} and (\ref{relrel4}) are not affected by this: 
Let $\theta_1,\theta_2\in\Bir_\R(\p^2)$ standard (resp. special) quintic transformations such that $\theta_2=\alpha\theta_1$ for some $\alpha\in\Aut_\R(\PP^2)$. By Lemma~\ref{lem quintlin} there exist $\beta_1,\beta_2,\gamma_1,\gamma_2\in\Aut_\R(\PP^2)$ such that 
\[\theta_1':=\beta_2\theta_1\beta_1\in\J_{\circ},\quad\theta_2':=\gamma_2\theta_2\gamma_1\in\J_{\circ}.\]
Then $\theta_2'=(\gamma_2\alpha_2\beta_2^{-1})\theta_1'(\beta_1^{-1}\gamma_1)$,
and relation~\ref{relrel1} implies that
\[w(\theta_2')=w(\gamma_2\alpha_2\beta_1^{-1})w(\theta_1)w(\beta_2^{-1}\gamma_1).\]
It means that the relations $\theta_2=\alpha\theta_1$ holds in $\mathcal{G}$ as well. 
\end{Rmk}

\begin{Rmk}\label{rmk examples}
In the proof of Theorem~\ref{prop technical thm}, relations \ref{relrel1}, \ref{relrel2} and \ref{relrel3} mostly turn up in the form of the following examples:\par 
{\bf(1)} Example of \ref{relrel1} - changing pencil: Let $\theta\in\mathcal{J}_{\circ}$ be a standard quintic transformation (see Definition~\ref{def 5.1}). Call its base-points $p_1,\bar{p}_1,p_2,\bar{p}_2,p_3,\bar{p}_3$, and the base-points of its inverse $p_1,\bar{p}_1,p_2,\bar{p}_2,p_4,\bar{p}_4$  where $p_3,p_4$ are non-real proper points of $\PP^2$. By Lemma~\ref{lem 5.1} and Remark~\ref{rmk quintlin 2} it sends the pencil of conics through $p_1,\bar{p}_1,p_3,\bar{p}_3$ onto the one through $p_1,\bar{p}_1,p_4,\bar{p}_4$ (or $p_2,\bar{p}_2,p_4,\bar{p}_4$, in which case we proceed analogously).
For $i=3,4$, the four points $p_1,\bar{p}_1,p_i,\bar{p}_i$ are not collinear, so there exist $\alpha_{1},\alpha_{2}\in\Aut_{\R}(\PP^2)$ such that $\alpha_1(\{p_1,p_3\})=\{p_1,p_2\}$ and $\alpha_2(\{p_1,p_4\})=\{p_1,p_2\}$.  Then $\alpha_2\theta\alpha_1^{-1}\in\mathcal{J}_{\circ}$ is a standard quintic transformation and the relation  $w(\alpha_2)w(\theta)w(\alpha_1)=w(\alpha_2\theta\alpha_1)$ is an example of \ref{relrel1}.

{\bf(2)} Example of \ref{relrel2} - changing pencil: Let $\tau\in\mathcal{J}_{\circ}$ be of degree 2 or $3$. By Lemma~\ref{lem char}, $\tau$ has exactly one real base-point. Let $r$ be the real base-point of $\tau$ and $s$ the real base-point of $\tau^{-1}$. Observe that $\tau$ sends the pencil of lines through $r$ onto the pencil of lines through $s$. There exist $\alpha_1,\alpha_2\in\Aut_{\R}(\PP^2)$ such that $(\alpha_1)^{-1}(r)=[1:0:0]=\alpha_2(s)$. Then $\alpha_2\tau\alpha_1$ is an element of $\mathcal{J}_*$ and the relation $w(\alpha_2)w(\tau)w(\alpha_1)=w(\alpha_2\tau\alpha_1)$ is an example of \ref{relrel2}.

{\bf(3)} Example of \ref{relrel2}: Let $\tau_1,\tau_2\in\mathcal{J}_{\circ}$ of degree 2 or 3 and suppose there exists $\alpha\in\Aut_\R(\PP^2)$ that sends the base-points of $(\tau_1)^{-1}$ onto the base-points of $\tau_2$.  
Then $\tau_2\alpha(\tau_1)^{-1}$ is linear and the relation $w(\tau_2)w(\alpha)w((\tau_1)^{-1})=w(\tau_2\alpha(\tau_1)^{-1})$ is an other example of \ref{relrel2}.

{\bf(4)} Example of \ref{relrel3}: Let $\tau_1,\tau_2\in\mathcal{J}_*$ be of degree 2 with base-points $p:=[1:0:0],r_1,r_2$ and $p,s_1,s_2$ respectively, and $\alpha\in\Aut_{\R}(\PP^2)$ with $\alpha(r_i)=s_i$ but $\alpha(p)\neq p$ (i.e. $\alpha\notin\mathcal{J}_*$). 
Then $\tau_3:=\tau_2\alpha(\tau_1)^{-1}$ is of degree $2$ 
and there exist $\beta_1,\beta_2\in\Aut_\R(\PP^2)$ such that $\beta_2\tau_3\beta_1\in\mathcal{J}_*$. The relation $w(\beta_2^{-1})w(\beta_2\tau_3\beta_1)w(\beta_1^{-1})=w(\tau_2)w(\alpha)w(\tau_1)$ is an example of \ref{relrel3}.
\end{Rmk}

\subsection{Proof of the structure theorem}
For a linear system $\Lambda$ in $\PP^2$, we write $m_{\Lambda}(q)$ for the multiplicity of $\Lambda$ in a point $q$. Similarly, for $f\in\Bir_\R(\PP^2)$, we write $m_f(q)$ for the multiplicity of $f$ in the point $q$.

\begin{Lem}\label{lem deg mult ineq}
Let $f\in\mathcal{J}_*\cup\mathcal{J}_{\circ}$ be non-linear and $\Lambda$ be a real linear system of degree $\deg(\Lambda)=D$. Suppose that 
\[\deg(f(\Lambda))\leq D\quad (\text{resp.}\ <D).\]
\begin{enumerate}
\item  If $f\in\mathcal{J}_*$, there exist two real or a pair of non-real conjugate base-points  $q_1,q_2$ of $f$ such that
\[m_{\Lambda}([1:0:0])+m_{\Lambda}(q_1)+m_{\Lambda}(q_2)\geq D\quad(\text{resp.}\ >D)\]

\item Suppose that $f\in\mathcal{J}_{\circ}$. Then there exists a base-point $q\notin\{p_1,\bar{p}_1,p_2,\bar{p}_2\}$ of $f$ of multiplicity $2$ such that
\[(2.1)\quad m_{\Lambda}(p_1)+m_{\Lambda}(p_2)+m_{\Lambda}(q)\geq D\quad(\text{resp.}\ >D)\]
or $f$ has a (single) simple base-point $r$ and
\[(2.2)\qquad 2m_{\Lambda}(p_i)+m_{\Lambda}(r)\geq D\quad(\text{resp.}\ >),\quad\text{where}\ m_f(p_i)=\nicefrac{\deg(f)}{2}.\]
\end{enumerate}
\end{Lem}

\begin{proof}
Define $d:=\deg(f)$ to be the degree of $f$. The claim for ``$<$" follows by putting a strict inequality everywhere below. 

(1) Suppose that $f\in\mathcal{J}_{*}$. Its characteristic is $(d;d-1,1^{2d-2})$ because it preserves the pencil of lines through $[1:0:0]$. Let $r_1,\dots,r_{2d-2}$ be its simple base-points. Since non-real base-points come in pairs, $f$ has an even number $N$ of real base-points. 
We order the base-points such that either $r_{2i-1},r_{2i}$ are real or $r_{2i}=\bar{r}_{2i-1}$ for $i=1,\dots,d-1$. Then
{\small
\begin{align*}\small D\geq\deg(f(\Lambda))=&dD-(d-1)m_{\Lambda}([1:0:0])-\sum_{i=1}^{d-1}(m_{\Lambda}(r_{2i-1})+m_{\Lambda}(r_{2i}))\\
=&D+\sum_{i=1}^{d-1}\left(D-m_{\Lambda}([1:0:0])-m_{\Lambda}(r_{2i-1})-m_{\Lambda}(r_{2i})\right)
\end{align*}
}
Hence there exists $i_0$ such that $D\leq m_0-m_{\Lambda}(r_{2i_0-1})-m_{\Lambda}(r_{2i_0})$.

(2) Suppose that $f\in\mathcal{J}_{\circ}$. By Lemma~\ref{lem char}, its characteristic is $(d;(\frac{d-1}{2})^4,2^{\frac{d-1}{2}})$ or $(d;(\nicefrac{d}{2})^2,(\frac{d-2}{2})^2,2^{\frac{d-2}{2}},1)$.

Assume that $f$ has no simple base-point, i.e. no base-point of multiplicity $1$. Call $r_1,\dots,r_{(d-1)/2}$ its base-points of multiplicity $2$ distinct from $p_1,\bar{p}_1,p_2,\bar{p}_2$. Then
 {\small
\begin{align*}D\geq \deg(f(\Lambda))&=dD-2m_{\Lambda}(p_1)\cdot\frac{d-1}{2}-2m_{\Lambda}(p_1)\cdot\frac{d-1}{2}-2\sum_{i=1}^{(d-1)/2}m_{\Lambda}(r_i)\\
&=D+2\sum_{i=1}^{(d-1)/2}\left(D-m_{\Lambda}(p_1)-m_{\Lambda}(p_2)-m_{\Lambda}(r_i)\right)
\end{align*}
}
which implies that there exists $i_0$ such that $0\geq D-m_{\Lambda}(p_1)-m_{\Lambda}(p_1)-m_{\Lambda}(r_{i_0})$. The claim for "$>$" follows analogously.

Assume that $f$ has a simple base-point $s$. Let $r_1,\dots,r_{(d-2)/2}$ be its base-points of multiplicity $2$ distinct from $p_1,\bar{p}_1,p_2,\bar{p}_2$. Then
{\small 
\begin{align*} D\geq\deg(f(\Lambda))&=dD-2m_{\Lambda}(p_j)\cdot\frac{d}{2}-2m_{\Lambda}(p_k)\cdot\frac{d-2}{2}-(2\sum_{i=1}^{(d-2)/2}m_{\Lambda}(r_i))-m_{\Lambda}(s)\\
&=D+(D-2m_{\Lambda}(p_j)-m_{\Lambda}(s))+2\sum_{i=1}^{(d-2)/2}(D-m_{\Lambda}(p_j)-m_{\Lambda}(p_k)-m_{\Lambda}(r_i))
\end{align*}
}
where $\{j,k\}=\{1,2\}$. The inequality implies there exist $i_0$ such that $0\geq D-m_{\Lambda}(p_j)-m_{\Lambda}(p_k)-m_{\Lambda}(r_{i_0})$ or that $0\geq D-2m_{\Lambda}(p_j)-m_{\Lambda}(s)$. 
\end{proof}

\begin{Not}
In the following diagrams, the points in the brackets are the base-points of the corresponding birational map (arrow). A dashed arrow indicates a birational map, and a drawn out arrow a linear tranformation.\par
Let $f_1,\dots,f_n\in\Aut_\R(\PP^2)\cup\mathcal{J}_*\cup\mathcal{J}_{\circ}$ such that $f_n\cdots f_1=\Id$. If $w(f_n)\cdots w(f_1)=1$ in the group $\mathcal{G}$, we say that the diagram
\[\xymatrix{\PP^2\ar@{-->}[r]^{f_1} &\PP^2\ar@{..>}[r]&\PP^2\ar@{-->}[r]^{f_{n-1}}&\PP^2\ar@/^1pc/@{-->}[lll]^{f_n}}\]
{\em corresponds to a relation in $\mathcal{G}$} or is {\em generated by relations} in $\mathcal{G}$. In the sequel, we replace $\PP^2$ by a linear system $\Lambda$ of curves in $\PP^2$ and its images by $f_1,\dots,f_{n-1}$.
\end{Not}

\begin{Lem}\label{lem rel 5,5}
Let $f,h\in\mathcal{J}_{\circ}$ be standard or special quintic transformations, $g\in\Aut_{\R}(\PP^2)$ and $\Lambda$ be a real linear system of degree $D$. Suppose that
\[\deg(h^{-1}(\Lambda))\leq D\quad \text{and}\quad \deg(fg(\Lambda))<D.\]
Then there exists $\theta_1\in\Aut_{\R}(\PP^2)$, $\theta_2,\dots,\theta_n\in\Aut_{\R}(\PP^2)\cup\mathcal{J}_{\circ}$ such that
\begin{enumerate}
\item $w(f)w(g)w(h)=w(\theta_n)\cdots w(\theta_1)$ holds in $\mathcal{G}$, i.e. the following diagram corresponds to a relation in $\mathcal{G}$:
\[\small\xymatrix{&\Lambda\ar[r]^g &g(\Lambda)\ar@{-->}[dr]^f&\\ 
h^{-1}(\Lambda)\ar@{-->}[ru]^h\ar@{-->}[r]^{\theta_1}&\ar@{..>}[r]&\ar@{-->}[r]^{\theta_n}&fg(\Lambda) }\]
\item $\deg(\theta_i\cdots\theta_1h^{-1}(\Lambda))<D$ for $i=2,\dots,n$.
\end{enumerate}
\end{Lem}

\begin{proof}
The maps $h^{-1}$ and $f$ have base-points $p_1,\bar{p}_1,p_2,\bar{p}_2,p_3,\bar{p}_3$ and $p_1,\bar{p}_1,p_2,\bar{p}_2,p_4,\bar{p}_4$ respectively, for some non-real points $p_3,p_4$ that are in $\PP^2$ or infinitely near one of $p_1,\dots,\bar{p}_2$. Denote by $m(q):=m_{\Lambda}(q)$ the multiplicity of $\Lambda$ at $q$. According to Lemma~\ref{lem deg mult ineq} we have
\begin{equation}\tag{$\text{Ineq}^0$}\label{eqn0}m(p_1)+m(p_2)+m(p_3)\geq D,\quad m_{g(\Lambda)}(p_1)+m_{g(\Lambda)}(p_2)+m_{g(\Lambda)}(p_4)>D\end{equation}
For a pair of non-real points $q,\bar{q}\in\PP^2$ or infinitely near, we denote by $\mathfrak{q}$ the set $\{q,\bar{q}\}$. We choose $\mathfrak{r}_1,\mathfrak{r}_2,\mathfrak{r}_3$ with $\{\mathfrak{r}_1,\mathfrak{r}_2,\mathfrak{r}_3\}=\{\mathfrak{p}_1,\mathfrak{p}_2,\mathfrak{p}_3\}$ such that $m_{\Lambda}(r_1)\geq m_{\Lambda}(r_2)\geq m_{\Lambda}(r_3)$ and such that if $r_i$ is infinitely near $r_j$, then $j<i$. In a similar way, we choose $\mathfrak{r}_4,\mathfrak{r}_5,\mathfrak{r}_6$ with $\{\mathfrak{r}_4,\mathfrak{r}_5,\mathfrak{r}_6\}=g^{-1}(\{\mathfrak{p}_1,\mathfrak{p}_2,\mathfrak{p}_4\})$. In particular, $r_1,r_4$ are proper points of $\PP^2$.

The two inequalities (\ref{eqn0}) translate to
\begin{equation}\tag{$\text{Ineq}^1$}\label{eqn1}m_{\Lambda}(r_1)+m_{\Lambda}(r_2)+m_{\Lambda}(r_3)\geq D,\quad m_{\Lambda}(r_4)+m_{\Lambda}(r_5)+m_{\Lambda}(r_6)>D\end{equation}

We now look at four cases, depending of the number of common base-points of $fg$ and $h^{-1}$. 
\vskip\baselineskip

\underline{Case 0:} If $h^{-1}$ and $fg$ have six common base-points, then $\alpha:=fgh$ is linear and $w(g)w(h)w(\alpha^{-1})=w(f^{-1})$ by Definition~\ref{def G}~\ref{relrel1} and Remark~\ref{rmk rel sp}~(\ref{relrel4}). The claim follows with $\theta_1=\theta_n=fgh$. 
\smallskip

\underline{Case 1}: {\em Suppose that $h^{-1}$ and $fg$ have exactly four common base-points.}
There exists $\alpha_1\in\Aut_{\R}(\PP^2)$ sending the common base-points onto $p_1,\dots,\bar{p}_2$ if all the common points are proper points of $\PP^2$, and onto $p_i,\bar{p}_i,p_3,\bar{p}_3$ if $p_3,\bar{p}_3$ are infinitely near $p_i,\bar{p}_i$. There exist $\alpha_2,\alpha_3\in\Aut_\R(\PP^2)$ such that $\alpha_1h\alpha_2\in\mathcal{J}_{\circ}$ and $\alpha_3fg(\alpha_1)^{-1}\in\mathcal{J}_{\circ}$ (see Lemma~\ref{lem quintlin}). Definition~\ref{def G}~\ref{relrel1} and Remark~\ref{rmk rel sp}~(\ref{relrel4}) imply that
\[w(\alpha_1)w(h)w(\alpha_2)=w(\alpha_1h\alpha_2),\quad w(\alpha_3)w(f)w((\alpha_1)^{-1})=w(\alpha_3fg(\alpha_1)^{-1}).\]
Since $\alpha_1h\alpha_2\in\mathcal{J}_{\circ}$ and $\alpha_3fg(\alpha_1)^{-1}\in\mathcal{J}_{\circ}$, we get
\begin{align*}
w(f)w(g)w(h)&=w((\alpha_3)^{-1})\ w(\alpha_3f(\alpha_1)^{-1})\ w(\alpha_1h\alpha_2)\ w((\alpha_2)^{-1})\\
&=w((\alpha_3)^{-1})\ w(\alpha_3fh\alpha_2)\ w((\alpha_2)^{-1})
\end{align*}
The claim follows with $\theta_1:=(\alpha_2)^{-1}$, $\theta_2:=\alpha_3fh\alpha_2\in\J_{\circ}$ and $\theta_3:=(\alpha_3)^{-1}$. 

\smallskip

\underline{Case 2:} {\em Suppose that the set $\mathfrak{r}_1\cup\mathfrak{r}_2\cup\mathfrak{r}_4\cup\mathfrak{r}_5$  consists of $6$ points $r_{i_1},\bar{r}_{i_1},\dots,r_{i_3},\bar{r}_{i_3}$.} 
If at least four of them are proper points of $\PP^2$, inequality~(\ref{eqn1}) yields 
\[2m_{\Lambda}(r_{i_1})+2m_{\Lambda}(r_{i_2})+2m_{\Lambda}(r_{i_3})>D,\]
which implies that the six points $r_{i_1},\bar{r}_{i_1},\dots,r_{i_3},\bar{r}_{i_3}$ are not contained in one conic. 
By this and by the chosen ordering of the points, there exists  a standard or special quintic transformation $\theta\in\mathcal{J}_{\circ}$ and $\alpha\in\Aut_{\R}(\PP^2)$ such that those six points are the base-points of $\theta\alpha$. By construction, we have 
\[\deg(\theta\alpha(\Lambda))=5D-4m_{\Lambda}(r_{i_1})-4m_{\Lambda}(r_{i_2})-4m_{\Lambda}(r_{i_3})<D,\] 
and $h^{-1},\theta\alpha$ and $\theta\alpha,fg$ each have four common base-points. We apply Case 1 to $h,\alpha,\theta$ and to $\theta^{-1} ,g\alpha^{-1},f$. 

If only two of the six points are proper points of $\PP^2$, then the chosen ordering yields $\mathfrak{q}=\mathfrak{r}_1=\mathfrak{r}_4$ and the points in $\mathfrak{r}_2\cup\mathfrak{r}_5$ are infinitely near points. Since $h,f$ are standard or special quintic transformations (so have at most two infinitely near base-points), it follows that $r_3,r_6$ are both proper points of $\PP^2$. We choose $i\in\{3,6\}$, $j\in\{2,5\}$ with $m_{\Lambda}(r_i)=\max\{m_{\Lambda}(r_3),m_{\Lambda}(r_6)\}$ and $m_{\Lambda}(r_j)=\max\{m_{\Lambda}(r_2),m_{\Lambda}(r_5)\}$.
We have 
\[2m_{\Lambda}(r_1)+2m_{\Lambda}(r_j)+2m_{\Lambda}(r_i)\geq 2m_{\Lambda}(r_4)+2m_{\Lambda}(r_5)+2m_{\Lambda}(r_6)>D.\] 
Thus the six points in $\mathfrak{r}_1\cup\mathfrak{r}_i\cup\mathfrak{r}_j$ are not contained in one conic and there exists a standard or special quintic transformation $\theta\in\mathcal{J}_{\circ}$ and $\alpha\in\Aut_{\R}(\PP^2)$ such those points are the base-points of $\theta\alpha$. Again, the maps $h^{-1},\theta\alpha$ and $\theta\alpha,fg$ have four common base-points. The above inequality implies $\deg(\theta\alpha(\Lambda))<D$. The maps $h,\alpha,\theta$ and the maps $\theta^{-1},g\alpha^{-1},f$ satisfy the assumptions of Case 1, the latter with ``$<$". We proceed as in Case 1 to get $\theta_1,\dots,\theta_n$. 

\smallskip

\underline{Case 3:} {\em Suppose that $\mathfrak{r}_1\cup\mathfrak{r}_2\cup\mathfrak{r}_4\cup\mathfrak{r}_5$ consists of eight points.} Then $\mathfrak{r}_1\cup\mathfrak{r}_2\cup\mathfrak{r}_4$ and $\mathfrak{r}_1\cup\mathfrak{r}_4\cup\mathfrak{r}_5$ each consist of six points. We have by inequality~(\ref{eqn1}) and by the chosen ordering that
\[2m_{\Lambda}(r_1)+2m_{\Lambda}(r_2)+2m_{\Lambda}(r_4)>2D,\qquad 2m_{\Lambda}(r_1)+2m_{\Lambda}(r_4)+2m_{\Lambda}(r_5)>2D,\]
so the points in each set $\mathfrak{r}_1\cup \mathfrak{r}_2\cup\mathfrak{r}_4$ and $\mathfrak{r}_1\cup\mathfrak{r}_4\cup\mathfrak{r}_5$ are not on one conic. Moreover, at least four points in each set are proper points of $\PP^2$ (because $r_1,r_4\in\PP^2$). Therefore, there exist standard or special quintic transformations $\theta_1,\theta_2\in\mathcal{J}_{\circ}$ and $\alpha_1,\alpha_2\in\Aut_{\R}(\PP^2)$ such that $\theta_1\alpha_1$ (resp. $\theta_2\alpha_2$) has base-points $\mathfrak{r}_1\cup\mathfrak{r}_2\cup\mathfrak{r}_4$ (resp. $\mathfrak{r}_1\cup\mathfrak{r}_4\cup\mathfrak{r}_5$). The above inequalities imply $\deg(\theta_i\alpha_i(\Lambda))<D$. The maps $h,\alpha_1^{-1},\theta_1$,  the maps $(\theta_1)^{-1},\alpha_2(\alpha_1)^{-1},\theta_2$ and the maps $(\theta_2)^{-1},g(\alpha_2)^{-1},f$ satisfy the assumptions of Case 1, the latter two with ``$<$". Proceeding analogously yields $\theta_1,\dots,\theta_n$.
\end{proof}

\begin{Rmk}\label{lem notcoll}
Let $f\in\mathcal{J}_*$, and $q_1,q_2$ two simple base-points of $f$. Then the points $[1:0:0],q_1,q_2$ are not collinear. (This means that they do not belong, as proper points of $\PP^2$ or infinitely near points, to the same line.)
\end{Rmk}

\begin{Lem}\label{lem rel dJ 2}
Let $f,h\in\mathcal{J}_*$ be of degree $2$, $g\in\Aut_{\R}(\PP^2)$ and $\Lambda$ be a real linear system of degree $D$. Suppose that
\[\deg(h^{-1}(\Lambda))\leq D \ (\text{resp.}<D),\qquad\deg(fg(\Lambda))<D\]
Then there exist $\theta_1\in\Aut_\R(\PP^2)$ and $\theta_2,\dots,\theta_n\in\Aut_{\R}(\PP^2)\cup\mathcal{J}_*\cup\mathcal{J}_{\circ}$ such that
\begin{enumerate}
\item $w(f)w(g)w(h)=w(\theta_n)\cdots w(\theta_1)$ holds in $\mathcal{G}$, i.e. the following commutative diagram corresponds to a relation in $\mathcal{G}$:
\[\small\xymatrix{&\Lambda\ar[r]^g &g(\Lambda)\ar@{-->}[dr]^f&\\ 
h^{-1}(\Lambda)\ar@{-->}[ru]^h\ar@{-->}[r]^{\theta_1}&\ar@{..>}[r]&\ar@{-->}[r]^{\theta_n}&fg(\Lambda) }\]
\item 
$\deg(\theta_i\cdots\theta_2\theta_1h^{-1}(\Lambda))<D,\quad i=2,\dots,n.$
\end{enumerate}
\end{Lem}
\begin{proof}
If $g\in\mathcal{J}_*$ then $w(f)w(g)w(h)=w(fgh)$ in $\mathcal{J}_*$. So, lets assume that $g\notin\mathcal{J}_*$. 
Let $p:=[1:0:0]$, and let $p,r_1,r_2$ be the base-points of $h^{-1}$, and $p,s_1,s_2$ the ones of $f$. The assumptions $\deg(h^{-1}(\Lambda))\leq D$ and $\deg(fg(\Lambda))<D$ imply 
\begin{equation}\tag{$\bigstar$}\label{eqnzero}m_{\Lambda}(p)+m_{\Lambda}(r_1)+m_{\Lambda}(r_2)\geq D,\quad m_{g(\Lambda)}(p)+m_{g(\Lambda)}(s_1)+m_{g(\Lambda)}(s_2)>D\end{equation}
Note that $r_1,r_2$ (resp. $s_1,s_2$) are both real or a pair of non-real conjugate points. We can assume that $m_{\Lambda}(r_1)\geq m_{\Lambda}(r_2)$, $m_{g(\Lambda)}(s_1)\geq m_{g(\Lambda)}(s_2)$ and that $r_1$ (resp. $s_1$) is a proper point of $\PP^2$ or in the first neighbourhood of $p$ (resp. $p$) and that $r_2$ (resp. $s_2$) is a proper point of $\PP^2$ or in the first neighbourhood of $p$ (resp. $p$) or $r_1$ (resp. $s_1$).

Note that if $\deg(h^{-1}(\Lambda))<D$, then by Lemma~\ref{lem deg mult ineq} "$>$" holds in all inequalities. 

- {\bf(1)} - If $h^{-1}$ and $fg$ have three common base-points, the map $fgh$ is linear and $w(f)w(g)w(h)=w(fgh)$ by Definition~\ref{def G}~\ref{relrel2}. The claim follows with $\theta_1=fgh$.
\smallskip

- {\bf(2)} - If $h^{-1}$ and $fg$ have exactly two (resp. one) common base-points, the map $fgh$ is of degree 2 (resp. 3) and there exists $\alpha_1,\alpha_2\in\Aut_\R(\PP^2)$, $\tau\in\mathcal{J}_*$ of degree $2$ (resp. $3$), such that $fgh=\alpha_2\tau\alpha_1$. Then by Definition~\ref{def G}~\ref{relrel3}
\[w(f)w(g)w(h)=w(\alpha_2)w(\tau)w(\alpha_1).\]
The claim follows with $\theta_1:=\alpha_1$, $\theta_2:=\tau$, $\theta_3=\alpha_2$.
\smallskip

- {\bf(3)} - Suppose that $h^{-1}$ and $fg$ have no common base-points. To construct $\theta_1,\dots,\theta_n$, we need to look at three cases depending on the $r_i$'s and $s_i$'s being real or non-real points. 
\smallskip

- {\bf(3.1)} -
Suppose that $r_2=\bar{r}_1$ and $s_2=\bar{s}_1$. Since they are base-points of $h^{-1},f$, they are all proper points of $\PP^2$. The following two constructions are summarised in the diagrams below.

$\bullet$ If $m_{\Lambda}(p)\geq m_{\Lambda}(g^{-1}(p))$, then Inequalities~(\ref{eqnzero}) imply
\[m_{\Lambda}(p)+2m_{\Lambda}(g^{-1}(s_1))\geq m_{\Lambda}(g^{-1}(p))+2m_{\Lambda}(g^{-1}(s_1)) >D.\]
In particular, the points $p,g^{-1}(s_1),g^{-1}(\bar{s}_1)$ are not collinear an there exists $\tau\in\J_*$ of degree $2$ with these points its base-points. Then
\[\deg(\tau(\Lambda))=2D-m_{\Lambda}(p)-2m_{\Lambda}(g^{-1}(s_1))<D.\]
We put $\theta_1:=\mathrm{Id}$ and $\theta_2:=\tau h\in\J_*$. The maps $\tau,fg$ have two common base-points, and proceed with as in {\bf(2)} with ``$<$" to get $\theta_3,\dots,\theta_n$. 

$\bullet$ If $m_{\Lambda}(p)<m_{\Lambda}(g^{-1}(p))$, then Inequalities~(\ref{eqnzero}) imply
\[m_{\Lambda}(g^{-1}(p))+2m_{\Lambda}(r_1)>m_{\Lambda}(p)+2m_{\Lambda}(r_1)\geq D.\]
In particular, the points $p,g(r_1),g(\bar{r}_1)$ are not collinear, and there exists $\tau'\in\J_*$ of degree $2$ with them as base-points. As above, we have $\deg(\tau'g(\Lambda))<D$.
We put $\theta_n:=f(\tau')^{-1}\in\J_*$. The maps $h^{-1},\tau'g$ have two common base-points and we proceed as in {\bf(2)} to get $\theta_1,\dots,\theta_{n-1}$.
{\footnotesize
\[\xymatrix{ \Lambda\ar@{-->}[d]^{h^{-1}}_(.3){[p,r_1,r_2]} \ar[rr]^g&&g(\Lambda)\ar@/^.5pc/@{-->}[dr]^{f}\ar@{-->}[d]_(.55){\tau'}_(.35){[p,g(r_1),g(r_2)]}\\
h^{-1}(\Lambda)&&\tau'g(\Lambda)\ar@{-->}[r]^(.4){\theta_n}&fg(\Lambda),} 
\ 
\xymatrix{ &\Lambda\ar@{-->}[dl]_{h^{-1}} \ar@{-->}[d]^(.6){\tau}^(.3){[p,g^{-1}(s_1),g^{-1}(s_2)]} \ar[rr]^g  &&g(\Lambda)\ar@{-->}[d]_{f}^(.3){[p,s_1,s_2]}\\
h^{-1}(\Lambda)\ar@{-->}[r]^{\theta_1}&\tau(\Lambda)&&fg(\Lambda)}  
 \]
 }
\smallskip

- {\bf(3.2)} - Assume that $r_2=\bar{r}_1$ and $s_1,s_2$ are real points. (If $r_1,r_2$ are real points and $s_2=\bar{s}_1$, we proceed analogously.) As $p,r_1,\bar{r}_2$ are the base-points of $h^{-1}$, the points $r_1,\bar{r}_2$ are proper points of $\PP^2$. The following constructions are summarised in the diagram below if they are not already pictured in the ones above.

$\bullet$ If $m_{g(\Lambda)}(p)> m_{g(\Lambda)}(g(p))$, then 
\[m_{g(\Lambda)}(p)+2m_{g(\Lambda)}(g(r_1))>m_{g(\Lambda)}(g(p))+2m_{g(\Lambda)}(g(r_1)) \stackrel{(\bigstar)}\geq D\]
hence $p,g(r_1),g(\bar{r}_1)$ are not collinear. We proceed as in {\bf(3.1)}.

$\bullet$ Suppose that $m_{g(\Lambda)}(g(p))\geq m_{g(\Lambda)}(p)$. 

If $s_1$ or $s_2$ are infinitely near to $p$, then $m_{g(\Lambda)}(s_1)\geq m_{g(\Lambda)}(s_2)$ implies $m_{g(\Lambda)}(p)\geq m(s_2)$, and
\[m_{g(\Lambda)}(g(p))+m_{g(\Lambda)}(p)+m_{g(\Lambda)}(s_1)\geq m_{g(\Lambda)}(p)+m_{g(\Lambda)}(s_2)+m_{g(\Lambda)}(s_1)\stackrel{(\bigstar)}>D.\]
In particular, $p,g(p),s_1$ are not collinear. Since $s_1$ is in the first neighbourhood of $p$ or a proper point of $\PP^2$ by ordering, there exists $\tau\in\mathcal{J}_*$ of degree 2 with base-points $p,g(p),s_1$. The above inequality implies $\deg(\tau g(\Lambda))<D$. We put $\theta_n:=f\tau^{-1}\in\mathcal{J}_*$ and apply {\bf(2)} to $h^{-1},g,\tau$ to construct $\theta_1,\dots,\theta_{n-1}$. 

If $s_1,s_2$ are proper points of $\PP^2$, then
\[m_{g(\Lambda)}(g(p))+m_{g(\Lambda)}(s_2)+m_{g(\Lambda)}(s_1)\geq m_{g(\Lambda)}(p)+m_{g(\Lambda)}(s_2)+m_{g(\Lambda)}(s_1)\stackrel{(\bigstar)}>D\]
In particular, $g(p),s_1,s_2$ are not collinear and there exists $\tau'\in\mathcal{J}_*$ with base-points $p,g^{-1}(s_1),g^{-1}(s_2)$. The above inequality implies $\deg(\tau'(\Lambda))<D$. We put $\theta_1:=\Id$, $\theta_2:=\tau'h$ and apply {\bf(2)} to $\tau',g,f$ to construct $\theta_3,\dots,\theta_n$.
{\footnotesize
\[\xymatrix{ \Lambda\ar@{-->}[d]^{h^{-1}}_(.3){[p,r_1,r_2]} \ar[rr]^g&&g(\Lambda)\ar@/^.5pc/@{-->}[dr]^{f}\ar@{-->}[d]^{\tau}_(.35){[p,g(p),s_1]}\\
h^{-1}(\Lambda)&&\tau g(\Lambda)\ar@{-->}[r]^(.4){\theta_n}&fg(\Lambda),} 
\ 
\xymatrix{ &\Lambda\ar@{-->}[dl]_{h^{-1}} \ar@{-->}[d]_{\tau'}^(.3){[p,g^{-1}(s_1),g^{-1}(s_2)]} \ar[rr]^g  &&g(\Lambda)\ar@{-->}[d]_{f}^(.3){[p,s_1,s_2]}\\
h^{-1}(\Lambda)\ar@{-->}[r]^{\theta_2\theta_1}&\tau'(\Lambda)&&fg(\Lambda)}  
 \]
 }
\smallskip

- {\bf(3.3)} - If $r_1,r_2,s_1,s_2$ are real points, let $\{a_1,a_2,a_3\}=\{p,r_1,r_2\}$ and $\{b_1,b_2,b_3\}=\{g^{-1}(p),g^{-1}(s_1),g^{-1}(s_2)\}$ such that $m_{\Lambda}(a_i)\leq m_{\Lambda}(a_{i+1})$ and $m_{\Lambda}(b_i)\leq m_{\Lambda}(b_{i+1})$, $i=1,2$, and if $a_i$ (resp. $b_i$) is infinitely near $a_j$ (resp. $b_j$) then $j>i$. In particular, $a_1,b_1$ are proper points of $\PP^2$. From inequalities~(\ref{eqnzero}), we obtain
\[m_{\Lambda}(a_1)+m_{\Lambda}(a_2)+m_{\Lambda}(b_1)>D,\quad m_{\Lambda}(a_1)+m_{\Lambda}(b_1)+m_{\Lambda}(b_2)>D.\] 
By them and the chosen ordering, there exists $\tau_1,\tau_2\in\mathcal{J}_*$ of degree 2, $\alpha_1,\alpha_2\in\Aut_{\R}(\PP^2)$ such that $\tau_1\alpha_1$, $\tau_2\alpha_2$ have base-points $a_1,a_2,b_1$ and $a_1,b_1,b_2$ respectively. 
The situation is summarised in the following diagram.
{\footnotesize
\[\xymatrix{ &\Lambda\ar@/^1pc/[rr]^g\ar@{-->}[d]_{\tau_1\alpha_1}^(.5){[a_1,a_2,b_1]}\ar@{-->}[dl]_{\tau_1}_(.25){[a_1,a_2,a_3]}\ar@{-->}[rrd]^{\tau_2\alpha_2}^(.3){[a_1,b_1,b_2]}&\qquad&g(\Lambda)\ar@{-->}[rd]^{\tau_2}^(.3){[b_1,b_2,b_3]}&\\
\tau_1(\Lambda)&\tau_1\alpha_1(\Lambda)&&\tau_2\alpha_2g(\Lambda)&\tau_2g(\Lambda)}\] 
}
By construction of $\tau_1,\tau_2$, we have
\begin{align*}&\deg(\tau_1\alpha_1(\Lambda))=2D-m_{\Lambda}(a_1)-m_{\Lambda}(a_2)-m_{\Lambda}(b_1)<D,\\
&\deg(\tau_2\alpha_2(\Lambda))=2D-m_{\Lambda}(a_1)-m_{\Lambda}(b_1)-m_{\Lambda}(b_2)<D.
\end{align*}
The maps $h^{-1},\tau_1\alpha_1$, the maps $\tau_1\alpha_1,\tau_2\alpha_2$ and the maps $\tau_2\alpha_2,f$ have two common base-points, respectively. We proceed with each pair as in case {\bf(2)}, with ``$<$" in the latter two cases, to construct $\theta_1,\dots,\theta_n$. 
\end{proof}

\begin{Lem}\label{lem rel dJ}
Let $f,h\in\mathcal{J}_*$, $g\in\Aut_{\R}(\PP^2)$ and $\Lambda$ be a real linear system of degree $D$. Suppose that
\[\deg(h^{-1}(\Lambda))\leq D \ (\text{resp.}<D),\qquad\deg(fg(\Lambda))<D\]
Then there exist $\theta_1\in\mathcal{J}_*$, $\theta_2\in\Aut_\R(\PP^2)$ and $\theta_3,\dots,\theta_n\in\Aut_{\R}(\PP^2)\cup\mathcal{J}_*\cup\mathcal{J}_{\circ}$ such that
\begin{enumerate}
\item $w(f)w(g)w(h)=w(\theta_n)\cdots w(\theta_1)$ holds in $\mathcal{G}$, i.e. the following commutative diagram corresponds to a relation in $\mathcal{G}$:
\[\small\xymatrix{&\Lambda\ar[r]^g &g(\Lambda)\ar@{-->}[dr]^f&\\ 
h^{-1}(\Lambda)\ar@{-->}[ru]^h\ar@{-->}[r]^{\theta_1}&\ar@{..>}[r]&\ar@{-->}[r]^{\theta_n}&fg(\Lambda) }\]
\item $\deg(\theta_1)=\deg(h)-1,\quad \deg(\theta_1(\Lambda))=\deg(\theta_2\theta_1(\Lambda))\leq D\ (\text{resp.} <D)$ \ and
\begin{align*}
&deg(\theta_i\cdots\theta_1h^{-1}(\Lambda))<D,\quad i=3,\dots,n.
\end{align*}
\end{enumerate}
\end{Lem}

\begin{proof}
If $g\in\mathcal{J}_*$ then $w(f)w(g)w(h)=w(fgh)$ in $\mathcal{J}_*$. So, lets assume that $g\notin\mathcal{J}_*$. 
Let $p:=[1:0:0]$. 
By Lemma~\ref{lem deg mult ineq} there exists $r_1,r_2$ base-points of $h^{-1}$ and $s_1,s_2$ base-points of $f$ such that 
\begin{equation}\tag{$\bigstar$}\label{eqnzero}m_{\Lambda}(p)+m_{\Lambda}(r_1)+m_{\Lambda}(r_2)\geq D,\quad m_{g(\Lambda)}(p)+m_{g(\Lambda)}(s_1)+m_{g(\Lambda)}(s_2)>D\end{equation}
and either $r_1,r_2$ (resp. $s_1,s_2$) are both real or a pair of non-real conjugate points. We can assume that $m_{\Lambda}(r_1)\geq m_{\Lambda}(r_2)$, $m_{g(\Lambda)}(s_1)\geq m_{g(\Lambda)}(s_2)$ and that $r_1$ (resp. $s_1$) is a proper point of $\PP^2$ or in the first neighbourhood of $p$ (resp. $p$) and that $r_2$ (resp. $s_2$) is a proper point of $\PP^2$ or in the first neighbourhood of $p$ (resp. $p$) or $r_1$ (resp. $s_1$).

Note that if $\deg(h^{-1}(\Lambda))<D$, then by Lemma~\ref{lem deg mult ineq} "$>$" holds in all inequalities. 
\smallskip

{\bf - Situation 1 -} {\it Assume that there exist $\tau_1,\tau_2\in\mathcal{J}_*$ of degree 2 with base-points $p,r_1,r_2$ and $p,s_1,s_2$ respectively.}
 
 Observe that $\tau_1h\in\mathcal{J}_*$ and $f(\tau_2)^{-1}\in\mathcal{J}_*$, and $\deg(\tau_1h)=\deg(h)-1$. 
From Inequality~(\ref{eqnzero}) we get
\begin{align*}&\deg(\tau_1(\Lambda))=2D-m_{\Lambda}(p)-m_{\Lambda}(r_1)-m_{\Lambda}(r_2)\leq D,\\
&deg(\tau_2g(\Lambda))=2D-m_{g(\Lambda)}(p)-m_{g(\Lambda)}(s_1)-m_{g(\Lambda)}(s_2)<D.
\end{align*}
We put $\theta_1:=\tau_1h$ and $\theta_n:=f(\tau_2)^{-1}$. The claim now follows from Lemma~\ref{lem rel dJ 2} for $\tau_1^{-1}$, $g$, $\tau_2$.
\smallskip

{\bf - Situation 2 -} {\it Assume that there exists no $\tau_1\in\mathcal{J}_*$ or no $\tau_2\in\mathcal{J}_*$ of degree 2 with base-points $p,r_1,r_2$ and $p,s_1,s_2$, respectively.} 

We claim that one of the two exists. Assume that neither $\tau_1$ nor $\tau_2$ exist. Since $p,r_1,r_2$ are not collinear by Lemma~\ref{lem notcoll}, it follows that $r_1,r_2$ are both in the first neighourhood $p$ \cite[\S2]{AC02}. Then $m_{\Lambda}(p)\geq m_{\Lambda}(r_1)+m_{\Lambda}(r_2)$. Inequality~(\ref{eqnzero}) implies 
\[2m_{\Lambda}(p)\geq m_{\Lambda}(p)+m_{\Lambda}(r_1)+m_{\Lambda}(r_2)\geq D.\] 
Similarly, we get $2m_{g(\Lambda)}(p)>D$. Then $m_{\Lambda}(p)+m_{\Lambda}(g^{-1}(p))>D$, which contradicts B\'ezout theorem. So, $\tau_1$ exists or $\tau_2$ exists.

- {\bf(1)} - Assume that $\tau_1$ exists and $\tau_2$ does not. We will construct an alternative for $\tau_2$ and then proceed as in Situation 2. The constructions are represented in the commutative diagram below. 
By the same argument in the as previous paragraph we get that $s_1,s_2$ are both proximate to $p$ and 
\begin{equation}\label{2star}\tag{$\bigstar\bigstar$}m_{g(\Lambda)}(p)>\nicefrac{D}{2}.\end{equation}

$\bullet$ If $r_1,r_2$ are real points, let $\{t_1,t_2,t_3\}=\{p,r_1,r_2\}$ such that $m_{\Lambda}(t_i)\geq m_{\Lambda}(t_{i+1})$ and such that if $t_i$ is infinitely near $t_j$ then $i>j$. In particular, $t_1$ is a proper point of $\PP^2$. By the chosen ordering, (\ref{eqnzero}) and (\ref{2star}) we have
\[m_{g(\Lambda)}(g(t_1))+m_{g(\Lambda)}(g(t_2))+m_{g(\Lambda)}(p)> \nicefrac{2D}{3}+\nicefrac{D}{2}=D.\]
So, the the points $p,g(t_1),g(t_2)$ are not collinear. Moreover, $t_1\in\PP^2$ and $t_2$ is either a proper point of $\PP^2$ as well or is infinitely near $t_1$. So,  
there exist $\tau_3\in\mathcal{J}_*$ of degree $2$ with base-points $p,g(t_1),g(t_2)$. The above inequality implies $\deg(\tau_3g(\Lambda))<D$. We proceed as in Situation $1$ with $\tau_1,g,\tau_3$.

$\bullet$ If $r_2=\bar{r}_1$, then $r_1,\bar{r}_2$ are proper points of $\PP^2$ because they are base-points of $\tau_1$. B\'ezout theorem and (\ref{2star}) imply 
\[D\geq m_{g(\Lambda)}(p)+m_{g(\Lambda)}(g(p))>\nicefrac{D}{2}+m_{g(\Lambda)}(g(p)).\]
So $m_{g(\Lambda)}(p)>\frac{D}{2}> m_{g(\Lambda)}(g(p))$. 
This and Inequalities~(\ref{eqnzero}) imply 
\[m_{g(\Lambda)}(p)+2m_{g(\Lambda)}(g(r_1))> m_{g(\Lambda)}(g(p))+2m_{g(\Lambda)}(g(r_1))\stackrel{(\bigstar)}\geq D.\]
Thus there exists $\tau_4\in\mathcal{J}_*$ of degree $2$ with base-points $p,g(r_1),g(\bar{r}_2)$. The above inequality implies that $\deg(\tau_4g(\Lambda))<D$. We define $\theta_n:=f\tau_4\in\mathcal{J}_*$ and proceed as in Situation 2 with $\tau_1,g,\tau_4$. 
\[\xymatrix{ &\Lambda\ar[rr]^g\ar@{-->}[d]^{\tau_1}^(.3){[p,r_1,r_2]}&\quad&g(\Lambda)\ar@/^.5pc/@{-->}[rd]^f\ar@{-->}[d]^(.4){\tau_3/}^(.6){\tau_4}_(.7){[p,g(r_1),g(\bar{r}_2)]}_(.4){[p,g(t_1),g(t_2)]/}&\\
h^{-1}(\Lambda)\ar@{-->}[ru]^h&\tau_1(\Lambda)&&\tau_ig(\Lambda)&fg(\Lambda)}\]

- {\bf(2)} - Suppose that $\tau_1$ does not exist and $\tau_2$ does. We proceed analogously as above; we similarly obtain key inequalities $m_{\Lambda}(p)\geq\frac{D}{2}$ and $m_{\Lambda}(p)\geq m_{g(\Lambda)}(g^{-1}(p))$ and the above strict inequalitis with $r_i$ replaced by $s_i$. This enables us to construct an alternative for $\tau_1$ just as above.
\end{proof}

\begin{Lem}\label{lem rel mixed}
Let $g\in\Aut_{\R}(\PP^2)$ and either $f\in\mathcal{J}_{\circ}$ be a standard or special quintic transformation and $h\in\mathcal{J}_*$. 
Let $\Lambda$ be a real linear system of degree $D$. Suppose that
\[\deg(h^{-1}(\Lambda))\leq D\ (\text{resp.}<D),\quad\deg(fg(\Lambda))<D\]
Then there exist $\theta_1\in\mathcal{J}_*$, $\theta_2\in\Aut_\R(\PP^2)$, $\theta_3,\dots,\theta_n\in\Aut_{\R}(\PP^2)\cup\mathcal{J}_*\cup\mathcal{J}_{\circ}$ such that
\begin{enumerate}
\item $w(f)w(g)w(h)=w(\theta_n)\cdots w(\theta_1)$ holds in $\mathcal{G}$, i.e. the following commutative diagram corresponds to a relation in $\mathcal{G}$:
\[\small\xymatrix{&\Lambda\ar[r]^g &g(\Lambda)\ar@{-->}[dr]^f&\\ 
h^{-1}(\Lambda)\ar@{-->}[ru]^h\ar@{-->}[r]^{\theta_1}&\ar@{..>}[r]&\ar@{-->}[r]^{\theta_n}&fg(\Lambda) }\]
\item 
$\deg(\theta_1)=\deg(h)-1,\quad \deg(\theta_1(\Lambda))=\deg(\theta_2\theta_1(\Lambda))\leq D\ (\text{resp.} <D)$ \ and
\begin{align*}
&deg(\theta_i\cdots\theta_1h^{-1}(\Lambda))<D,\quad i=3,\dots,n.
\end{align*}
\end{enumerate}
If  $h\in\mathcal{J}_{\circ}$ is a standard or special quintic transformation and $f\in\mathcal{J}_*$, we find $\theta_1\in\Aut_\R(\PP^2)$, $\theta_2,\dots,\theta_n\in\Aut_\R(\PP^2)\cup\mathcal{J}_*\cup\mathcal{J}_{\circ}$ satisfying (1) and  
\[\deg(\theta_i\cdots\theta_2\theta_1h^{-1}(\Lambda))<D,\quad i=2,\dots,n.\]
\end{Lem}

\begin{proof}
We only look at the situation, where $f\in\mathcal{J}_{\circ}$, $h\in\mathcal{J}_*$, because if $f\in\mathcal{J}_*$, $h\in\mathcal{J}_{\circ}$ the constructions are done analogously.

We need to play with the multiplicities of the six base-points of $f$, so we call them $s_1,\bar{s}_1,s_2,\bar{s}_2,s_3,\bar{s}_3$ and order them such that $m_{g(\Lambda)}(s_1)\geq m_{g(\Lambda)}(s_2)\geq m_{g(\Lambda)}(s_3)$ and if $s_i$ is infinitely near $s_j$, then $i>j$. In particular, $s_1$ is a proper point of $\PP^2$. By Lemma~\ref{lem deg mult ineq} we have 
\begin{equation}\tag{$\text{Ineq}^3$}\label{eqn6} \quad m_{g(\Lambda)}(s_1)+m_{g(\Lambda)}(s_2)+m_{g(\Lambda)}(s_3)>D.\end{equation}
Let $p:=[1:0:0]$. By Lemma~\ref{lem deg mult ineq} there exist two real or two non-real conjugate base-points $r_1,r_2$ of $h$, such that 
\begin{equation}\tag{$\text{Ineq}^4$}\label{eqn5}m_{\Lambda}(p)+m_{\Lambda}(r_1)+m_{\Lambda}(r_2)\geq D.\end{equation}
We can assume that $m_{\Lambda}(r_1)\geq m_{\Lambda}(r_2)$ and that $r_1$ is a proper point of $\PP^2$ or in the first neighbourhood of $p$ and $r_2$ is a proper base-point of $\PP^2$ or in the first neighbourhood of $p$ or $r_1$. 
We now look at two cases, depending on whether $r_1,r_2$ are real points or not. If $\deg(h^{-1}(\Lambda))<D$, then "$>$" holds in the above inequalities (Lemma~\ref{lem deg mult ineq}) and we will have "$<$" everywhere.
\smallskip

\underline{Case 1:} Suppose that $r_1,r_2$ are real points. We can pick $t\in\{p,r_1,r_2\}$ a proper point of $\PP^2$ such that $m_{\Lambda}(t)=\max\{m_{\Lambda}(p),m_{\Lambda}(r_1),m_{\Lambda}(r_2)\}$. Inequalities~(\ref{eqn6}) and~(\ref{eqn5}) impy $m_{g(\Lambda)}(s_1),m_{\Lambda}(t)>\frac{D}{3}$, so
\[m_{g(\Lambda)}(g(t))+2m_{g(\Lambda)}(s_1)>D.\]
In particular, $g(t),s_1,\bar{s}_1$ are not collinear, and since $t$ and $s_1$ are proper points of $\PP^2$, there exists $\tau\in\mathcal{J}_*$ of degree $2$, $\alpha\in\Aut_{\R}(\PP^2)$ such that $\tau\alpha$ has base-points $g(t),s_1,\bar{s}_1$. At least one of the points $s_2,s_3$ is not on a line contracted by $\tau\alpha$, say $s_i$, because $s_1,\bar{s}_1,s_2,\bar{s}_2,s_3,\bar{s}_3$ are not all on one conic. There exists $\beta_1\in\Aut_\R(\PP^2)$ that sends $s_1,s_i$ onto $p_1,p_2$. There exists $\beta_2\in\Aut_\R(\PP^2)$ that sends $(\tau\alpha)(s_1),(\tau\alpha)(s_i)$ onto $p_1,p_2$. Then $\beta_2\tau\alpha\beta_1^{-1}\in\J_{\circ}$. There also exists $\beta_3\in\Aut_\R(\PP^2)$ such that $\beta_3f\beta_1^{-1}\in\J_{\circ}$. 
Definition~\ref{def G}~\ref{relrel1}, \ref{relrel2} and Remark~\ref{rmk rel sp}~(\ref{relrel4}) imply
\[w(\beta_2)w(\tau)w(\alpha\beta_1^{-1})=w(\beta_2\tau\alpha\beta_1^{-1}),\quad w(\beta_3)w(f)w(\beta_1^{-1})=w(\beta_3f\beta_1^{-1}),\]
and thus
\begin{align*}w(f)w(\alpha^{-1})w(\tau^{-1})&=w(\beta_3^{-1})w(\beta_3f\beta_1^{-1})w(\beta_1)w(\alpha^{-1})w(\alpha\beta_1^{-1})w(\beta_1\alpha^{-1}\tau^{-1}\beta_2^{-1})w(\beta_2)\\
&\underset{\beta_3f\beta_1^{-1}\in\mathcal{J}_{\circ}}{\stackrel{\beta_2\tau\alpha\beta_1^{-1}\in\mathcal{J}_{\circ}}=}w(\beta_3^{-1})w(\beta_3f\alpha^{-1}\tau^{-1}\beta_2^{-1})w(\beta_2).
\end{align*}
According to the above inequality, we have
\[\deg(\tau\alpha g(\Lambda))=2D-m_{g(\Lambda)}(g(t))-2m_{g(\Lambda)}(s_1)<D.\]
We define $\theta_{n-2}:=\beta_2^{-1}$, $\theta_{n-1}:=\beta_3f\alpha^{-1}\tau^{-1}\beta_2\in\mathcal{J}_{\circ}$ and 
$\theta_n:=\beta_3^{-1}$. The maps $h,g\alpha^{-1},\tau$ satisfy the assumptions of Lemma~\ref{lem rel dJ} and find $\theta_1,\dots,\theta_{n-3}$ accordingly. 
The construction is visualised in the following commutative diagram.
{\small
\[\xymatrix{\Lambda\ar[rrrr]^g&&&& g(\Lambda)\ar@{-->}[dlll]_(.6){\tau\alpha}_(.45){[g(t),s_1,\bar{s}_1]} \ar@{-->}[dr]^f \ar[dl]_{\beta_1}\\
h^{-1}(\Lambda)\ar@{-->}[u]^h&\tau\alpha g(\Lambda)\ar[r]_{\beta_2=\theta_{n-2}}&\beta_2^{-1}\tau\alpha g(\Lambda)\ar@{-->}[r]^(.55){\beta_1(\tau\alpha)^{-1}\beta_2^{-1}}\ar@/_1pc/@{-->}[rr]_{\theta_{n-1}}&\beta_1g(\Lambda)\ar@{-->}[r]^{\beta_3f\beta_1^{-1}}&\beta_3fg(\Lambda)&fg(\Lambda)\ar[l]^{\beta_3=\theta_n^{-1}}}\]
}

\underline{Case 2:} Suppose that $r_2=\bar{r}_1$. 

{\bf(1)} If $g(r_1),g(\bar{r}_1)\in\{s_1,\bar{s}_1\dots,s_3,\bar{s}_3\}$, then $r_1,\bar{r}_1$ cannot be infinitely near $p$, so they are proper points of $\PP^2$. Furthermore, $p,r_1,\bar{r}_1$ are not collinear (Remark~\ref{lem notcoll}), so there exists $\tau\in\mathcal{J}_*$ of degree 2 with base-points $p,r_1,\bar{r}_1$. At least one point in $\{s_1,s_2,s_3\}\setminus\{g(r_1),g(\bar{r}_1)\}$ is not on a line contracted by $\tau g^{-1}$, say $s_1$, because $s_1,\dots,\bar{s}_3$ are not all on one conic. 
By the same argument as in Case 1, there exist $\beta_1,\beta_2,\beta_3\in\mathcal{J}_{\circ}$ such that $\beta_2\tau g^{-1}\beta_1^{-1}\in\mathcal{J}_{\circ}$ and $\beta_3f\beta_1^{-1}\in\mathcal{J}_{\circ}$. 
Just like in the Case 1, we obtain with \ref{relrel1}, \ref{relrel2} and Remark~\ref{rmk rel sp}~(\ref{relrel4})
\[w(f)w(g)w(\tau)=w(\beta_3^{-1})w(\beta_3fg\tau^{-1}\beta_2^{-1})w(\beta_2).\]
Note that $\tau h\in\mathcal{J}_*$ and $\deg(\tau h)=\deg(h)-1$. Furthermore, the above inequality implies that 
\[\deg(\tau(\Lambda))=2D-m_{\Lambda}(p)-2m_{\Lambda}(r_1)<D.\]
We define $\theta_1:=\tau h$, $\theta_2:=\beta_2$, $\theta_3:=\beta_3fg\tau^{-1}(\beta_2)^{-1}\in\mathcal{J}_{\circ}$ and $\theta_4=(\beta_3)^{-1}$. 
The construction is summarised in the following commutative diagram:
{\small
\[\xymatrix{&\Lambda\ar[rrrr]^g \ar@{-->}[dr]^(.6){\tau}^(.45){[p,r_1,\bar{r}_1]}&&&& g(\Lambda) \ar@{-->}[dr]^f \ar[dl]_{\beta_1}\\
h^{-1}(\Lambda)\ar@{-->}[ur]^h\ar@{-->}[rr]^{\theta_1}_{\deg(\theta_1)=\deg(h)-1}&&\tau(\Lambda)\ar[r]_{\beta_2=\theta_2}&\beta_2^{-1}\tau g(\Lambda)\ar@{-->}[r]^(.55){\beta_1g(\beta_2\tau)^{-1}}\ar@/_1pc/@{-->}[rr]_{\theta_3}&\beta_1g(\Lambda)\ar@{-->}[r]^{\beta_3f\beta_1^{-1}}&\beta_3fg(\Lambda)&fg(\Lambda)\ar[l]^{\beta_3=(\theta_4)^{-1}}}\]
}
\indent{\bf(2)} Assume that $g(r_1),g(\bar{r}_1)\notin\{s_1,\bar{s}_1\dots,s_3,\bar{s}_3\}$. 
\smallskip

$\bullet$ If $m_{\Lambda}(p)\geq m_{\Lambda}(r_1)$, then $m_{\Lambda}(p)=\max\{m_{\Lambda}(p),m_{\Lambda}(r_1),m_{\Lambda}(\bar{r}_2)\}$ and we proceed as in Case 1 with $p$ instead of $t$.
\smallskip

$\bullet$ If $m_{\Lambda}(p)<m_{\Lambda}(r_1)$, then $r_1,\bar{r}_1$ are proper points of $\PP^2$ and there exists $\tau\in\mathcal{J}_*$ of degree $2$ with base-points $p,r_1,\bar{r}_1$. As in {\bf(1)}, we have
\[\deg(\tau(\Lambda))\leq D,\quad\deg(\tau h)=\deg(h)-1,\]
and we put $\theta_1:=\tau h\in\mathcal{J}_*$. Inequality~(\ref{eqn6}) and $m_{\Lambda}(p)<m_{\Lambda}(r_1)$ imply $m_{g(\Lambda)}(g(r_1))\geq\frac{D}{3}$. Inequality~(\ref{eqn5}) and the order of the $s_i$ imply
\[m_{g(\Lambda)}(g(r_1))+m_{g(\Lambda)}(s_1)+m_{g(\Lambda)}(s_2)>D.\]
Since moreover $r_1,s_1\in\PP^2$ and $s_2$ is a proper point of $\PP^2$ or in the first neighbourhood of $s_1$, there exists a standard or special quintic transformation $\theta\in\mathcal{J}_{\circ}$, $\alpha\in\Aut_{\R}(\PP^2)$ such that $\theta\alpha$ has base-points $g(r_1),g(\bar{r}_1),s_1,\bar{s}_1,s_2,\bar{s}_2$. Then 
\[\deg(\theta\alpha g(\Lambda))=5D-4m_{g(\Lambda)}(g(r_1))-4m_{g(\Lambda)}(s_1)-4m_{g(\Lambda)}(s_2)<D.\] 
The construction is visualised in the following diagram.
\[\small\xymatrix{ &\Lambda\ar[rr]^g\ar@{-->}[d]_{\tau}^(.4){[p,r_1,\bar{r}_1]}&&g(\Lambda)\ar@{-->}[d]^{\theta\alpha}_(.4){[r_1,s_1,s_2]}\ar@{-->}[rd]^f^(.3){[s_1,s_2,s_3]}&\\
h^{-1}(\Lambda)\ar@{-->}[r]_{\deg(h)-1}^{\theta_1}\ar@{-->}[ru]^h&\tau(\Lambda)&&\theta\alpha g(\Lambda)&fg(\Lambda)}\]
The maps $\tau,\alpha g,\theta$ are satisfy the conditions of {\bf(1)}, and $\theta,\alpha,f$ satisfy the assumptions of Lemma~\ref{lem rel 5,5} with ``$<$". We get $\theta_2,\dots,\theta_n$ from them.
\end{proof}

\begin{Lem}\label{lem exist}
Let $f\in\J_{\circ}$, $\Lambda$ and $q$ as in Lemma~\ref{lem deg mult ineq}~(2)~(2.1) with $q\in\PP^2$.
Then one of the following holds: 
\begin{enumerate}
\item there exists $\theta\in\J_{\circ}$ of degree $3$ with base-points $p_1,\bar{p}_1,p_2,\bar{p}_2,q$, 
\item $\deg(f)$ is even and there exists $\theta\in\J_{\circ}$ of degree $2$ with base-points $p_i,\bar{p}_i,q$ and \mbox{$2m(p_i)+m(q)\geq D$}, where $m_f(p_i)=\frac{\deg(f)}{2}$.
\end{enumerate}
\end{Lem}
\begin{proof}   
If $f$ has odd degree, its characteristic (Lemma~\ref{lem char}) implies that $q$ is not collinear with any two of $p_1,\dots,\bar{p}_2$, and (1) follows from Lemma~\ref{lem sigma deg 3}. If $f$ has even degree, let $p_i,\bar{p}_i$ its base-points of multiplicity $\frac{\deg(f)}{2}$. Then $q,p_i,\bar{p}_i$ are not collinear by Lemma~\ref{lem char}. By Lemma~\ref{lem sigma} there exists $\theta\in\J_{\circ}$ of degree $2$ with base-points $q,p_i,\bar{p}_i$. Suppose that (2) does not hold, i.e. \mbox{$2m_{\Lambda}(p_i)+m_{\Lambda}(q)<D$}. By assumption, we have $2(m_{\Lambda}(p_1)+m_{\Lambda}(p_2)+m_{\Lambda}(q))\geq 2D$, and it implies $2m_{\Lambda}(p_j)+m_{\Lambda}(q)>D$ for $j\neq i$. Thus the points $p_j,\bar{p}_j,q$ are not collinear as well. The characteristic of $f$ implies that $q$ is not collinear with $p_1,p_2$ and $p_1,\bar{p}_2$. Summarised, $q$ is not collinear with any two of $p_1,\dots,\bar{p}_2$, and now Lemma~\ref{lem sigma deg 3} implies (1).
\end{proof}

\begin{proof}[Proof of Theorem~\ref{prop technical thm}]
We prove the following: If $f_1,\dots,f_m\in\Aut_{\R}(\PP^2)\cup\mathcal{J}_*\cup\mathcal{J}_{\circ}$ such that 
\[f_m\cdots f_1=\Id \quad\text{in}\ \Bir_{\R}(\PP^2),\]
then 
\[w(f_m)\cdots w(f_1)=1\quad\text{in}\ \mathcal{G}.\]
It then follows that the natural surjective homomorphism $\mathcal{G}\rightarrow\Bir_{\R}(\PP^2)$ is an isomorphism.\par
Let $\Lambda_0$ be the linear system of lines in $\PP^2$, and define
\[\Lambda_i:=(f_i\cdots f_1)(\Lambda_0)\]
It is the linear system of the map $(f_i\cdots f_1)^{-1}$ and of degree $d_i:=\deg(f_i\cdots f_1)$. Define
\[ D:=\max\{d_i\mid i=1,\dots,m\},\ n:=\max\{i\mid d_i=D\},\ k:=\sum_{i=1}^n(\deg(f_i)-1)\]
We use induction on the lexicographically ordered pair $(D,k)$.

If $D=1$, then $f_1,\dots,f_m$ are linear maps, and thus $w(f_m)\cdots w(f_1)=1$ holds in $\Aut_{\R}(\PP^2)$ and hence in $\mathcal{G}$. 

So, lets assume that $D>1$. By definition of $n$, we have $\deg(f_{n+1})\geq 2$. We may assume that $f_n$ is a linear map - else we can insert $\Id$ after $f_n$, because $w(f_m)\cdots w(f_1)=w(f_m)\cdots w(f_{n+1})w(\Id )w(f_n)\cdots w(f_1)$ and this does not change the pair $(D,k)$. 

If $f_{n-1},f_n,f_{n+1}$ are all contained in $\J_{\circ}$ or in $\J_*$, we have $w(f_{n+1})w(f_n)w(f_{n+1})=w(f_{n+1}f_nf_{n-1})$ and this decreases the pair $(D,k)$. 

So, lets assume that not all three are contained in the same group. We now use Lemma~\ref{lem rel 5,5}, \ref{lem rel dJ}, \ref{lem rel mixed}, \ref{lem exist} to find $\theta_1,\dots,\theta_N\in\Aut_{\R}(\PP^2)\cup\mathcal{J}_*\cup\mathcal{J}_{\circ}$ such that 
\begin{equation}\tag{$\ast$}\label{1stern}w(f_{n+1})w(f_n)w(f_{n-1})=w(\theta_N)\cdots w(\theta_1)\end{equation}
and such that the pair $(D',k')$ associated to $f_m\cdots f_{n+2}\theta_N\cdots \theta_1f_{n-2}\cdots f_1$ is strictly smaller than $(D,k)$. 
\smallskip

{\bf(1)} If $f_{n-1},f_{n+1}\in\mathcal{J}_*$, we apply Lemma~\ref{lem rel dJ} to $f_{n-1},f_n,f_{n+1}$; there exist $\theta_1\in\J_*$, $\theta_2\in\Aut_\R(\PP^2)$, $\theta_3,\dots,\theta_N\in\Aut_\R(\PP^2)\cup\J_*\cup\J_{\circ}$ satisfying (\ref{1stern}) such that
\begin{align*}&\deg(\theta_1)=\deg(f_{n-1})-1,\quad\deg(\theta_2\theta_1(\Lambda_{n-2}))=\deg(\theta_1(\Lambda_{n-2}))\leq D,\\
&\deg(\theta_i\cdots\theta_3\theta_2\theta_1(\Lambda_{n-2}))<D,\quad i=3,\dots,N.
\end{align*}
We obtain a new pair $(D',k')$ where $D'<D$ or $D'=D$, $n'\leq n$ and 
\[k'\leq\sum_{i=1}^{n-2}(\deg(f_i)-1)+(\deg(\theta_1)-1)<\sum_{i=1}^{n-2}(\deg(f_i)-1)+(\deg(f_{n-1})-1)=k.\]
\smallskip

{\bf(2)} 
Suppose that $f_{n-1}\in\mathcal{J}_{\circ}$. Denote by $m_i(t)$ the multiplicity of $\Lambda_i$ in $t$. 
By Lemma~\ref{lem deg mult ineq}, there exists a base-point $q$ of $(f_{n-1})^{-1}$ of multiplicity $2$ such that 
\[m_{n-1}(p_1)+m_{n-1}(p_2)+m_{n-1}(q)\geq D,\] 
or $(f_{n-1})^{-1}$ has a simple base-point $r$ and 
\begin{equation}\tag{sbp}\label{eq:2.2} 2m_{n-1}(p_i)+m_{n-1}(r)\geq D,\qquad\text{where}\quad m_f(p_i)=\nicefrac{\deg(f_{n-1})}{2}.\end{equation}
We can assume that $q$ is either a proper point of $\PP^2$ or in the first neighbourhood of one of $p_1,\bar{p}_1,p_2,\bar{p}_2$. 

$\bullet$ First, assume that we are in the case where $q$ exists. 

If $q$ is a non-real point, then the above inequality implies that $p_1,\bar{p}_1,p_2,\bar{p}_2,q,\bar{q}$ are not all on contained in one conic, so there exists a standard or special quintic transformation $g_{n-1}\in\J_{\circ}$ with base-points $p_1,\dots,\bar{p}_2,q,\bar{q}$. 

If $q$ is a real point, then $q\in\PP^2$ because a real point cannot be infinitely near one of $p_1,\dots,\bar{p}_2$. By Lemma~\ref{lem exist} there exists $g_{n-1}\in\mathcal{J}_{\circ}$ of degree $3$ with base-points $p_1,\dots,\bar{p}_2,q$, or the degree of $f_{n-1}$ is even and there exists $g_{n-1}\in\J_{\circ}$ of degree $2$ with base-points $p_i,\bar{p}_i,q$ and $m_{n-1}(q)+2m_{n-1}(p_i)\geq D$ where $m_{(f_{n-1})^{-1}}(p_i)=\frac{\deg(f_{n-1})}{2}$. 

$\bullet$ If $q$ does not exist, $(f_{n-1})^{-1}$ has a simple base-point $r$ that satisfies conditions~(\ref{eq:2.2}).
If $r$ is not a proper point of $\PP^2$, we replace it by the (real) proper point of $\PP^2$ to which $r$ is infinitely near. This does not change the above inequality. In either case, Lemma~\ref{lem sigma} ensures the existence of $g_{n-1}\in\J_{\circ}$ of degree $2$ with base-points $p_i,\bar{p}_i,r$ for $m_{(f_{n-1})^{-1}}(p_i)=\frac{\deg(f_{n-1})}{2}$.

In each of the above cases, the above inequalities imply
\[\deg(g_{n-1}(\Lambda_{n-1}))\leq D.\]
We define $\theta_1:=g_{n-1}f_{n-1}\in\J_{\circ}$. By construction of $g_{n-1}$ we have
\begin{equation}\tag{$\ast\ast$}\label{eq*}
\deg(\theta_1)<\deg(f_{n-1})
\end{equation}

If $f_{n+1}\in\J_{\circ}$, we similarly find the above conditions with ``$>$" and $g_{n+1}\in\J_{\circ}$ standard or special quintic transformation or of degree $2$ or $3$ such that 
\[\deg(g_{n+1}(\Lambda_n))<D,\]
and define $\theta_N:=f_{n+1}(g_{n+1})^{-1}\in\J_{\circ}$ (we do not care about its degree). 

If $f_{n-1},f_{n+1}\in\J_{\circ}$, we apply Lemma~\ref{lem rel 5,5} to $g_{n-1},f_n,g_{n+1}$ to find $\theta_2\in\Aut_\R(\PP^2)$, $\theta_3,\dots,\theta_{N-1}\in\Aut_\R(\PP^2)\cup\J_*\cup\J_{\circ}$ satisfying (\ref{1stern}) and
\[\deg(\theta_i\cdots\theta_3\theta_2 g_{n-1}(\Lambda_{n-1}))<D,\quad i=3,\dots,N-1.\]

If $f_{n-1}\in\J_{\circ}$ and $f_{n+1}\in\J_*$, we apply Lemma~\ref{lem rel mixed} to $g_{n-1},f_n,f_{n+1}$ to find $\theta_2\in\Aut_\R(\PP^2)$, $\theta_3,\dots,\theta_N\in\Aut_\R(\PP^2)\cup\J_*\cup\J_{\circ}$ satisfying (\ref{1stern}) and
\[\deg(\theta_i\cdots\theta_3\theta_2g_{n-1}(\Lambda_{n-1}))<D,\quad i=3,\dots,N.\]

If $f_{n-1}\in\J_*$ and $f_{n+1}\in\J_{\circ}$, we apply Lemma~\ref{lem rel mixed} to $f_{n-1},f_n,g_{n+1}$ to find $\theta_1\in\J_*$, $\theta_2\in\Aut_\R(\PP^2)$ $\theta_3,\dots,\theta_{N-1}\in\Aut_\R(\PP^2)\cup\J_*\cup\J_{\circ}$ satisfying (\ref{1stern}) and
\begin{align*}\tag{$\ast\ast\ast$}&\deg(\theta_1)=\deg(f_{n-1})-1,\quad \deg(\theta_2\theta_1(\Lambda_{n-2}))=\deg(\theta_1(\Lambda_{n-2}))\leq D,\\
& \deg(\theta_i\cdots\theta_2\theta_1(\Lambda_{n-2}))<D,\quad i=3,\dots,N-1.\end{align*}
The constructions are visualised in the following commuting diagrams corresponding to relations in $\mathcal{G}$.
{\tiny
\[\xymatrix{&\Lambda_{n-1}\ar[r]^{f_n}\ar@{-->}[d]^{g_{n-1}}&\Lambda_n\ar@{-->}[rd]^{f_{n+1}\in\J_{\circ}}\ar@{-->}[d]_{g_{n+1}}&\\
\Lambda_{n-2}\ar@{-->}[ur]^{\J_{\circ}\ni f_{n-1}}\ar[r]^{\theta_1}_{\deg(\theta_1)<\deg(f_{n-1})}&\ar@{..>}[r]^{\theta_{N-1}\cdots\theta_2}_{\text{Lem.}~\ref{lem rel 5,5}}&\ar@{-->}[r]^{\theta_N}&\Lambda_{n+1}
}
\xymatrix{&\Lambda_{n-1}\ar[r]^{f_n}_{\text{Lem.}~\ref{lem rel mixed}}\ar@{-->}[d]^{g_{n-1}}&\Lambda_n\ar@{-->}[d]^{f_{n+1}\in\J_*}\\
\Lambda_{n-2}\ar@{-->}[ur]^{\J_{\circ}\ni f_{n-1}}\ar[r]^{\theta_1}_{\deg(\theta_1)<\deg(f_{n-1})}&\ar@{..>}[r]^{\theta_N\cdots\theta_2}_{\text{Lem.}~\ref{lem rel mixed}}&\Lambda_{n+1}
}
\xymatrix{\Lambda_{n-1}\ar[r]^{f_n}&\Lambda_n\ar@{-->}[rd]^{f_{n+1}\in\J_{\circ}}\ar@{-->}[d]_{g_{n+1}}&\\
\Lambda_{n-2}\ar@{-->}[u]^{\J_*\ni f_{n-1}}\ar@{..>}[r]^{\theta_{N-1}\cdots\theta_1}_{\text{Lem.}~\ref{lem rel mixed}}&\ar@{-->}[r]^{\theta_N}&\Lambda_{n+1}
}
\]
}
We claim that in each case $(D',k')<(D,k)$. The properties listed above imply that $D'\leq D$, and if $D=D'$ then $n'\leq n$ and, because $\theta_2\in\Aut_\R(\PP^2)$,
\begin{align*}k'\leq\sum_{i=1}^{n-2}(\deg(f_i)-1)+(\deg(\theta_1)-1)\underset{(\ast\ast\ast)}{\stackrel{(\ast\ast)}<}\sum_{i=1}^{n-2}(\deg(f_i)-1)+(\deg(f_{n-1})-1)=k. 
\end{align*}
\end{proof}

\section{A quotient of $\Bir_{\R}(\PP^2)$}\label{section quotient}

Let $\varphi_0\colon\mathcal{J}_{\circ}\rightarrow\bigoplus_{(0,1]}\Z/2\Z$ be the map given in Definition~\ref{def varphicirc}. By Theorem~\ref{prop technical thm} and Remark~\ref{rmk general amalgam}, the group $\Bir_{\R}(\PP^2)$ is isomorphic to the free product \mbox{$\Aut_{\R}(\PP^2)\ast\mathcal{J}_*\ast\mathcal{J}_{\circ}$} modulo all the pairwise intersections of $\Aut_\R(\PP^2),\mathcal{J}_*,\mathcal{J}_{\circ}$ and the relations \ref{relrel1}, \ref{relrel2}, \ref{relrel3} (see Definition~\ref{def G}). Define the map
\[\Phi\colon\Aut_{\R}(\PP^2)\ast\mathcal{J}_*\ast\mathcal{J}_{\circ}\longrightarrow \bigoplus_{(0,1]} \Z/2\Z,\quad f\mapsto\begin{cases} \varphi_{\circ}(f),&f\in\mathcal{J}_{\circ}\\ 0,&f\in\Aut_{\R}(\PP^2)\cup\mathcal{J}_*\end{cases}\]
It is a surjective homomorphism of groups because $\varphi_{\circ}$ is a surjective homomorphism of groups (Lemma~\ref{lem hom}).
We shall now show that there exists a homomorphism $\varphi$ such that the diagram
\begin{equation}\label{phi}\tag{Diag. $\varphi$}\small\xymatrix{ \Aut_{\R}(\PP^2)\ast\mathcal{J}_*\ast\mathcal{J}_{\circ}\ar[r]^{\pi}\ar[d]^{\Phi} & \mathcal{G}\simeq\Bir_{\R}(\PP^2)\ar@{..>}[dl]^{\quad\exists\ \varphi\ (\text{Prop.}~\ref{thm quotient})} \\
\bigoplus_{(0,1]}\Z/2\Z} \end{equation}
is commutative, where $\pi$ is the quotient map. For this, it suffices to show that $\ker(\pi)\subset\ker(\Phi)$. We will first show that the relations given by the pairwise intersections of $\Aut_\R(\PP^2),\mathcal{J}_*,\mathcal{J}_{\circ}$ are contained in $\ker(\Phi)$ and then it is left to prove that relations \ref{relrel1}, \ref{relrel2}, \ref{relrel3} are contained in it.

\begin{Lem}\label{lem inters}\item
$(1)$ Let $f_1\in\Aut_\R(\PP^2)$, $f_2\in\mathcal{J}_{\circ}$ such that $\pi(f_1)=\pi(f_2)$. Then $\Phi(f_1)=\Phi(f_2)=0$.
$(2)$ Let $f_1\in\mathcal{J}_*$, $f_2\in\mathcal{J}_{\circ}$ such that $\pi(f_1)=\pi(f_2)$. Then $\Phi(f_1)=\Phi(f_2)=0$. 

In particular, $\Phi$ induces a homomorphism from the generalised amalgamated product of $\Aut_\R(\PP^2),\mathcal{J}_*,\mathcal{J}_{\circ}$ along all pairwise intersections onto $\bigoplus_{(0,1]}\Z/2\Z$.
\end{Lem}
\begin{proof}
(1) We have $\pi(f_1)=\pi(f_2)\in\Aut_\R(\PP^2)\cap\mathcal{J}_{\circ}\subset\mathcal{J}_{\circ}$. In particular, $\varphi_{\circ}(\pi(f_i))=0$, $i=1,2$ (Remark~\ref{rmk map}, (\ref{rmk map 1})), and so $\Phi(f_1)=\Phi(f_2)=0$ by definition of $\Phi$.

(2) Lets first figure out what exactly $\mathcal{J}_*\cap\mathcal{J}_{\circ}$ consists of. First of all, it is not empty because the quadratic involution
\[[x:y:z]\dashrightarrow[y^2+z^2:xy:xz]\]
is contained in it. Let $f\in\mathcal{J}_*\cap\mathcal{J}_{\circ}$ be of degree $d$. By Lemma~\ref{lem char}, its characteristic is $(d; (\frac{d-1}{2})^4,2^{\frac{d-1}{2}})$ or $(d;(\frac{d}{2})^2,(\frac{d-2}{2})^2,2^{\frac{d-2}{2}},1)$. Since $f\in\mathcal{J}_*$, it has characteristic $(d;d-1,1^{2d-2})$. If follows that $d\in\{1,2,3\}$. 

Linear and quadratic elements of $\mathcal{J}_{\circ}$ are sent by $\varphi_{\circ}$ onto $0$ (Remark~\ref{rmk map}~(\ref{rmk map 1})). Elements of $\mathcal{J}_{\circ}$ of degree 3 decompose into quadratic elements of $\mathcal{J}_{\circ}$ (Remark~\ref{rmk deg 3}) and are hence sent onto zero by $\varphi_{\circ}$ as well. In particular, $\Phi(f_1)=\Phi(f_2)=0$. 

Since $\Aut_\R(\PP^2),\J_*\subset\ker(\Phi)$, (1) and (2) imply that $\Phi$ induces a homomorphism from the generalised amalgamated product of $\Aut_\R(\PP^2),\mathcal{J}_*,\mathcal{J}_{\circ}$ along all pairwise intersections onto $\bigoplus_{(0,1]}\Z/2\Z$.
\end{proof}

\begin{Rmk}\label{lem linear 2}
Let $\theta\in\mathcal{J}_{\circ}$ be a standard quintic transformation with $S(\theta)=\{(q,\bar{q})\}$, $S(\theta^{-1})=\{(q',\bar{q}')\}$. 
Let $\alpha_q,\alpha_{q'}\in\Aut_{\R}(\PP^2)$ that fix $p_1$ and send $q$ (resp. $q'$) onto $p_2$. Then $\alpha_{q'}\theta(\alpha_q)^{-1}\in\mathcal{J}_{\circ}$ is a standard quintic transformation with base-points $p_1,\dots,\bar{p}_2$, $\alpha_q(p_2),\ \alpha_q(\bar{p}_2)$. Lemma~\ref{lem linear} and the definition of $\varphi_{\circ}$ (Definition~\ref{def map Jcirc} and Definition~\ref{def varphicirc})
imply $\varphi_{\circ}(\theta)=\varphi_{\circ}\left(\alpha_{q'}\theta(\alpha_q)^{-1}\right)$.
\end{Rmk}

\begin{Prop}\label{thm quotient}
The homomorphism $\Phi$ induces a surjective homomorphism of groups
\[\varphi\colon\Bir_{\R}(\PP^2)\longrightarrow\bigoplus_{(0,1]}\Z/2\Z\]
which is given as follows: 
Let $f\in\Bir_{\R}(\PP^2)$ and write $f=f_n\cdots f_1$, where $f_1,\dots,f_n\in\Aut_{\R}(\PP^2)\cup\mathcal{J}_*\cup\mathcal{J}_{\circ}$. Then $\varphi(\Aut_\R(\PP^2)\cup\mathcal{J}_*)=0$ and
\[\varphi(f)=\sum_{i=1}^n\Phi(f_i)=\sum_{f_i\in\mathcal{J}_{\circ}}\varphi_{\circ}(f_i)\]
The kernel $\ker(\varphi)$ contains all elements of degree $\leq 4$. 
\end{Prop}

\begin{proof}
Let $\pi\colon\Aut_{\R}(\PP^2)\ast\mathcal{J}_*\ast\mathcal{J}_{\circ}\rightarrow\mathcal{G}\simeq\Bir_{\R}(\PP^2)$ be the quotient map (Remark~\ref{rmk general amalgam}). We want to show that there exists a homomorphism $\varphi\colon\Bir_\R(\PP^2)\rightarrow\bigoplus_{(0,1]}\Z/2\Z$ such that the diagram (\ref{phi}) 
is commutative. It suffices to show that $\ker(\pi)\subset\ker(\Phi)$. By Lemma~\ref{lem inters}, $\Phi$ induces a homomorphism from the generalised amalgamated product of $\Aut_\R(\PP^2),\mathcal{J}_*,\mathcal{J}_{\circ}$ along all intersections onto $\bigoplus_{(0,1]}\Z/2\Z$. So, by Remark~\ref{rmk general amalgam}, it suffices to show that $\Phi$ sends the relations \ref{relrel1}, \ref{relrel2}, \ref{relrel3} from Definition~\ref{def G} onto zero. 

By definition of $\Phi$ and Remark~\ref{rmk map}~(\ref{rmk map 1}), linear, quadratic and cubic transformations in $\mathcal{J}_{\circ}$ and the group $\mathcal{J}_*$ are sent onto zero by $\Phi$, hence relations~\ref{relrel2} and \ref{relrel3} are contained in $\ker(\Phi)$. 
So, we just have to bother with relation~\ref{relrel1}:

Lets $\theta_1,\theta_2\in\mathcal{J}_{\circ}$ be standard quintic transformations, $\alpha_1,\alpha_2\in\Aut_\R(\PP^2)$ such that
\[\theta_2=\alpha_2\theta_1\alpha_1.\]
We want to show that $\Phi(\theta_1)=\Phi(\theta_2)$. By definition of $\Phi$, we have $\Phi(\theta_i)=\varphi_{\circ}(\theta_i)$, $i=1,2$, so we need to show that $\varphi_{\circ}(\theta_1)=\varphi_{\circ}(\theta_2)$.

If $\alpha_1\in\mathcal{J}_{\circ}$, then $\alpha_2=\theta_2\alpha_1^{-1}\theta_1^{-1}\in\mathcal{J}_{\circ}$. Similarly, if $\alpha_2\in\mathcal{J}_{\circ}$, then also $\alpha_1\in\mathcal{J}_{\circ}$. 
In either case, we have $\varphi_{\circ}(\alpha_2\theta_1\alpha_2)=\varphi_{\circ}(\alpha_2)\varphi_{\circ}(\theta_1)\varphi_{\circ}(\alpha_2)=\varphi_{\circ}(\theta_1)$. 

Suppose that $\alpha_1,\alpha_2\notin\mathcal{J}_{\circ}$. Let $S(\theta_1)=\{(p_3,\bar{p}_3)\}$ and $S((\theta_1)^{-1})=\{(p_4,\bar{p}_4)\}$. The map $\alpha_1^{-1}$ sends the base-points of $\theta_1$ onto the base-points of $\theta_2$, and since $\alpha_1^{-1}\notin\J_{\circ}$, it sends $p_3$ onto one of $p_1,\bar{p}_1,p_2,\bar{p}_2$. Similarly, $\alpha_2$ sends $p_4$ onto one of $p_1,\dots,\bar{p}_2$.
By Remark~\ref{rmk Jcirc elt} there exist $\beta_1,\beta_2,\gamma_1,\gamma_2\in\J_{\circ}\cap\Aut_\R(\PP^2)$ permutations of $p_1,\dots,\bar{p}_2$ such that 
\begin{align*} &(\beta_1\alpha_1^{-1}\gamma_1)(p_1)=p_1,& (\beta_1\alpha_1^{-1}\gamma_1)(\gamma_1^{-1}(p_3))=p_2,\\ 
&(\beta_2\alpha_2\gamma_2)(p_1)=p_1,& (\beta_2\alpha_2\gamma_2)(\gamma_2^{-1}(p_4))=p_2.
\end{align*} 
Remark~\ref{lem linear 2} with $\alpha_{p_3}:=\beta_1\alpha_1^{-1}\gamma_1$ and $\alpha_{p_4}:=\beta_2\alpha_2\gamma_2$ implies that
\[\varphi(\theta_2)\underset{\text{linear}}{\stackrel{\beta_1,\beta_2\in\J_{\circ}}=}\varphi(\beta_2\theta_2\beta_1^{-1})\stackrel{\theta_2=\alpha_2\theta_1\alpha_1}=\varphi_{\circ}(\alpha_{p_4}(\gamma_2^{-1}\theta_1\gamma_1^{-1})(\alpha_{p_3})^{-1})\stackrel{\text{Rmk.}\ref{lem linear 2}}=\varphi_{\circ}(\theta_1).\]

So, $\ker(\pi)\subset\ker(\Phi)$, and $\Phi$ induces a morphism $\varphi\colon\Bir_\R(\PP^2)\rightarrow\bigoplus_{(0,1]}\Z/2\Z$. It is surjective because $\varphi_{\circ}$ is surjective (Lemma~\ref{lem hom}).

If $f\in\Bir_{\R}(\PP^2)$ is of degree 2 or 3 there exists $\alpha,\beta\in\Aut_{\R}(\PP^2)$ such that $\beta f\alpha\in\mathcal{J}_*$. Hence $\varphi(f)=0$. If $\deg(f)=4$, $f$ is a composition of quadratic maps, hence $\varphi(f)=0$.
\end{proof}

Let $X$ be a real rational variety. Recall that we denote by $X(\R)$ its set of real points, and by $\Aut(X(\R))\subset\Bir(X)$ the subgroup of transformations defined at each point of $X(\R)$. We now look at restrictions of $\varphi\colon\Bir_\R(\PP^2)\rightarrow\bigoplus_{(0,1]}\Z/2\Z$ to $\Aut(X(\R))$. 

\begin{Cor}\label{cor:aut1}
There exist surjective group homomorphisms 
\[\Aut(\PP^2(\R))\rightarrow\bigoplus_{(0,1]}\Z/2\Z,\qquad\Aut(\A^2(\R))\rightarrow\bigoplus_{(0,1]}\Z/2\Z.\] 
\end{Cor}

\begin{proof}
We identify $\A^2(\R)$ with $\PP^2(\R)\setminus L_{p_1,\bar{p}_1}$. All quintic transformations are contained in $\Aut(\PP^2(\R))$ (Lemma~\ref{lem 5.1}) and preserve $C_3:=L_{p_1,\bar{p}_1}\cup L_{p_2,\bar{p}_2}$. For any standard quintic transformation $\theta$ there exists a permutation $\alpha\in\J_{\circ}\cap\Aut_\R(\PP^2)$ of $p_1,\dots,\bar{p}_2$ such that $\alpha\theta$ preserves $L_{p_i,\bar{p}_i}$, $i=1,2$, i.e. is contained in $\Aut(\A^2(\R))$. Therefore, the restriction of $\varphi$ to $\Aut(\PP^2(\R))$ and $\Aut(\A^2(\R))$ is surjective.
\end{proof}

We now look at two real rational quadric surfaces, namely $\mathcal{Q}_{3,1}$ and $\FF_0\simeq\PP^1\times\PP^1$.
Let $\mathcal{Q}_{3,1}\subset\PP^3$ be the variety given by the equation $w^2=x^2+y^2+z^2$. Its real part $\mathcal{Q}_{3,1}(\R)$ is the $2$-sphere $\mathbb{S}^2$. Consider the stereographic projection
\[\p\colon\mathcal{Q}_{3,1}\dashrightarrow\PP^2,\quad [w:x:y:z]\dashmapsto[w-z:x:y]\]
It is a real birational transformation obtained by first blowing-up the point $[1:0:0:1]$ and then blowing down the singular hyperplane section $w=z$ onto the points $p_2,\bar{p}_2$. It sends the exceptional divisor of $[1:0:0:1]$ onto the line $L_{p_2,\bar{p}_2}$ passing through $p_2,\bar{p}_2$. The inverse $\p^{-1}$ is an isomorphism around $p_1,\bar{p}_1$ and $\p$ sends a general hyperplane section onto a general conic passing through $p_2,\bar{p}_2$.\par

\begin{Cor}\label{cor:aut2}
The following homomorphism is surjective.
\[\Aut(\mathcal{Q}_{3,1}(\R))\rightarrow\bigoplus_{(0,1]}\Z/2\Z,\quad f\mapsto \varphi(\p^{-1}f\p)\] 
\end{Cor}
\begin{proof}
By \cite[Theorem 1]{KM09} (see also \cite[Theorem 1.4]{BM12}), the group $\Aut(\mathcal{Q}_{3,1}(\R))$ is generated by $\Aut_{\R}(\mathcal{Q}_{3,1})=\PO(3,1)$ and the family of standard cubic transformations $($see \cite[Example 5.1]{BM12} for definition$)$.
To prove the Lemma, it suffices to show that every standard quintic transformation $\theta\in\mathcal{J}_{\circ}$ that contracts onto $p_2$ the conic passing through all its base-points except $p_2$ is conjugate via $\p$ to a standard cubic transformation in $\Aut(\mathcal{Q}_{3,1}(\R))$. The standard quintic transformation $\theta$ preserves the line $L_{p_2,\bar{p}_2}$, hence $\p\theta\p^{-1}$ has only non-real base-points, i.e. $\p\theta\p^{-1}\in\Aut(\mathcal{Q}_{3,1}(\R))$. So, by \cite[Lemma 5.4 (3)]{BM12}, we have to show that $(\p\theta\p^{-1})^{-1}$ sends a general hyperplane section onto a cubic section passing through $\p^{-1}(p_1),\p^{-1}(\bar{p}_1)$, $\p^{-1}(p_3),\p^{-1}(\bar{p}_3)$ with multiplicity~$2$.\par
Let $p_1,\bar{p}_1$, $p_2,\bar{p}_2$, $p_3,\bar{p}_3$ be the base-points of $\theta$. The map $\p$ sends a general hyperplane section onto a general conic $C$ passing through $p_2,\bar{p}_2$. The curve $\theta^{-1}(C)$ is a curve of degree $6$ with multiplicity $3$ in $p_2,\bar{p}_2$ and multiplicity $2$ in $p_1,\bar{p}_1,p_3,\bar{p}_3$. Therefore, $(\p^{-1}\theta^{-1})(C)\subset\mathcal{Q}_{3,1}$ is a curve of self-intersection 18 passing through $\p^{-1}(p_1),\p^{-1}(\bar{p}_1)$, $\p^{-1}(p_3),\p^{-1}(\bar{p}_3)$ with multiplicity $2$. It follows that $(\p^{-1}\theta\p)^{-1}$ sends a general hyperplane section onto a cubic section having multiplicity~$2$ at these four points. 
\end{proof}

\begin{Cor}\label{cor:F0}
For any real birational map $\psi\colon\FF_0\dashrightarrow\PP^2$, the group $\psi\Aut(\FF_0(\R))\psi^{-1}$ is a subgroup of $\ker(\varphi)$. 
\end{Cor}
\begin{proof}
By \cite[Theorem 1.4]{BM12}, the group $\Aut(\FF_0(\R))$ is generated by $\Aut_\R(\FF_0)\simeq\mathrm{PGL}(2,\R)^2\rtimes\Z/2\Z$ and the involution
\[\tau\colon([u_0:u_1],[v_0:v_1])\dashmapsto([u_0,u_1],[u_0v_0+u_1v_1:u_1v_0-u_0v_1]).\]
Consider the real birational map 
\[\psi\colon\PP^2\dashrightarrow\FF_0,\quad[x:y:z]\dashmapsto([x:z],[y:z]),\] 
with inverse $\psi^{-1}\colon([u_0:u_1],[v_0,v_1])\dashmapsto[u_0v_1:u_1v_0:u_1v_1]$. 
A quick calculation shows that the conjugate by $\psi$ of these generators of $\Aut(\FF_0(\R))$ are of degree at most $3$. Proposition~\ref{thm quotient} implies that they are contained in $\ker(\varphi)$. In particular, $\psi^{-1}\Aut(\FF_0(\R))\psi\subset\ker(\varphi)$. Since $\ker(\varphi)$ is a normal subgroup of $\Bir_\R(\PP^2)$, the same statement holds for any other real birational map $\PP^2\dashrightarrow\FF_0$.
\end{proof}

\begin{Cor}[Corollary~\ref{cor 1}]\label{cor quotient}
For any $n\in\N$ there is a normal subgroup of $\Bir_{\R}(\PP^2)$ of index $2^n$ containing all elements of degree $\leq 4$.
The same statement holds for $\Aut(\PP^2(\R))$, $\Aut(\A^2(\R))$ and $\Aut(\mathcal{Q}_{3,1}(\R))$. 
\end{Cor}

\begin{proof}
Let $pr_{\delta_1,\dots,\delta_n}\colon\bigoplus_{(0,1]}\Z/2\Z\rightarrow(\Z/2\Z)^n$ be the projection onto the $\delta_1,\dots,\delta_n$-th factors. Then $pr_{\delta_1,\dots,\delta_n}\circ\varphi$ has kernel of index $2^n$ containing $\ker(\varphi)$ and thus all elements of degree $\leq4$. By Corollary~\ref{cor:aut1} and Corollary~\ref{cor:aut2}, the same argument works for $\Aut(\PP^2(\R))$, $\Aut(\A^2(\R))$ and $\Aut(\mathcal{Q}_{3,1}(\R))$. 
\end{proof}

Lemma~\ref{lem quintlin} and Theorem~\ref{thm BM} imply that $\Bir_\R(\PP^2)$ is generated by $\Aut_\R(\PP^2),\sigma_1,\sigma_0$ and all standard quintic transformations in $\mathcal{J}_{\circ}$. This generating set is not far from being minimal:

\begin{Cor}\label{cor generating set}
The group $\Bir_\R(\PP^2)$ is not generated by $\Aut_\R(\PP^2)$ and a countable family of elements. The same statement holds for $\Aut(\PP^2(\R))$, $\Aut(\A^2(\R))$ and $\Aut(\mathcal{Q}_{3,1}(\R))$, replacing $\Aut_\R(\PP^2)$ for the latter two by the affine automorphism group of $\A^2$ and by $\Aut_\R(\mathcal{Q}_{3,1})$ respectively.
\end{Cor}

\begin{proof}
If $\Bir_\R(\PP^2)$ was generated by $\Aut_\R(\PP^2)$ and a countable family $\{f_n\}_{n\in\N}$ of elements of $\Bir_\R(\PP^2)$ then by Proposition~\ref{thm quotient}, the countable family would yield a countable generating set of $\oplus_{(0,1]}\Z/2\Z$, which is impossible. 
The same argument works for $\Aut(\PP^2(\R))$, $\Aut(\mathcal{Q}_{3,1}(\R))$ and $\Aut(\A^2(\R))$ - for the latter two we replace $\Aut_\R(\PP^2)$ respectively by $\Aut_\R(\mathcal{Q}_{3,1})$ and by the subgroup of affine automorphisms of $\A^2$, which corresponds to $\Aut_\R(\PP^2)\cap\Aut(\A^2(\R))$. 
\end{proof}

Corollary~\ref{cor:aut1}, Corollary~\ref{cor:aut2}, Corollary~\ref{cor:F0} and Corollary~\ref{cor generating set} imply Corollary~\ref{cor 2}.

\begin{Cor}[Corollary~\ref{cor 4}] 
The normal subgroup of $\Bir_\R(\PP^2)$ generated by any countable set of elements of $\Bir_\R(\PP^2)$ is a proper subgroup of $\Bir_\R(\PP^2)$. The same statement holds for $\Aut(\PP^2(\R))$, $\Aut(\A^2(\R))$ and $\Aut(\mathcal{Q}_{3,1}(\R))$. 
\end{Cor}
\begin{proof}
Let $S\subset\Bir_\R(\PP^2)$ be a countable set of elements. Its image $\varphi(S)\subset\bigoplus_{(0,1]}\Z/2\Z$ is a countable set and hence a proper subset of $\bigoplus_{(0,1]}\Z/2\Z$. Since $\varphi$ is surjective (Proposition~\ref{thm quotient}), the preimage $\varphi^{-1}(\varphi(S))\subsetneq\Bir_\R(\PP^2)$ is a proper subset. The group $\oplus_{(0,1]}\Z/2\Z$ is abelian, so the set $\varphi^{-1}(\varphi(S))$ contains the normal subgroup of $\Bir_\R(\PP^2)$ generated by $S$, which in particular is a proper subgroup of $\Bir_\R(\PP^2)$.
\end{proof}

\begin{Rmk}\label{rmk conjugate}\item
(1) The morphism $\varphi\colon \Bir_\R(\PP^2)\rightarrow\bigoplus_{(0,1]}\Z/2\Z$ does not have any sections: If it had a section, then for any $k\in\N$ the group $(\Z/2\Z)^k$ would embed into $\Bir_\R(\PP^2)$, which is not possible by \cite{B07}.

\noindent(2) Over $\C$, the analogues of $\mathcal{J}_{\circ}$ and $\mathcal{J}_*$ are conjugate because the pencil of conics through four points can be send by an element of $\Bir_\C(\PP^2)$ onto the pencil of lines through one point.  
This is not true over $\R$: by Proposition~\ref{thm quotient}, one is contained in $\ker(\varphi)$ and the other is not. 

\noindent(3) No proper normal subgroup of $\Bir_{\R}(\PP^2)$ of finite index is closed with respect to the Zariski or Euclidean topology because $\Bir_\R(\PP^2)$ is topologically simple \cite{BZ15}.

\noindent(4) The group $\Bir_\C(\PP^2)$ does not contain any proper normal subgroups of countable index: Assume that $\{\Id\}\neq N$ is a normal subgroup of countable index. The image of $\mathrm{PGL}_3(\C)$ in the quotient is countable, hence $\mathrm{PGL}_3(\C)\cap N$ is non-trivial. Since $\mathrm{PGL}_3(\C)$ is a simple group, we have $\mathrm{PGL}_3(\C)\subset N$. Since the normal subgroup generated by $\mathrm{PGL}_3(\C)$ is $\Bir_\C(\PP^2)$ \cite[Lemma 2]{Giz}, we get that $N=\Bir_\C(\PP^2)$.
\end{Rmk}

\section{The kernel of the quotient}\label{section kernel}

In this section, we prove that the kernel of $\varphi\colon\Bir_\R(\PP^2)\rightarrow\bigoplus_{(0,1]}\Z/2\Z$ is the smallest normal subgroup containing $\Aut_\R(\PP^2)$, which will turn out to be the commutator subgroup of $\Bir_\R(\PP^2)$. It implies that $\varphi$ is in fact the abelianisation of $\Bir_\R(\PP^2)$.

The key idea is to show that $\J_*$ is contained in the normal subgroup generated by $\Aut_\R(\PP^2)$ (Lemma~\ref{lem quadr conj}) and that in the decomposition of an element of $\J_{\circ}$, we can group the standard or special quintic transformations having the same image by $\varphi_{\circ}$.
Then we apply the following: if two standard or special quintic transformations $\theta_1,\theta_2$ are sent by $\varphi_{\circ}$ onto the same image, then $\theta_2$ can be obtained by composing $\theta_1$ with a suitable amount of linear and quadratic elements (Lemma~\ref{rmk tilde 3.2}, \ref{lem tilde 3}, \ref{lem tilde 3.3}), which will imply that $\theta_1(\theta_2)^{-1}$ is contained the normal subgroup of $\Bir_\R(\PP^2)$ generated by $\Aut_\R(\PP^2)$ (Lemma~\ref{lem claim 3}, \ref{lem quad in normal}).

\begin{Def}
We denote by $\langle\langle\Aut_{\R}(\PP^2)\rangle\rangle$ the smallest normal subgroup of $\Bir_{\R}(\PP^2)$ containing $\Aut_{\R}(\PP^2)$.
\end{Def}

\subsection{Geometry between cubic and quintic transformations}
One part of the proof that $\ker(\varphi)=\langle\langle\Aut_\R(\PP^2)\rangle\rangle$ is to see that if two standard quintic transformations are sent onto the same standard vector in $\bigoplus_{(0,1]}\Z/2\Z$, then one is obtained from the other by composing from the right and the left with suitable cubic maps, which in turn can be written as composition of quadratic maps. For this, we first have to dig into the geometry of cubic maps.

\begin{Lem}\label{lem deg 3}
Let $q\in\PP^2(\C)\setminus\{p_1,\bar{p}_1,p_2,\bar{p}_2\}$ be a non-real point such that $C_q$ is irreducible. 
Then there exists a real point $r\in L_{q,\bar{p}_2}$ and $f\in\mathcal{J}_{\circ}$ of degree $2$ or $3$ with $r$ among its base-points such that 
\begin{enumerate}
\item $f(C_{q})=C_q$,
\item the image of $q$ by $f$ 
is infinitely near $p_1$ and corresponds to the tangent direction of $f(C_{q})$ in $p_1$,
\item either $\deg(f)=3$ and $C_r$ is smooth or $\deg(f)=2$ and $C_r$ is singular,
\end{enumerate}
\end{Lem}
\begin{figure}
\def\svgwidth{0.8\textwidth}
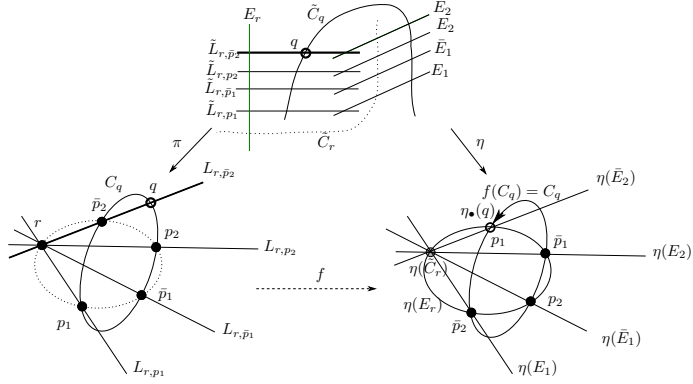
\caption{The cubic transformation of Lemma~\ref{lem deg 3}}
\end{figure}

\begin{proof}
Since $C_q$ is irreducible, $q$ is not collinear with any of $p_1,\bar{p}_1,p_2,\bar{p}_2$. It follows that $L_{q,\bar{p}_2}\neq L_{q,p_2}$, and so $L_{q,\bar{p}_2}$ and $\bar{L}_{q,\bar{p}_2}$ intersect in exactly one point $r$, which is a real point. 
 
If $r$ is not collinear with any two of $p_1,\bar{p}_1,p_2,\bar{p}_2$, then Lemma~\ref{lem sigma deg 3} ensures the existence of $f\in\mathcal{J}_{\circ}$ of degree $3$ with singular point $r$. The line $L_{q,\bar{p}_2}$ is contracted onto $p_i$ or $\bar{p}_i$, $i\in\{1,2\}$. By composing with elements of $\Aut_\R(\PP^2)\cap\mathcal{J}_{\circ}$, we can assume that $L_{q,\bar{p}_2}$ is contracted onto $p_1$ and that $f$ preserves the conic $L_{p_1,p_2}\cup L_{\bar{p}_1,\bar{p}_2}$, and thus induces the identity map on $\PP^1$ (Lemma~\ref{lem scale on P1}), and therefore preserves $C_{q}$. It follows that the image of $q$ by $f$ is infinitely near $p_1$ and corresponds to the tangent direction of $f(C_{q})=C_{q}$.

If $r$ is collinear with two of $p_1,\bar{p}_1,p_2,\bar{p}_2$, it is collinear with $p_1,\bar{p}_1$ and not collinear with any other two of the four points. Lemma~\ref{lem sigma} implies that there exists $f\in\mathcal{J}_{\circ}$ of degree $2$ with base-points $r,p_2,\bar{p}_2$. After composing with a linear map in $\J_{\circ}$, we may assume that $f$ contracts $L_{r,\bar{p}_2}=L_{q,\bar{p}_2}$ onto $p_1$ -- then $f(\{p_1,\bar{p}_1\})=\{p_2,\bar{p}_2\}$ -- and such that $f(p_1)=p_2$. Then the image of $q$ by $f$ is infinitely near $p_1$. We claim that $f(C_q)=C_q$. Call $\hat{f}$ the automorphism of $\PP^1$ induced by $f$. We calculate $\hat{f}^{-1}$ (cf. proof of Lemma~\ref{lem scale on P1}). The conditions on $f$ imply $f(L_{p_1,p_2})=L_{p_1,p_2}$. Then $\hat{f}^{-1}\colon [u:v]\mapsto[(r_1^2+(r_0+r_2)^2)u:(r_1^2+(r_0-r_2)^2)v]$, where $r=[r_0:r_1:r_2]$. Since $r\in L_{p_1,\bar{p}_1}$, we have $r_2=0$ and so $\hat{f}^{-1}=\Id$. In particular, $f(C_q)=C_q$.
\end{proof}

\begin{Lem}\label{rmk tilde 3.2}
Let $\theta_1,\theta_2\in\mathcal{J}_{\circ}$ be special quintic transformations with $S(\theta_i)=\{(q_i,\bar{q}_i)\}$. If $C_{q_1}=C_{q_2}$ or $C_{q_1}=C_{\bar{q}_2}$, then there exist $\alpha_1,\alpha_2\in\mathcal{J}_{\circ}\cap\Aut_\R(\PP^2)$ such that $\theta_2=\alpha_2\theta_1\alpha_1$.  
\end{Lem}
\begin{proof}
We can assume that $q_1$ is infinitely near $p_1$ as the proof is the same if it is infinitely near $\bar{p}_1,p_2$ or $\bar{p}_2$. If $q_2$ or $\bar{q}_2$ are infinitely near $p_1$ as well, then $\{q_1,\bar{q}_1\}=\{q_2,\bar{q}_2\}$. Therefore $\theta_2\theta_1^{-1}\in\mathcal{J}_{\circ}\cap\Aut_\R(\PP^2)$. \par
Suppose that $q_2$ or $\bar{q}_2$ is infinitely near $p_2$. The automorphism $\alpha\colon[x:y:z]\mapsto[z:-y:x]$ is contained in $\mathcal{J}_{\circ}$ and exchanges $p_1$ and $p_2$, while inducing the identity map on $\PP^1$. Then $\theta_2\alpha$ and $\theta_1$ are in the case above. 
\end{proof}

\begin{Lem}\label{lem tilde 3}
Let $\theta_1,\theta_2\in\mathcal{J}_{\circ}$ be standard quintic transformations with $S(\theta_i)=\{(q_i,\bar{q}_i)\}$, $i=1,2$. Assume that $C_{q_1}=C_{q_2}$ or $C_{q_1}=C_{\bar{q}_2}$. 

Then there exist $\tau_1,\dots,\tau_8\in\mathcal{J}_{\circ}$ of degree $\leq 2$ such that $\theta_2=\tau_8\cdots\tau_5\theta_1\tau_4\cdots\tau_1$.
\end{Lem}

\begin{proof}
By exchanging the names of $q_2,\bar{q}_2$, we can assume that $C_{q_1}=C_{q_2}$. It suffices to show that there exist $g_1,\dots,g_4\in\mathcal{J}_{\circ}$ of degree $\leq 3$ such that $\theta_2=g_4g_3\theta_1g_2g_2$, since every element of $\mathcal{J}_{\circ}$ of degree $3$ can be written as composition of two qudratic elements of $\mathcal{J}_{\circ}$ (Remark~\ref{rmk deg 3}). We give an explicit construction of the $g_i$'s.

According to Lemma~\ref{lem deg 3} there exist for $i=1,2$ a real point $r_i\in L_{q_i,\bar{p}_2}$ and $f_i\in\mathcal{J}_{\circ}$ of degree $2$ or $3$ and $r_i$ among its base-points such that $f_i$ preserves $C_{q_i}$ and the image $t_i$ of $q_i$ by $f_i$ 
is infinitely near $p_1$ and corresponds to the tangent direction of $C_{q_i}$ in $p_1$. 
The real conic $C_{r_i}$ is not contracted by $\theta_i$, and 
$\theta_i$ is an isomorphism around $r_i$ (Lemma~\ref{lem 5.1}).
Recall that $\theta_i$ preserves the set $\{C_1,C_2\}$ and the conic $C_3$. 

If $C_{r_i}$ is irreducible (and hence $\deg(f_i)=3$), then $\theta_i(C_{r_i})=C_{\theta_i(r_i)}$ is irreducible as well. 
Therefore, there exists $h_i\in\mathcal{J}_{\circ}$ of degree $3$ with base-point $\theta_i(r_i)$ (Lemma~\ref{lem sigma deg 3}). 

If $C_{r_i}$ is singular (and hence $\deg(f_i)=2$), then $C_{\theta_i(r_i)}$ is singular as well.  
Therefore, there exists $h_i\in\mathcal{J}_{\circ}$ of degree $2$ with base-points $\theta_i(r_i)$ among its base-points (Lemma~\ref{lem sigma}).

By composing $h_i$ with elements in $\mathcal{J}_{\circ}\cap\Aut_\R(\PP^2)$, we can assume that $h_i$ sends the line $\theta_i(L_{q_i,\bar{p}_2})$ onto $p_1$ (Remark~\ref{rmk deg 3}). 
Then $h_i\theta_i(f_i)^{-1}\in\mathcal{J}_{\circ}$ is a special quintic transformation with base-points $p_1,\bar{p}_1,p_2,\bar{p}_2,t_i,\bar{t}_i$, where $t_i$ is infinitely near $p_1$ and  corresponds to the tangent direction of $C_{q_i}$ on $p_1$. As $C_{q_1}=C_{q_2}$ by assumption, the maps $h_1\theta_1(f_1)^{-1}$ and $h_2\theta_2(f_2)^{-1}$ have exactly the same base-points. By Lemma~\ref{rmk tilde 3.2} we have $h_1\theta_1(f_1)^{-1}= \beta h_2\theta_2(f_2)^{-1}\alpha$ for some $\alpha,\beta\in\Aut_\R(\PP^2)\cap\mathcal{J}_{\circ}$. In particular, 
\[\theta_2=(h_2)^{-1}\beta^{-1}h_1\theta_1(f_1)^{-1}\alpha^{-1}f_2.\] 
\end{proof}

\begin{Lem}\label{lem tilde 3.3}
Let $\theta_1,\theta_2\in\mathcal{J}_{\circ}$ be a standard and a special quintic transformation respecitvely with $S(\theta_i)=\{(q_i,\bar{q}_i)\}$. Assume that $C_{q_1}=C_{q_2}$ or $C_{q_1}=C_{\bar{q}_2}$. 

Then there exists $\tau_1,\dots,\tau_{4}\in\mathcal{J}_{\circ}$ of degree $\leq2$ such that $\theta_2=\tau_{4}\tau_{3}\theta_1\tau_2\tau_1$. 
\end{Lem}

\begin{proof}
By exchanging the names of $q_1,\bar{q}_1,q_2,\bar{q}_2$, we can assume that $C_{q_1}=C_{q_2}$ and that $q_2$ is infinitely near $p_j$, $j\in\{1,2\}$. By Lemma~\ref{lem deg 3} there exists $f\in\mathcal{J}_{\circ}$ of degree $2$ or $3$ such that $f(C_{q_1})=C_{q_1}=C_{q_2}$ and the image $t$ of $q_1$ by $f$ is infinitely near $p_j$. Let $r$ be the real base-point of $f$. It is not on a conic contracted by $\theta_1$,
and $\theta_1$ is an isomorphism around $r$ (Lemma~\ref{lem 5.1}). 

If $C_r$ is irreducible (i.e. $\deg(f)=3$), the conic $\theta_1(C_r)=C_{\theta_1(r)}$ is irreducible as well. By Lemma~\ref{lem sigma deg 3} there exists $g\in\mathcal{J}_{\circ}$ of degree $3$ with double point $\theta_1(r)$. 

If $C_r$ is singular (i.e. $\deg(f)=2$), the conic $C_{\theta_1(r)}$ is singular as well. By Lemma~\ref{lem sigma} there exists $g\in\mathcal{J}_{\circ}$ of degree $2$ with $\theta_1(r)$ among its base-points.

The map $g\theta_1f^{-1}$ is a special quintic transformation with base-points $p_1,\bar{p}_1,p_2,\bar{p}_2,t,\bar{t}$, where $t$ is infinitely near $p_j$ and corresponds to the tangent directions $C_{q_1}$ in $p_j$. By assumption, we have $C_{q_1}=C_{q_2}$, so by Lemma~\ref{rmk tilde 3.2} there exists $\alpha,\beta\in\Aut_\R(\PP^2)\cap\mathcal{J}_{\circ}$ such that $\beta g\theta_1f^{-1}\alpha=\theta_2$. 
We can write $f$ and $g$ as composition of at most two quadratic transformations in $\J_{\circ}$ by Remark~\ref{rmk deg 3} and thus obtain $\tau_1,\dots,\tau_4$.
\end{proof}

\subsection{The smallest normal subgroup containing $\Aut_\R(\PP^2)$}

\begin{Lem}\label{lem quad in normal}
Any quadratic map in $\Bir_\R(\PP^2)$ is contained in $\langle\langle\Aut_\R(\PP^2)\rangle\rangle$.
\end{Lem}

\begin{proof}
Let $\tau\in\Bir_\R(\PP^2)$ be of degree 2. Pick two base-points $q_1,q_2$ of $\tau$ that are either a pair of non-real conjugate points or two real base-points, such that either both are proper points of $\PP^2$ or $q_1$ is a proper point of $\PP^2$ and $q_2$ is in the first neighbourhood of $q_1$. Let $t_1,t_2$ be base-points of $\tau^{-1}$ such that for $i=1,2$, $\tau$ sends the pencil of lines through $q_i$ onto the pencil of lines through $t_i$. Pick a general point $r\in\PP^2$ and let $s:=\tau(r)$. 
There exists $\alpha\in\Aut_\R(\PP^2)$ that sends $q_1,q_2,r$ onto $t_1,t_2,s$. The map $\tilde{\tau}:=\tau\alpha$ is of degree $2$, fixes $s$, and $t_1,t_2$ are base-points of $\tilde{\tau}$ and $\tilde{\tau}^{-1}$. 

Since $r$ is general, also $s$ is general, and there exists $\theta\in\Bir_\R(\PP^2)$ of degree 2 with base-points $t_1,t_2,s$. Observe that the map $\theta\tilde{\tau}\theta^{-1}$ is linear. In particular, $\tau$ is contained in $\langle\langle\Aut_\R(\PP^2)\rangle\rangle$.
\end{proof}

The following lemma is classical. 

\begin{Lem}\label{lem quadr conj}
The group $\mathcal{J}_*$ is generated by its quadratic and linear elements. In particular, $\mathcal{J}_*\subset\langle\langle\Aut_\R(\PP^2)\rangle\rangle$.
\end{Lem}

\begin{proof}
The group $\mathcal{J}_*$ is isomorphic to $\mathrm{PGL}_2(\R(x))\rtimes\mathrm{PGL}_2(\R)$ and is therefore generated by the elementary matrices in each factor. In other words, $\mathcal{J}_*$ is generated by its elements of the form
\[(x,y)\dashmapsto(ax,y),\quad (x,y)\dashmapsto(x+b,y),\quad(x,y)\dashmapsto(\nicefrac{1}{x},y),\quad a,b\in\R^*,\]
\[(x,y)\dashmapsto(x,\alpha(x)y),\quad (x,y)\dashmapsto(x,y+\beta(x)),\quad(x,y)\dashmapsto(x,\nicefrac{1}{y}),\quad \alpha,\beta\in\R(x).\]
The map $(x,y)\dashmapsto(x+b,y)$ can be obtained by conjugating $(x,y)\dashmapsto(x+1,y)$ with $(x,y)\dashmapsto(bx,y)$, so we only need $b=1$ in the first line. Similarly, we only need $\beta=1$ in the second line. Any element of $\R[x]$ can be factored into polynomials of degree $\leq2$, hence it suffices to take $\alpha\in\R[x]$ of degree $\leq2$ and $\frac{1}{\alpha}\in\R[x]$ of degree $\leq2$. This yields a generating set of elements of degree $\leq3$. Any of the above maps of degree $3$ have two non-real simple base-points in $\PP^2$ and thus, analogously to Remark~\ref{rmk deg 3}, we can decompose them into quadratic maps in $\J_*$. 
\end{proof}

\subsection{The kernel is equal to $\langle\langle\Aut_\R(\PP^2)\rangle\rangle$}
Take an element of $\ker(\varphi_{\circ})$. It is the composition of linear, quadratic and standard and special quintic elements (Lemma~\ref{lem Jcirc gen}). The next three lemmata show that we can choose the order of the linear, quadratic and standard and special quintic elements so that the ones belonging to the same coset are just one after another. These lemmata will be the remaining ingredients to prove that $\ker(\varphi)=\langle\langle\Aut_\R(\PP^2)\rangle\rangle$.

\begin{Lem}\label{lem claim 1}
Let $\tau,\theta\in\mathcal{J}_{\circ}$ be a quadratic and a standard (or special) quintic transformation respectively. Then there exist $\tilde{\tau}_1,\tilde{\tau}_2\in\mathcal{J}_{\circ}$ of degree $2$ and $\tilde{\theta}_1,\tilde{\theta}_2\in\mathcal{J}_{\circ}$ standard or special quintic transformations such that $\tau\theta=\tilde{\theta}_1\tilde{\tau}_1$ and $\theta\tau=\tilde{\tau}_2\tilde{\theta}_2$, i.e. we can "permute" $\tau,\theta$. 
\end{Lem}
\begin{proof}
The map $\tau^{-1}$ has base-points $p_i,\bar{p}_i,r$, for some $r\in\PP^2(\R)$, $i\in\{1,2\}$, and $\theta$ is an isomorphism around $r$ (Lemma~\ref{lem 5.1}).
Let $p_{j_i}\in\{p_1,p_2\}$ be the image by $\theta$ of the contracted conic $C_i$ not passing through $p_i$. 
The map $\theta\tau$ is of degree $6$, and $p_{j_i},\bar{p}_{j_i}$ are base-points of $(\theta\tau)^{-1}$. 
They are of multiplicity $3$ because $p_{j_i}$ is the image by $\theta\tau$ of the curve $(\tau^{-1})(C_i)$ of degree $3$. 
By Lemma~\ref{lem sigma} there exists $\tilde{\tau}\in\mathcal{J}_{\circ}$ of degree $2$ with base-points $\theta(r),p_{j_i},\bar{p}_{j_i}$. 
Then the map $\tilde{\theta}:=\tilde{\tau}\theta\tau\in\mathcal{J}_{\circ}$ is of degree $5$. We claim that it is a standard or special quintic transformations.
Since $\tau^{-1}\in\J_{\circ}$, the point $q\in\{p_1,p_2\}$ different from $p_i$ is not on a line contracted by $\tau^{-1}$. 
Then $\tau^{-1}(q),\tau^{-1}(\bar{q})$ and the two non-real base-points of $\tau$ are base-points of $\tilde{\theta}$ that are proper points in $\PP^2$; the map $\tilde{\theta}$ has two more base-points, which could be proper points of $\PP^2$ or not.
Thus $\tilde{\theta}$ is a standard or special quintic transformation. 
We put $\tilde{\tau}_2:=\tilde{\tau}^{-1}$, $\tilde{\theta}_2:=\tilde{\theta}$. 
A similar construction yields $\tilde{\theta}_1,\tilde{\tau}_1$.
\end{proof}

\begin{Lem}\label{lem claim 2}
Let $\theta_1,\theta_2\in\mathcal{J}_{\circ}$ be standard or special quintic transformations (both can be either) such that $\varphi_0(\theta_1)\neq\varphi_0(\theta_2)$. Then there exist $\theta_3,\theta_4\in\mathcal{J}_{\circ}$ standard or special quintic transformations, such that 
\[\theta_2\theta_1=\theta_4\theta_3,\quad \varphi_0(\theta_1)=\varphi_0(\theta_4),\quad \varphi_0(\theta_2)=\varphi_0(\theta_3)\] 
i.e. we can "permute" $\theta_1,\theta_2$.
\end{Lem}
\begin{proof}
Let $p_1,\bar{p}_1,p_2,\bar{p}_2,p_3,\bar{p}_3$ be the base-points of $\theta_1$ and $p_1,\bar{p}_1,p_2,\bar{p}_2,p_4,\bar{p}_4$ the ones of $\theta_2$. The assumption $\varphi_0(\theta_1)\neq\varphi_0(\theta_2)$ implies that $p_4\notin C_{p_3}\cup C_{\bar{p}_3}$. 

Let $p_5$ be the image of $p_4$ by $(\theta_1)^{-1}$, which is either a proper point of $\PP^2$ or in the first neighbourhood of one of $p_1,\bar{p}_1,p_2,\bar{p}_2$.
Because $p_4,\bar{p}_4,p_1,\bar{p}_1,p_2,\bar{p}_2$ are not on one conic, the points $p_5,\bar{p}_5,p_1,\dots,\bar{p}_2$ are not on one conic. So, by Lemma~\ref{lem quintlin} there exists a standard or special quintic transformation $\theta_3\in\mathcal{J}_{\circ}$ with base-points $p_1,\dots,\bar{p}_2,p_5,\bar{p}_5$. The map $\theta_4:=\theta_2\theta_1(\theta_3)^{-1}\in\mathcal{J}_{\circ}$ is a standard or special quintic transformation. In fact, its inverse has base-points 
$p_1,\dots,\bar{p}_2,q_3,\bar{q}_3$ where $q_3$ is the image of $p_3$ by $\theta_2$ and is a proper point of $\PP^2$ or infinitely near one of $p_1,\bar{p}_1,p_2,\bar{p}_2$.
We have by construction $\theta_2\theta_1=\theta_4\theta_3$. The equalities $\varphi_{\circ}(\theta_1)=\varphi_{\circ}( \theta_4)$ and $\varphi_{\circ}(\theta_2)=\varphi_{\circ}(\theta_3)$ follow from the construction and Remark~\ref{rmk map}~(\ref{rmk map 5}).
\end{proof}

\begin{Lem}\label{lem claim 3}
Let $\theta_1,\theta_2\in\mathcal{J}_{\circ}$ be standard or special quintic transformations $($both can be either$)$ such that $\varphi_0(\theta_1)=\varphi_0(\theta_2)$. Then $\theta_1(\theta_2)^{-1}\in\langle\langle\Aut_{\R}(\PP^2)\rangle\rangle$.
\end{Lem}
\begin{proof}
Let $S(\theta_1)=\{(p_3,\bar{p}_3)\}$ and $S(\theta_2)=\{(p_4,\bar{p}_4)\}$. The assumption $\varphi_0(\theta_1)=\varphi_0(\theta_2)$ implies that there exists some $\lambda\in\R$ such that $\pi_{\circ}(C_{p_3})=\lambda\pi_{\circ}(C_{p_4})$ or $\pi_{\circ}(C_{p_3})=\lambda\pi_{\circ}(C_{\bar{p}_4})$ in $\PP^1$. We can ssume that $\pi_{\circ}(C_{p_3})=\lambda\pi_{\circ}(C_{p_4})$, after exchanging the names of $p_4,\bar{p}_4$ if necessary.

Suppose that $\lambda>0$. By Lemma~\ref{lem scale on P1} there exist $\tau_1\in\mathcal{J}_{\circ}$ of degree 2 such that $\pi_{\circ}(\tau_1(C_{p_3}))=\pi_{\circ}(C_{p_4})$, i.e. $\tau_1(C_{p_3})=C_{p_4}$. Let $r$ be the real base-point of $\tau$. The map $\theta_1$ is an isomorphism around $r$ (Lemma~\ref{lem 5.1}), and $\theta_1(r)$ is a base-point of $(\theta_1\tau_1)^{-1}$.
Let $p_{j_i}$ be the image by $\theta_1$ of the contracted conic not passing through $p_i$. The map $\theta_1\tau_1$ is of degree $6$ and $p_{j_i},\bar{p}_{j_i}$ are base-points of $(\theta_1\tau_1)^{-1}$ of multiplicity $3$. By Lemma~\ref{lem sigma} there exists $\tau_2\in\mathcal{J}_{\circ}$ of degree $2$ with base-points $\theta(r),p_{j_i},\bar{p}_{j_i}$. The map $\tau_2\theta_1\tau_1\in\mathcal{J}_{\circ}$ is a standard or special quintic transformation contracting the conics $C_{p_4},C_{\bar{p}_4}$. 
Hence, by Lemma~\ref{lem tilde 3},~\ref{rmk tilde 3.2}, and~\ref{lem tilde 3.3}, there exist $\nu_1,\dots,\nu_{2m}\in\mathcal{J}_{\circ}$ of degree $\leq2$ such that $\theta_2=\nu_{2m}\cdots\nu_{m+1}(\tau_2\theta_1\tau_n^{-1})\nu_m\cdots\nu_1$. Then
\[\theta_1(\theta_2)^{-1}=\left(\theta_1(\nu_m\cdots\nu_1)^{-1}(\tau_1)\theta_1^{-1})\right (\tau_2)^{-1}(\nu_{2m}\cdots\nu_{m+1})^{-1}.\]
By Lemma~\ref{lem quad in normal}, all quadratic elements of $\mathcal{J}_{\circ}$ belong to $\langle\langle\Aut_\R(\PP^2)\rangle\rangle$, so $\theta_1(\theta_2)^{-1}$ is contained in $\langle\langle\Aut_{\R}(\PP^2)\rangle\rangle$.

If $\lambda<0$, take $\alpha_{p_4}\in\Aut_\R(\PP^2)$ fixing $p_1$ and sending $p_4$ onto $p_2$. If $S((\theta_2)^{-1})=\{(p_4',\bar{p}_4')\}$, we analogously pick $\alpha_{p_4'}\in\Aut_\R(\PP^2)$. Then $\alpha_{p_4'}\theta_2(\alpha_{p_4})^{-1}\in\J_{\circ}$ is a standard or special quintic transformation with base-points $p_1,\dots,\bar{p}_2,\alpha_{p_4}(p_2),\alpha_{p_4}(\bar{p}_2)$. Lemma~\ref{lem linear} implies that $\pi_{\circ}(C_{\alpha_{p_4}(p_2)})=\mu\pi_{\circ}(C_{p_4})$ for some $\mu\in\R_{<0}$. In particular, $\pi_{\circ}(C_{p_3})=\lambda\pi_{\circ}(C_{p_4})=\lambda\mu^{-1}\pi_{\circ}(C_{\alpha_{p_4}(p_2)})$. We proceed as above with $\theta_1$ and $\theta_2':=\alpha_{p_4'}\theta_2(\alpha_{p_4})^{-1}$, $\lambda':=\lambda\mu^{-1}>0$ and obtain that
\[\theta_1\alpha_{p_4}(\theta_2)^{-1}(\alpha_{p_4'})^{-1}=\theta_1(\alpha_{p_4'}\theta_2(\alpha_{p_4})^{-1})^{-1}\in\langle\langle\Aut_\R(\PP^2)\rangle\rangle.\]
The claim follows after conjugating $\theta_1\alpha_{p_4}(\theta_2)^{-1}$ with $\theta_2$.
\end{proof}

\begin{Prop}\label{thm kernel}
Let $\varphi\colon\Bir_{\R}(\PP^2)\rightarrow\bigoplus_{(0,1]}\Z/2\Z$ be the surjective group homomorphism defined in Theorem $\ref{thm quotient}$. Then
\[\ker(\varphi)=\langle\langle\Aut_{\R}(\PP^2)\rangle\rangle\]
\end{Prop}

\begin{proof}
By definition of $\varphi$ (see Proposition~\ref{thm quotient}), $\Aut_{\R}(\PP^2)$ is contained in $\ker(\varphi)$, hence $\langle\langle\Aut_{\R}(\PP^2)\rangle\rangle\subset\ker(\varphi)$. Lets prove the other inclusion. Consider the commutative diagram (\ref{phi}):
\[\small\xymatrix{ \Aut_{\R}(\PP^2)\ast\mathcal{J}_*\ast\mathcal{J}_{\circ}\ar[r]^{\pi}\ar@/_1pc/[rr]_{\Phi} & \mathcal{G}\simeq\Bir_{\R}(\PP^2)\ar[r]^{\varphi} &\bigoplus\limits_{(0,1]}\Z/2\Z} \]
where $\varphi\pi=\varphi_{\circ}$. The groups $\Aut_\R(\PP^2)$ and $\J_*$ are sent onto zero by $\Phi$, so $\ker(\Phi)$ is the normal subgroup generated by $\Aut_\R(\PP^2),\J_*$ and $\ker(\varphi_{\circ})$. Then $\ker(\varphi)=\pi(\ker(\Phi))$ implies that $\ker(\varphi)$ is the normal subgroup generated by $\Aut_\R(\PP^2),\mathcal{J}_*$ and $\ker(\varphi_{\circ})$. Moreover, $\Aut_{\R}(\PP^2)$ and $\mathcal{J}_*$ are contained in $\langle\langle\Aut_{\R}(\PP^2)\rangle\rangle$ (Lemma~\ref{lem quadr conj}), thus it suffices to prove that $\ker(\varphi_0)$ is contained in $\langle\langle\Aut_{\R}(\PP^2)\rangle\rangle$.

By Lemma~\ref{lem Jcirc gen}, every $f\in\ker(\varphi_{\circ})$ is the composition of linear, quadratic and standard quintic elements of $\mathcal{J}_{\circ}$. 
Note that a quadratic or quintic element composed with a linear element is still a quadratic or standard quintic element respectively, so we can assume that $f$ decomposes into quadratic and standard quintic elements. 

Let $\tau,\theta\in\J_{\circ}$ be a quadratic and a standard or special quintic transformation. 
By Lemma~\ref{lem claim 1} we can replace compositions $\tau\theta$ by $\theta'\tau'$ where $\tau',\theta'\in\J_{\circ}$ are a quadratic and a standard or special quintic transformation, and $\varphi_{\circ}(\tau\theta)=\varphi_{\circ}(\theta)=\varphi_{\circ}(\theta')=\varphi_{\circ}(\theta'\tau')$.
Let $\theta_1,\theta_2\in\J_{\circ}$ be standard quintic or special quintic transformations such that $\varphi_{\circ}(\theta_1)\neq\varphi_{\circ}(\theta_2)$. By Lemma~\ref{lem claim 2} we can write $\theta_2\theta_1=\theta_1'\theta_2'$ where $\theta_1',\theta_2'\in\J_{\circ}$ are standard or special quintic transformations and $\varphi_{\circ}(\theta_1)=\varphi_{\circ}(\theta_1')$ and $\varphi_{\circ}(\theta_2)=\varphi_{\circ}(\theta_2')$. 
So, using these two lemmas, we can create a new decomposition 
\[f=\theta_{n_k}\cdots\theta_{n_2+1}\theta_{n_2}\cdots\theta_{n_1+1}\theta_{n_1}\cdots\theta_1\tau_l\cdots\tau_1\] 
where $\tau_1,\dots,\tau_l\in\J_{\circ}$ are of degree $2$ and $\theta_1,\dots,\theta_{n_k}\in\J_{\circ}$ are standard and special quintic transformations, and the elements of each sequence $\theta_{n_i+1},\theta_{n_i+2},\dots,\theta_{n_{i+1}}$ have the same image by $\varphi_{\circ}$.
Since $f\in\ker(\varphi_{\circ})$, each sequence $\theta_{n_i+1},\theta_{n_i+2},\dots,\theta_{n_{i+1}}$ has an even number of elements, i.e. all $n_i$ are even.
We have $\varphi_{\circ}(\theta_i)=\varphi_{\circ}(\theta_i^{-1})$, hence
Lemma~\ref{lem claim 3} implies that $(\theta_{n_i+2(j+1)} \theta_{n_i+2j+1})\in\langle\langle\Aut_{\R}(\PP^2)\rangle\rangle$ for all $j\geq0$. 
Lemma~\ref{lem quad in normal} implies that all $\tau_1,\dots,\tau_l$ are contained in $\langle\langle\Aut_\R(\PP^2)\rangle\rangle$. The claim follows. 
\end{proof}

\begin{Cor}\label{cor thm kernel} \item 
$(1)$ We have $\langle\langle\Aut_\R(\PP^2)\rangle\rangle=\ker(\varphi)=\left[\Bir_\R(\PP^2),\Bir_\R(\PP^2)\right]$.

\noindent $(2)$ The sequence of iterated commutated subgroups of $\Bir_\R(\PP^2)$ is stationary. More specifically: Let $H:=[\Bir_\R(\PP^2),\Bir_\R(\PP^2)]$. Then $[H,H]=H$.
\end{Cor}

\begin{proof}
Since $\mathrm{PGL}_3(\R)$ is perfect, the group $\langle\langle\Aut_\R(\PP^2)\rangle\rangle$ is contained in the derived subgroup. This and Proposition~\ref{thm kernel} imply (1) and (2). 
\end{proof}

Theorem~\ref{thm 2} is the summary of Proposition~\ref{thm quotient}, Corollary~\ref{cor generating set}, Proposition~\ref{thm kernel}, Corollary~\ref{cor thm kernel}.

\begin{Rmk}\label{rmk kernel}
(1) The kernel of $\varphi$ is the normal subgroup $N$ generated by all squares in $\Bir_\R(\PP^2)$: On one hand, for any group $G$, its commutator subgroup $[G,G]$ is contained in the normal subgroup of $G$ generated by all squares. On the other hand, since $\bigoplus_{(0,1]}\Z/2\Z$ is Abelian and all its elements are of order $2$, the normal subgroup of $\Bir_\R(\PP^2)$ generated by the squares is contained in $\ker(\varphi)$. The claim now follows from $\ker(\varphi)=[\Bir_\R(\PP^2),\Bir_\R(\PP^2)]$ (Corollary~\ref{cor thm kernel}). 

(2) Endowed with the Zariski topology or the Euclidean topology (see \cite{BF13}), the group $\Bir_\R(\PP^2)$ does not contain any non-trivial proper closed normal subgroups and $\langle\langle\Aut_\R(\PP^2)\rangle\rangle$ is dense in $\Bir_\R(\PP^2)$ \cite{BZ15}. In particular, the quotient topology on $\bigoplus_{(0,1]}\Z/2\Z$ is the trivial topology.
\end{Rmk}


\end{document}

%% file: fig05.pdf_tex
\begingroup%
  \makeatletter%
  \providecommand\color[2][]{%
    \errmessage{(Inkscape) Color is used for the text in Inkscape, but the package 'color.sty' is not loaded}%
    \renewcommand\color[2][]{}%
  }%
  \providecommand\transparent[1]{%
    \errmessage{(Inkscape) Transparency is used (non-zero) for the text in Inkscape, but the package 'transparent.sty' is not loaded}%
    \renewcommand\transparent[1]{}%
  }%
  \providecommand\rotatebox[2]{#2}%
  \ifx\svgwidth\undefined%
    \setlength{\unitlength}{588.28703152bp}%
    \ifx\svgscale\undefined%
      \relax%
    \else%
      \setlength{\unitlength}{\unitlength * \real{\svgscale}}%
    \fi%
  \else%
    \setlength{\unitlength}{\svgwidth}%
  \fi%
  \global\let\svgwidth\undefined%
  \global\let\svgscale\undefined%
  \makeatother%
  \begin{picture}(1,0.25906747)%
    \put(0,0){\includegraphics[width=\unitlength,page=1]{fig05.pdf}}%
    \put(0.13389861,0.18765148){\color[rgb]{0,0,0}\makebox(0,0)[lb]{\smash{\SB{$\bar{p}_1$}}}}%
    \put(0,0){\includegraphics[width=\unitlength,page=2]{fig05.pdf}}%
    \put(0.03910737,0.18102746){\color[rgb]{0,0,0}\makebox(0,0)[lb]{\smash{\SB{$q$}}}}%
    \put(0.10894772,0.06628548){\color[rgb]{0,0,0}\makebox(0,0)[lb]{\smash{\SB{$p_1$}}}}%
    \put(0,0){\includegraphics[width=\unitlength,page=3]{fig05.pdf}}%
    \put(0.39112037,0.1959561){\color[rgb]{0,0,0}\makebox(0,0)[lb]{\smash{\SB{$\bar{p}_1$}}}}%
    \put(0.35578872,0.06836165){\color[rgb]{0,0,0}\makebox(0,0)[lb]{\smash{\SB{$p_1$}}}}%
    \put(0.2073909,0.13120167){\color[rgb]{0,0,0}\makebox(0,0)[lb]{\smash{\SB{$\alpha$}}}}%
    \put(0.46953079,0.13120159){\color[rgb]{0,0,0}\makebox(0,0)[lb]{\smash{\SB{$\sigma_1$}}}}%
    \put(0,0){\includegraphics[width=\unitlength,page=4]{fig05.pdf}}%
    \put(0.575111,0.19204117){\color[rgb]{0,0,0}\rotatebox{-0.2394346}{\makebox(0,0)[lb]{\smash{\SB{$[1:0:0]$}}}}}%
    \put(0,0){\includegraphics[width=\unitlength,page=5]{fig05.pdf}}%
    \put(0.54883189,0.05910311){\color[rgb]{0,0,0}\makebox(0,0)[lb]{\smash{\SB{$\bar{p}_1$}}}}%
    \put(0.67106147,0.10585273){\color[rgb]{0,0,0}\makebox(0,0)[lb]{\smash{\SB{$p_1$}}}}%
    \put(0,0){\includegraphics[width=\unitlength,page=6]{fig05.pdf}}%
    \put(0.81437638,0.05676341){\color[rgb]{0,0,0}\makebox(0,0)[lb]{\smash{\SB{$\bar{p}_1$}}}}%
    \put(0.94698671,0.10870347){\color[rgb]{0,0,0}\makebox(0,0)[lb]{\smash{\SB{$p_1$}}}}%
    \put(0,0){\includegraphics[width=\unitlength,page=7]{fig05.pdf}}%
    \put(0.74197329,0.1301635){\color[rgb]{0,0,0}\makebox(0,0)[lb]{\smash{\SB{$\beta$}}}}%
    \put(0,0){\includegraphics[width=\unitlength,page=8]{fig05.pdf}}%
    \put(0.0148733,0.07010948){\color[rgb]{0,0,0}\makebox(0,0)[lb]{\smash{\SB{$p_2$}}}}%
    \put(0,0){\includegraphics[width=\unitlength,page=9]{fig05.pdf}}%
    \put(0.10446474,0.24436557){\color[rgb]{0,0,0}\makebox(0,0)[lb]{\smash{\SB{$\bar{p}_2$}}}}%
    \put(0,0){\includegraphics[width=\unitlength,page=10]{fig05.pdf}}%
    \put(0.19978805,0.16338022){\color[rgb]{0,0,0}\makebox(0,0)[lb]{\smash{\SB{$[1:0:0]=\alpha(q)$}}}}%
    \put(0,0){\includegraphics[width=\unitlength,page=11]{fig05.pdf}}%
    \put(0.26972494,0.07813438){\color[rgb]{0,0,0}\makebox(0,0)[lb]{\smash{\SB{$\alpha(p_2)$}}}}%
    \put(0.2851967,0.24100396){\color[rgb]{0,0,0}\makebox(0,0)[lb]{\smash{\SB{$\alpha(\bar{p}_2)$}}}}%
    \put(0,0){\includegraphics[width=\unitlength,page=12]{fig05.pdf}}%
    \put(0.67061399,0.19100304){\color[rgb]{0,0,0}\rotatebox{-0.2394346}{\makebox(0,0)[lb]{\smash{\SB{$\bar{t}$}}}}}%
    \put(0.52113097,0.16816535){\color[rgb]{0,0,0}\rotatebox{-0.2394346}{\makebox(0,0)[lb]{\smash{\SB{$t$}}}}}%
    \put(0.94530678,0.18830943){\color[rgb]{0,0,0}\makebox(0,0)[lb]{\smash{\SB{$\bar{p}_2$}}}}%
    \put(0.78616426,0.17391714){\color[rgb]{0,0,0}\makebox(0,0)[lb]{\smash{\SB{$p_2$}}}}%
    \put(0.84916323,0.19723157){\color[rgb]{0,0,0}\rotatebox{-0.2394346}{\makebox(0,0)[lb]{\smash{\SB{$\beta([1:0:0])$}}}}}%
    \put(0,0){\includegraphics[width=\unitlength,page=13]{fig05.pdf}}%
    \put(0.42074118,0.02427966){\color[rgb]{0,0,0}\makebox(0,0)[lb]{\smash{\SB{$f:=\beta\sigma_1\alpha$}}}}%
  \end{picture}%
\endgroup%

%% file: fig06.pdf_tex
\begingroup%
  \makeatletter%
  \providecommand\color[2][]{%
    \errmessage{(Inkscape) Color is used for the text in Inkscape, but the package 'color.sty' is not loaded}%
    \renewcommand\color[2][]{}%
  }%
  \providecommand\transparent[1]{%
    \errmessage{(Inkscape) Transparency is used (non-zero) for the text in Inkscape, but the package 'transparent.sty' is not loaded}%
    \renewcommand\transparent[1]{}%
  }%
  \providecommand\rotatebox[2]{#2}%
  \ifx\svgwidth\undefined%
    \setlength{\unitlength}{1415.09814645bp}%
    \ifx\svgscale\undefined%
      \relax%
    \else%
      \setlength{\unitlength}{\unitlength * \real{\svgscale}}%
    \fi%
  \else%
    \setlength{\unitlength}{\svgwidth}%
  \fi%
  \global\let\svgwidth\undefined%
  \global\let\svgscale\undefined%
  \makeatother%
  \begin{picture}(1,0.31980765)%
    \put(0,0){\includegraphics[width=\unitlength,page=1]{fig06.pdf}}%
    \put(0.24730154,-0.17051238){\color[rgb]{0,0,0}\makebox(0,0)[lb]{\smash{}}}%
    \put(0.29252809,0.11215351){\color[rgb]{0,0,0}\makebox(0,0)[lb]{\smash{}}}%
    \put(0.38644187,0.15740488){\color[rgb]{0,0,0}\makebox(0,0)[lb]{\smash{\SB{$\tilde{C}_r$}}}}%
    \put(0.29292117,0.19855303){\color[rgb]{0,0,0}\makebox(0,0)[lb]{\smash{\SB{$\tilde{L}_{r,p_1}$}}}}%
    \put(0.51621595,0.23073293){\color[rgb]{0,0,0}\makebox(0,0)[lb]{\smash{\SB{$E_1$}}}}%
    \put(0.52361209,0.25986556){\color[rgb]{0,0,0}\makebox(0,0)[lb]{\smash{\SB{$\bar{E}_1$}}}}%
    \put(0.51654562,0.28382638){\color[rgb]{0,0,0}\makebox(0,0)[lb]{\smash{\SB{$E_2$}}}}%
    \put(0.51298314,0.3074457){\color[rgb]{0,0,0}\makebox(0,0)[lb]{\smash{\SB{$\bar{E}_2$}}}}%
    \put(0.29526284,0.27316133){\color[rgb]{0,0,0}\makebox(0,0)[lb]{\smash{\SB{$\tilde{L}_{r,\bar{p}_2}$}}}}%
    \put(0.29372242,0.24815192){\color[rgb]{0,0,0}\makebox(0,0)[lb]{\smash{\SB{$\tilde{L}_{r,p_2}$}}}}%
    \put(0.29125028,0.22653807){\color[rgb]{0,0,0}\makebox(0,0)[lb]{\smash{\SB{$\tilde{L}_{r,\bar{p}_1}$}}}}%
    \put(0.34098511,0.30817363){\color[rgb]{0,0,0}\makebox(0,0)[lb]{\smash{\SB{$E_q$}}}}%
    \put(0.48474087,0.1686867){\color[rgb]{0,0,0}\makebox(0,0)[lb]{\smash{}}}%
    \put(0,0){\includegraphics[width=\unitlength,page=2]{fig06.pdf}}%
    \put(0.01870868,0.16569299){\color[rgb]{0,0,0}\makebox(0,0)[lb]{\smash{\SB{$q$}}}}%
    \put(0.04946025,0.04215709){\color[rgb]{0,0,0}\makebox(0,0)[lb]{\smash{\SB{$p_1$}}}}%
    \put(0.15055644,0.04821503){\color[rgb]{0,0,0}\makebox(0,0)[lb]{\smash{\SB{$\bar{p}_1$}}}}%
    \put(0.17187959,0.1545896){\color[rgb]{0,0,0}\makebox(0,0)[lb]{\smash{\SB{$p_2$}}}}%
    \put(0.08489075,0.18841264){\color[rgb]{0,0,0}\makebox(0,0)[lb]{\smash{\SB{$\bar{p}_2$}}}}%
    \put(0.19710017,0.04917705){\color[rgb]{0,0,0}\makebox(0,0)[lb]{\smash{\SB{$L_{r,\bar{p}_1}$}}}}%
    \put(0.22045474,0.13643586){\color[rgb]{0,0,0}\makebox(0,0)[lb]{\smash{\SB{$L_{r,p_2}$}}}}%
    \put(0.16542857,0.20331071){\color[rgb]{0,0,0}\makebox(0,0)[lb]{\smash{\SB{$L_{r,\bar{p}_2}$}}}}%
    \put(0.09658676,0.01665474){\color[rgb]{0,0,0}\makebox(0,0)[lb]{\smash{\SB{$L_{r,p_1}$}}}}%
    \put(0,0){\includegraphics[width=\unitlength,page=3]{fig06.pdf}}%
    \put(0.36291501,0.11921162){\color[rgb]{0,0,0}\makebox(0,0)[lb]{\smash{\SB{$f$}}}}%
    \put(0,0){\includegraphics[width=\unitlength,page=4]{fig06.pdf}}%
    \put(0.55778675,0.10934925){\color[rgb]{0,0,0}\makebox(0,0)[lb]{\smash{\SB{$\eta(\tilde{C}_r)$}}}}%
    \put(0.65847796,0.04250787){\color[rgb]{0,0,0}\makebox(0,0)[lb]{\smash{\SB{$\bar{p}_1$}}}}%
    \put(0.78204333,0.07185123){\color[rgb]{0,0,0}\makebox(0,0)[lb]{\smash{\SB{$p_1$}}}}%
    \put(0.7893119,0.15158714){\color[rgb]{0,0,0}\makebox(0,0)[lb]{\smash{\SB{$\bar{p}_2$}}}}%
    \put(0.70554314,0.19662896){\color[rgb]{0,0,0}\makebox(0,0)[lb]{\smash{\SB{$p_2$}}}}%
    \put(0.70500907,0.00398011){\color[rgb]{0,0,0}\makebox(0,0)[lb]{\smash{\SB{$\eta(E_1)$}}}}%
    \put(0.80453232,0.04310481){\color[rgb]{0,0,0}\makebox(0,0)[lb]{\smash{\SB{$\eta(\bar{E}_1)$}}}}%
    \put(0.83776882,0.13790338){\color[rgb]{0,0,0}\makebox(0,0)[lb]{\smash{\SB{$\eta(E_2)$}}}}%
    \put(0.75986489,0.19974412){\color[rgb]{0,0,0}\makebox(0,0)[lb]{\smash{\SB{$\eta(\bar{E}_2)$}}}}%
    \put(0.57739318,0.06899358){\color[rgb]{0,0,0}\makebox(0,0)[lb]{\smash{\SB{$\eta(E_r)$}}}}%
    \put(0,0){\includegraphics[width=\unitlength,page=5]{fig06.pdf}}%
  \end{picture}%
\endgroup%

%% file: fig08.pdf_tex
\begingroup%
  \makeatletter%
  \providecommand\color[2][]{%
    \errmessage{(Inkscape) Color is used for the text in Inkscape, but the package 'color.sty' is not loaded}%
    \renewcommand\color[2][]{}%
  }%
  \providecommand\transparent[1]{%
    \errmessage{(Inkscape) Transparency is used (non-zero) for the text in Inkscape, but the package 'transparent.sty' is not loaded}%
    \renewcommand\transparent[1]{}%
  }%
  \providecommand\rotatebox[2]{#2}%
  \ifx\svgwidth\undefined%
    \setlength{\unitlength}{1021.38612282bp}%
    \ifx\svgscale\undefined%
      \relax%
    \else%
      \setlength{\unitlength}{\unitlength * \real{\svgscale}}%
    \fi%
  \else%
    \setlength{\unitlength}{\svgwidth}%
  \fi%
  \global\let\svgwidth\undefined%
  \global\let\svgscale\undefined%
  \makeatother%
  \begin{picture}(1,0.30202785)%
    \put(0,0){\includegraphics[width=\unitlength,page=1]{fig08.pdf}}%
    \put(0.39280033,0.19906192){\color[rgb]{0,0,0}\makebox(0,0)[lb]{\smash{\SBb{$\eta_{\alpha}$}}}}%
    \put(0.81824571,0.23480946){\color[rgb]{0,0,0}\makebox(0,0)[lb]{\smash{\SBb{$\PP^1\times\PP^1$}}}}%
    \put(0,0){\includegraphics[width=\unitlength,page=2]{fig08.pdf}}%
    \put(-0,0.24222037){\color[rgb]{0,0,0}\makebox(0,0)[lb]{\smash{\SBb{$Z$}}}}%
    \put(0,0){\includegraphics[width=\unitlength,page=3]{fig08.pdf}}%
    \put(0.07210679,0.27727399){\color[rgb]{0,0,0}\makebox(0,0)[lb]{\smash{\SBb{$\tilde{\pi}_{\circ}^{-1}(h)$}}}}%
    \put(0.18453577,0.27822442){\color[rgb]{0,0,0}\makebox(0,0)[lb]{\smash{\SBb{$\tilde{\pi}_{\circ}^{-1}(\bar{h})$}}}}%
    \put(0,0){\includegraphics[width=\unitlength,page=4]{fig08.pdf}}%
    \put(0.26125892,0.09557185){\color[rgb]{0,0,0}\makebox(0,0)[lb]{\smash{\SBb{$\tilde{\pi}_{\circ}$}}}}%
    \put(0,0){\includegraphics[width=\unitlength,page=5]{fig08.pdf}}%
    \put(0.4565551,0.10874955){\color[rgb]{0,0,0}\makebox(0,0)[lb]{\smash{\SBb{$\HH$}}}}%
    \put(0,0){\includegraphics[width=\unitlength,page=6]{fig08.pdf}}%
    \put(0.48416124,0.06604742){\color[rgb]{0,0,0}\makebox(0,0)[lb]{\smash{\SBb{$h$}}}}%
    \put(0,0){\includegraphics[width=\unitlength,page=7]{fig08.pdf}}%
    \put(0.45450423,0.02486987){\color[rgb]{0,0,0}\makebox(0,0)[lb]{\smash{\SBb{$\bar{h}$}}}}%
    \put(0,0){\includegraphics[width=\unitlength,page=8]{fig08.pdf}}%
  \end{picture}%
\endgroup%

%% file: fig10.pdf_tex
\begingroup%
  \makeatletter%
  \providecommand\color[2][]{%
    \errmessage{(Inkscape) Color is used for the text in Inkscape, but the package 'color.sty' is not loaded}%
    \renewcommand\color[2][]{}%
  }%
  \providecommand\transparent[1]{%
    \errmessage{(Inkscape) Transparency is used (non-zero) for the text in Inkscape, but the package 'transparent.sty' is not loaded}%
    \renewcommand\transparent[1]{}%
  }%
  \providecommand\rotatebox[2]{#2}%
  \ifx\svgwidth\undefined%
    \setlength{\unitlength}{827.44308832bp}%
    \ifx\svgscale\undefined%
      \relax%
    \else%
      \setlength{\unitlength}{\unitlength * \real{\svgscale}}%
    \fi%
  \else%
    \setlength{\unitlength}{\svgwidth}%
  \fi%
  \global\let\svgwidth\undefined%
  \global\let\svgscale\undefined%
  \makeatother%
  \begin{picture}(1,0.42904493)%
    \put(0,0){\includegraphics[width=\unitlength,page=1]{fig10.pdf}}%
    \put(0.01413983,0.24002759){\color[rgb]{0,0,0}\makebox(0,0)[lb]{\smash{\SB{$\PP^2\setminus\PP^2(\R)$}}}}%
    \put(0,0){\includegraphics[width=\unitlength,page=2]{fig10.pdf}}%
    \put(0.08034265,0.19361997){\color[rgb]{0,0,0}\makebox(0,0)[lb]{\smash{\SB{$p_1$}}}}%
    \put(0.07441797,0.08287837){\color[rgb]{0,0,0}\makebox(0,0)[lb]{\smash{\SB{$\bar{p}_1$}}}}%
    \put(0.21068551,0.19493445){\color[rgb]{0,0,0}\makebox(0,0)[lb]{\smash{\SB{$p_2$}}}}%
    \put(0.2146634,0.08527331){\color[rgb]{0,0,0}\makebox(0,0)[lb]{\smash{\SB{$\bar{p}_2$}}}}%
    \put(0.59714361,0.39247146){\color[rgb]{0,0,0}\makebox(0,0)[lb]{\smash{\SB{$X\setminus X(\R)$}}}}%
    \put(0,0){\includegraphics[width=\unitlength,page=3]{fig10.pdf}}%
    \put(0.34102838,0.35442174){\color[rgb]{0,0,0}\makebox(0,0)[lb]{\smash{\SB{$q$}}}}%
    \put(0.39181131,0.3091013){\color[rgb]{0,0,0}\makebox(0,0)[lb]{\smash{\SB{$\bar{q}$}}}}%
    \put(0,0){\includegraphics[width=\unitlength,page=4]{fig10.pdf}}%
    \put(0.15736343,0.26006331){\color[rgb]{0,0,0}\makebox(0,0)[lb]{\smash{\SB{$p_3$}}}}%
    \put(0.27094792,0.14135289){\color[rgb]{0,0,0}\makebox(0,0)[lb]{\smash{\SB{$\bar{q}$}}}}%
    \put(0.55114588,0.15856981){\color[rgb]{0,0,0}\makebox(0,0)[lb]{\smash{\SB{$\pi_{\circ}$}}}}%
    \put(0.63893344,0.33907558){\color[rgb]{0,0,0}\makebox(0,0)[lb]{\smash{\SB{$\tilde{\pi}_{\circ}$}}}}%
    \put(0.83675125,0.20905232){\color[rgb]{0,0,0}\makebox(0,0)[lb]{\smash{\SB{$\PP^1\setminus\PP^1(\R)$}}}}%
    \put(0,0){\includegraphics[width=\unitlength,page=5]{fig10.pdf}}%
    \put(0.91448711,0.30247241){\color[rgb]{0,0,0}\makebox(0,0)[lb]{\smash{\SB{$\R$}}}}%
    \put(0,0){\includegraphics[width=\unitlength,page=6]{fig10.pdf}}%
    \put(0.88337541,0.3844996){\color[rgb]{0,0,0}\makebox(0,0)[lb]{\smash{\SB{$\pi_{\circ}(q)=[a+ib:1]$}}}}%
    \put(0,0){\includegraphics[width=\unitlength,page=7]{fig10.pdf}}%
    \put(0.88165979,0.25806524){\color[rgb]{0,0,0}\makebox(0,0)[lb]{\smash{\SB{$\pi_{\circ}(\bar{q})=[a-ib:1]$}}}}%
    \put(0,0){\includegraphics[width=\unitlength,page=8]{fig10.pdf}}%
    \put(0.13741165,0.00479227){\color[rgb]{0,0,0}\makebox(0,0)[lb]{\smash{\SBa{$C_{q}$}}}}%
    \put(0,0.14236123){\color[rgb]{0,0,0}\makebox(0,0)[lb]{\smash{\SBa{$C_{\bar{q}}$}}}}%
    \put(0,0){\includegraphics[width=\unitlength,page=9]{fig10.pdf}}%
    \put(0.34737312,0.21754853){\color[rgb]{0,0,0}\makebox(0,0)[lb]{\smash{\SBa{$\tilde{C}_{q}$}}}}%
    \put(0.37443951,0.41591931){\color[rgb]{0,0,0}\makebox(0,0)[lb]{\smash{\SBa{$\tilde{C}_{\bar{q}}$}}}}%
    \put(0,0){\includegraphics[width=\unitlength,page=10]{fig10.pdf}}%
    \put(0.91169447,0.33760209){\color[rgb]{0,0,0}\makebox(0,0)[lb]{\smash{\SB{$1-\nu(C_q)$}}}}%
    \put(0,0){\includegraphics[width=\unitlength,page=11]{fig10.pdf}}%
  \end{picture}%
\endgroup%

%% file: fig03.pdf_tex
\begingroup%
  \makeatletter%
  \providecommand\color[2][]{%
    \errmessage{(Inkscape) Color is used for the text in Inkscape, but the package 'color.sty' is not loaded}%
    \renewcommand\color[2][]{}%
  }%
  \providecommand\transparent[1]{%
    \errmessage{(Inkscape) Transparency is used (non-zero) for the text in Inkscape, but the package 'transparent.sty' is not loaded}%
    \renewcommand\transparent[1]{}%
  }%
  \providecommand\rotatebox[2]{#2}%
  \ifx\svgwidth\undefined%
    \setlength{\unitlength}{1389.64688999bp}%
    \ifx\svgscale\undefined%
      \relax%
    \else%
      \setlength{\unitlength}{\unitlength * \real{\svgscale}}%
    \fi%
  \else%
    \setlength{\unitlength}{\svgwidth}%
  \fi%
  \global\let\svgwidth\undefined%
  \global\let\svgscale\undefined%
  \makeatother%
  \begin{picture}(1,0.48756915)%
    \put(0,0){\includegraphics[width=\unitlength,page=1]{fig03.pdf}}%
    \put(0.03380529,0.19377308){\color[rgb]{0,0,0}\makebox(0,0)[lb]{\smash{\SB{$r$}}}}%
    \put(0.0642344,0.0626606){\color[rgb]{0,0,0}\makebox(0,0)[lb]{\smash{\SB{$p_1$}}}}%
    \put(0.19746454,0.10088099){\color[rgb]{0,0,0}\makebox(0,0)[lb]{\smash{\SB{$\bar{p}_1$}}}}%
    \put(0.20568862,0.18246634){\color[rgb]{0,0,0}\makebox(0,0)[lb]{\smash{\SB{$p_2$}}}}%
    \put(0.11028926,0.21868019){\color[rgb]{0,0,0}\makebox(0,0)[lb]{\smash{\SB{$\bar{p}_2$}}}}%
    \put(0.21569153,0.29937955){\color[rgb]{0,0,0}\makebox(0,0)[lb]{\smash{\SB{$\pi$}}}}%
    \put(0.61442553,0.30087513){\color[rgb]{0,0,0}\makebox(0,0)[lb]{\smash{\SB{$\eta$}}}}%
    \put(0.52642786,0.13338612){\color[rgb]{0,0,0}\makebox(0,0)[lb]{\smash{\SB{$\eta(\tilde{C}_r)$}}}}%
    \put(0.58317406,0.05912086){\color[rgb]{0,0,0}\makebox(0,0)[lb]{\smash{\SB{$\bar{p}_2$}}}}%
    \put(0.2130903,-0.03345268){\color[rgb]{0,0,0}\makebox(0,0)[lb]{\smash{}}}%
    \put(0.70900252,0.09354651){\color[rgb]{0,0,0}\makebox(0,0)[lb]{\smash{\SB{$p_2$}}}}%
    \put(0.71640421,0.17019792){\color[rgb]{0,0,0}\makebox(0,0)[lb]{\smash{\SB{$\bar{p}_1$}}}}%
    \put(0.63416336,0.1719696){\color[rgb]{0,0,0}\makebox(0,0)[lb]{\smash{\SB{$p_1$}}}}%
    \put(0.59386536,0.20981347){\color[rgb]{0,0,0}\makebox(0,0)[lb]{\smash{\SB{$\eta_{\bullet}(q)$}}}}%
    \put(0.18530374,0.24220283){\color[rgb]{0,0,0}\makebox(0,0)[lb]{\smash{\SB{$q$}}}}%
    \put(0.25914517,0.25439021){\color[rgb]{0,0,0}\makebox(0,0)[lb]{\smash{}}}%
    \put(0,0){\includegraphics[width=\unitlength,page=2]{fig03.pdf}}%
    \put(0.37081198,0.4323633){\color[rgb]{0,0,0}\makebox(0,0)[lb]{\smash{\SB{$q$}}}}%
    \put(0.40614786,0.29869901){\color[rgb]{0,0,0}\makebox(0,0)[lb]{\smash{\SB{$\tilde{C}_r$}}}}%
    \put(0.25954545,0.34148645){\color[rgb]{0,0,0}\makebox(0,0)[lb]{\smash{\SB{$\tilde{L}_{r,p_1}$}}}}%
    \put(0.28100362,0.05298139){\color[rgb]{0,0,0}\makebox(0,0)[lb]{\smash{\SB{$L_{r,\bar{p}_1}$}}}}%
    \put(0.33578438,0.16309444){\color[rgb]{0,0,0}\makebox(0,0)[lb]{\smash{\SB{$L_{r,p_2}$}}}}%
    \put(0.25495164,0.26484957){\color[rgb]{0,0,0}\makebox(0,0)[lb]{\smash{\SB{$L_{r,\bar{p}_2}$}}}}%
    \put(0.55512649,0.39374048){\color[rgb]{0,0,0}\makebox(0,0)[lb]{\smash{\SB{$E_1$}}}}%
    \put(0.55913872,0.422521){\color[rgb]{0,0,0}\makebox(0,0)[lb]{\smash{\SB{$\bar{E}_1$}}}}%
    \put(0.56204642,0.45400602){\color[rgb]{0,0,0}\makebox(0,0)[lb]{\smash{\SB{$E_2$}}}}%
    \put(0.55803415,0.47628658){\color[rgb]{0,0,0}\makebox(0,0)[lb]{\smash{\SB{$\bar{E}_2$}}}}%
    \put(0.26193001,0.42087569){\color[rgb]{0,0,0}\makebox(0,0)[lb]{\smash{\SB{$\tilde{L}_{r,\bar{p}_2}$}}}}%
    \put(0.25960389,0.39338052){\color[rgb]{0,0,0}\makebox(0,0)[lb]{\smash{\SB{$\tilde{L}_{r,p_2}$}}}}%
    \put(0.25784395,0.37007116){\color[rgb]{0,0,0}\makebox(0,0)[lb]{\smash{\SB{$\tilde{L}_{r,\bar{p}_1}$}}}}%
    \put(0.16359292,0.0030357){\color[rgb]{0,0,0}\makebox(0,0)[lb]{\smash{\SB{$L_{r,p_1}$}}}}%
    \put(0.66952686,0.00305973){\color[rgb]{0,0,0}\makebox(0,0)[lb]{\smash{\SB{$\eta(E_1)$}}}}%
    \put(0.77972958,0.04201532){\color[rgb]{0,0,0}\makebox(0,0)[lb]{\smash{\SB{$\eta(\bar{E}_1)$}}}}%
    \put(0.84634462,0.15626354){\color[rgb]{0,0,0}\makebox(0,0)[lb]{\smash{\SB{$\eta(E_2)$}}}}%
    \put(0.77232789,0.25289237){\color[rgb]{0,0,0}\makebox(0,0)[lb]{\smash{\SB{$\eta(\bar{E}_2)$}}}}%
    \put(0.12673738,0.24624354){\color[rgb]{0,0,0}\makebox(0,0)[lb]{\smash{\SB{$C_q$}}}}%
    \put(0.39237529,0.46747137){\color[rgb]{0,0,0}\makebox(0,0)[lb]{\smash{\SB{$\tilde{C}_q$}}}}%
    \put(0.62182709,0.23472982){\color[rgb]{0,0,0}\makebox(0,0)[lb]{\smash{\SB{$f(C_q)=C_q$}}}}%
    \put(0,0){\includegraphics[width=\unitlength,page=3]{fig03.pdf}}%
    \put(0.403889,0.12606999){\color[rgb]{0,0,0}\makebox(0,0)[lb]{\smash{\SB{$f$}}}}%
    \put(0.51902616,0.08609166){\color[rgb]{0,0,0}\makebox(0,0)[lb]{\smash{\SB{$\eta(E_r)$}}}}%
    \put(0.30848967,0.46728547){\color[rgb]{0,0,0}\makebox(0,0)[lb]{\smash{\SB{$E_r$}}}}%
    \put(0.45487831,0.3119588){\color[rgb]{0,0,0}\makebox(0,0)[lb]{\smash{}}}%
  \end{picture}%
\endgroup%